\documentclass[11pt,twoside]{article}

\AtBeginDocument{%
  \fontsize{11pt}{13pt}\selectfont
}

\makeatother

\renewenvironment{abstract}
  {\small
   \begin{quote}
   \noindent\textbf{Abstract.}\ }
  {\end{quote}}

\usepackage{amsthm}

\theoremstyle{plain}
\newtheorem{lemma}{Lemma}[section]
\newtheorem{theorem}[lemma]{Theorem}
\newtheorem{proposition}[lemma]{Proposition}
\newtheorem{corollary}[lemma]{Corollary}

\theoremstyle{definition}

\theoremstyle{remark}
\newtheorem{remark}[lemma]{Remark}

\usepackage{graphicx}
\usepackage{caption}

\usepackage{titlesec}
\usepackage{newtxtext,newtxmath}

\titleformat{\section}[block]
  {\centering\large\bfseries\color{black}}
  {\thesection.}{0.6em}{}

\titleformat{\subsection}[hang]
  {\normalfont\large\bfseries\color{black}}
  {\thesubsection}{1em}{}

\titleformat{\subsubsection}[hang]
  {\normalfont\normalsize\bfseries\color{black}}
  {\thesubsubsection}{1em}{}

\usepackage[T1]{fontenc}
\usepackage{amsmath}
\usepackage[margin=0.87in]{geometry}
\usepackage{bm}
\usepackage{xcolor}

\definecolor{A}{RGB}{188,129,222}
\definecolor{C}{RGB}{0,139,139}
\definecolor{D}{RGB}{0,76,153}
\definecolor{B}{RGB}{25,25,112}
\definecolor{S}{RGB}{20,0,240}
\definecolor{G}{RGB}{3,105,215}
\definecolor{R1}{RGB}{139,58,58}
\definecolor{R2}{RGB}{208,32,144}
\definecolor{R3}{RGB}{70,176,246}
\definecolor{br}{RGB}{135,38,87}

\usepackage[
  breaklinks=true,
  colorlinks=true,
  linkcolor=D,
  urlcolor=blue,
  citecolor=D,
  bookmarks=true,
  bookmarksopen=true,
  bookmarksdepth=3
]{hyperref}
\usepackage{orcidlink}

\usepackage{fancyhdr}

\fancypagestyle{plain}{%
  \fancyhf{}

}

\usepackage{xcolor}

\usepackage[backend=biber]{biblatex}
\usepackage{float}
\usepackage{tocloft}
\hypersetup{
  linkcolor=D,
  linktoc=all
}

\title{\textbf{Stability of 3D Stratified Plane Couette Flow: A New Lift-up Mechanism and Potential Vorticity}}

\author{Yiting Yao\,\orcidlink{0009-0000-5232-7230}
\thanks{Department of Mathematics, University of Bath, Bath, UK.\\
E-mail: yy2644@bath.ac.uk
}}
\date{}

\makeatletter
\renewcommand{\@maketitle}{%
  \newpage
  \vspace*{3pt}
  \begin{center}
    {\LARGE\bfseries\@title\par}
    \vskip 1em
    {\large\bfseries\@author\par}
  \end{center}
  \vspace{-3pt}
}
\makeatother

\begin{document}

\maketitle
\vspace{-3pt}

\begin{abstract}
Since the pioneering work of Taylor and Goldstein in 1931 laid the foundation for the study of parallel stratified shear flows, non-parallel configurations have received comparatively little attention. We study the stability of the 3D Boussinesq system near the plane Couette flow $V^*=(y,0,0)$ under the vertically stratified background $\eta^*=\alpha z$. The shear and stratification directions are perpendicular. This geometry produces two structural effects that are absent in the parallel configuration: an additional nonlocal coupling in the non-zero-mode system, which obstructs a directly useful symmetrization, and a wave–shear coupled lift-up mechanism for the zero modes. The transient growth due to the lift-up effect can be suppressed either by sufficiently strong stratification  or under suitable assumptions on the initial data. To study the nonzero-mode system, we introduce a reformulation based on the linearized potential vorticity (PV), a fundamental quantity in geophysical fluid dynamics, to recover a better energy structure. We establish the inviscid damping of $u_{1,\neq}$ and $u_{2,\neq}$. We further apply the linearized PV formulation to study the linear stability of the tilted Couette flow. Finally, we use this formulation to prove the nonlinear stability for perturbations on $\mathbb{T}\times\mathbb{R}\times\mathbb{T}$ and show that strong stratification can enhance the stability.
\end{abstract}

\vspace{-1pt}
\tableofcontents
\vspace{-4pt}
\section{Introduction}
Consider the 3D Boussinesq system\begin{flalign}
\left\{
\begin{aligned}
& \partial_t V+(V\cdot\bm{\nabla})V=-\bm{\nabla}P+\nu\Delta V+\eta\vec{\bm{e}}_3,\\
   & \bm{\nabla}\cdot V=0,\\
&   \partial_t\eta+V\cdot\bm{\nabla}\eta=\kappa\Delta\eta.
\end{aligned}
\right.\label{eq:A1}
\end{flalign}Here $V\in\mathbb{R}^3$ is the velocity field, $\eta\in\mathbb{R}$ is the fluid temperature, and $P$ is the pressure. $\nu,\kappa \geq0$ are the viscosity and thermal diffusivity, respectively. The viscosity $\nu$ is inversely proportional to the Reynolds number $\bm{\text{Re}}$, a dimensionless parameter used to predict the behaviour of the fluid flow.

System (\ref{eq:A1}) is a standard model in geophysical fluid dynamics. The
term $\eta\vec{\bm{e}}_3$ represents the buoyancy force. The interaction
between stratification and fluid transport may have either a
stabilizing or a destabilizing effect, depending on the physical regime; see
\cite{Vallis_2017} for further background.\vspace{1pt}

In this article, we are concerned with the following background state\begin{equation}
V^*=(y,0,0)^T,\qquad \eta^*=\alpha z,\qquad \bm{\nabla}P^*=(0,0,\alpha z)^T,\qquad \alpha>0.\label{rotequ}\end{equation}

In the normalization used here, $\sqrt{\alpha}$ is the \textit{Brunt-V\"ais\"al\"a} frequency, measuring the strength of the stratification.  With our sign convention, $\alpha>0$ stands for stable stratification, i.e. colder fluid lies beneath warmer fluid. The flow $V^*$ is known as the \textit{plane Couette flow}, and its
shear direction is perpendicular to the vertical stratification. We are interested in the stability of this steady state, and we consider perturbations on either of the domains
\[\Omega_{\mathbb R}
   :=\mathbb T\times\mathbb R\times\mathbb R,
\qquad
\Omega_{\mathbb T}
   :=\mathbb T\times\mathbb R\times\mathbb T.\]
That is, the perturbations are periodic in \(x\), unbounded in \(y\), and either
unbounded or periodic in \(z\). When \(z\in\mathbb T\), the affine profile
\(\eta^*=\alpha z\) is understood as a prescribed background stratification
on the universal cover, while periodicity is imposed only on the perturbation.
The same convention applies to the corresponding background pressure. We write the perturbations as follows:\begin{align*}
V=V^*(y)+U(t,x,y,z),\qquad \eta=\eta^*(z)+\theta(t,x,y,z)\qquad
\bm{\nabla}P=\bm{\nabla}P^*(y,z)+\bm{\nabla}p(t,x,y,z).
\end{align*}Here $U=(u_1,u_2,u_3)^T,\theta$ and $p$ are defined on $\mathbb{R}^+\times\Omega_{\mathbb{D}}$ for $\mathbb{D}\in\left\{\mathbb{R},\mathbb{T}\right\}$. The perturbed system is formulated as\begin{flalign*}
\left\{
\begin{aligned}
& \partial_t U+(U\cdot\bm{\nabla})U+y\partial_x U+u_2\vec{\bm{e}}_1=-\bm{\nabla}p+\nu\Delta U+\theta\vec{\bm{e}}_3,\\
   & \bm{\nabla}\cdot U=0,\\
&   \partial_t\theta+U\cdot\bm{\nabla}\theta+y\partial_x\theta+\alpha u_3=\kappa\Delta\theta.
\end{aligned}
\right.
\end{flalign*}

The study of hydrodynamic stability dates back to the nineteenth century, both in theoretical and experimental settings, with foundational contributions by Helmholtz, Kelvin, Rayleigh, and Reynolds; see, for instance, \cite{kelvin1887stability, rayleigh1880stability, Drazin_2002, reynolds1883iii}. Plane Couette flow is one of the simplest examples, yet its nonlinear dynamics
are remarkably subtle. In the classical Navier Stokes setting, the linearized
flow has no exponentially growing modes for any Reynolds number. Nevertheless,
experiments and numerical simulations show that, at sufficiently high Reynolds
number, finite-amplitude perturbations may trigger a transition from laminar to
turbulent motion.  This apparent discrepancy between linear stability and
nonlinear transition is commonly referred to as the \textit{Sommerfeld paradox}. To understand this transition mechanism, Trefethen et al. \cite{doi:10.1126/science.261.5121.578} emphasized the role of non-normal transient growth and proposed the stability transition threshold problem. This problem was later given a rigorous mathematical formulation and resolution by Bedrossian, Germain, and Masmoudi \cite{bedrossian2015dynamicsnearsubcriticaltransition, bedrossian2015dynamicsnearsubcriticaltransition1}. The problem can be stated as follows:

 For a specific norm $||\cdot||_{\mathcal{Y}}$, find an exponent $S$, known as the \textit{transition threshold}, such that\begin{align*}
&||u_0||_{\mathcal{Y}}\lesssim \nu^{S}\qquad\Rightarrow\qquad\text{stability},\\
&||u_0||_{\mathcal{Y}}\gg \nu^{S}\qquad\Rightarrow\qquad\text{possible instability}.
 \end{align*}

During the past decade, substantial progress has been made in understanding the stability thresholds of Couette flow and related shear flows. In three dimensions, Bedrossian, Germain, and Masmoudi carried out pioneering work on the classical Navier Stokes system and proved nonlinear stability of Plane Couette flow for perturbations of size $\mathcal{O}\left(\nu^{3/2}\right)$ in Sobolev space $H^{9/2+}$ and described the associated subcritical dynamics \cite{bedrossian2015dynamicsnearsubcriticaltransition, bedrossian2015dynamicsnearsubcriticaltransition1,9e0bb666-1e4e-3aa2-87b0-f916c2336491}. Wei and Zhang subsequently obtained a larger stability threshold $\mathcal{O}\left(\nu\right)$ in $H^2$ \cite{wei2021transition}. Related stability-threshold results have also been developed for several coupled fluid systems, including magnetohydrodynamics, the Boussinesq equations, and the compressible Navier--Stokes equations. In particular, a number of recent works have investigated the influence of rotation and stratification on Couette flow stability. In \cite{coti2025suppression}, Coti Zelati and Del Zotto pointed out the suppression of lift-up effect by stable stratification for vertical Couette flow $(z,0,0)$. The nonlinear stability was also proved by Coti Zelati, Del Zotto and Widmayer by combining mixing and dispersion \cite{coti2024stability}. In particular, this work appears to be the first to obtain a threshold greater than $\nu$ (i.e., $\nu^\sigma$ for some $\sigma<1$) in a 3D fluid system and illustrates the stabilization effect of stable stratification. Besides, Cui, Wang and Wang studied the 3D Boussinesq system without background stratification \cite{cui2025stabilitythresholdcouetteflow}. In this setting, the threshold is much smaller due to the lift-up effects. In the presence of rotation, both instability and stability can happen. In the unstable regime, Fan, Han and Wang rigorously established the nonlinear instability when $f\in\left(\frac{2}{17}(5-2\sqrt{2}),\frac{2}{17}(5+2\sqrt{2})\right)$ \cite{fan2025inertialinstabilitycouetteflow}. In the stable regime, the nonlinear stability was investigated in \cite{huang2024stabilitycouetteflow3d,li2025transitionthresholdnavierstokescoriolishigh,huang2024sobolevstabilitythreshold3d,zelati2025stabilityviscousthreedimensionalrotating}. We mention the work by Li, Sun, Wang, Wei and Zhang \cite{li2025transitionthresholdnavierstokescoriolishigh}, who proved the stability threshold $\mathcal{O}(\nu^{2/3+})$ on $\mathbb{T}\times\mathbb{R}^2$ and $\mathcal{O}(\nu^{5/6})$ on $\mathbb{T}\times\mathbb{R}\times\mathbb{T}$ by exploiting the anisotropic effect together with a new dispersive structure. In the presence of both rotation and stable stratification, recent work by \cite{huang2026stabilitythreshold3dboussinesq} Huang, Xu, Luo and Sun obtained the threshold $\mathcal{O}(\nu^{14/15})$. In two dimensions, vertical Couette flow in the Boussinesq system has also been studied extensively \cite{123,REN2026104421,REN2026111383,deng2021stability}. These developments motivate the study of stratified shear flows beyond the vertical, parallel setting.

\subsection{\textbf{Non-parallel shear flow in the stratified fluid}}

The classical linear theory of parallel stratified shear flows $U^*=U(z)\vec{\bm{e}}_1$ begins with the work of
Taylor and Goldstein in 1931 \cite{40ebecca-35da-3042-86b1-955a8e0fef33,10.1098/rspa.1931.0116}.  Their equation, now known as the Taylor–Goldstein equation, extends Rayleigh’s stability equation by incorporating buoyancy. Subsequent work by Miles, Howard, Booker, and Bretherton developed the spectral theory of parallel stratified shear flows \cite{Miles_1961,Howard_1961,Booker_Bretherton_1967}. In particular, the Miles-Howard stability criterion implies spectral stability of a two-dimensional parallel shear flow when the
Richardson number\footnote{The Richardson number is defined as $\text{Ri}:=-\frac{g}{\rho_0}\frac{d\rho^*}{dz}\left(\frac{dU}{dz}\right)^{-2}$ for parallel stratified shear flow $U=U(z)$.} $\text{Ri}>\frac{1}{4}$. There have also been studies on nonlinear stability criteria in terms of the Richardson number. For instance, in \cite{PhysRevLett.52.2352}, it was shown that $\text{Ri}>1$ could lead to nonlinear stability.

By contrast, non-parallel stratified shear flows are less well understood. It has been shown very recently that, following the Miles-Howard method, no stability criterion can be obtained for non-parallel shear flows \cite{teixeira2025integralconstraintslinearinstability}. Moreover, unlike parallel shear flows, it is possible that the instability of non-parallel shear flow cannot be suppressed, regardless of how strong the stratification is \cite{Candelier_Le}. More interestingly, for plane Couette flow set on domain $\mathbb{T}\times[-1,1]\times\mathbb{T}$, \cite{FacchiniEtAl2018} suggested that spectral instability can occur due to the presence of stable stratification.

From a mathematical perspective, significant progress has been made in understanding the stability of parallel shear flows in both two and three dimensions; see, for example,
\cite{coti2024stability,cui2025stabilitythresholdcouetteflow,nonlinearinvisciddamping,chen2026quantitativestability2dmonotone,10.1063/5.0091052}.
To the best of our knowledge, however, no rigorous result is currently available on the quantitative stability of three-dimensional non-parallel shear flows in the Boussinesq system (in the absence of rotation). This gap motivates us to investigate the stability of the plane Couette flow considered here.

\subsection{\textbf{Main results}}

Our results consist of three parts. First, we establish the linear stability of the plane stratified Couette flow. Second, we apply the method developed in the first part to a generalized configuration: the tilted Couette flow. Third, we establish the nonlinear stability of the plane stratified Couette flow.

\subsubsection{\textbf{Linear stability of the plane Couette flow}}\label{norot}

We start by analyzing the dynamics with $\nu=\kappa$, the linearized system reads
\begin{flalign}
\left\{
\begin{aligned}
&\partial_t U+u_2\vec{\bm{e}}_1+y\partial_x U-\nu\Delta U-\theta\vec{\bm{e}}_3=\bm{\nabla}\Delta^{-1}\left[2\partial_x u_2-\partial_z\theta\right],\\
    &\bm{\nabla}\cdot U=0,\\
    &\partial_t\theta+y\partial_x\theta+\alpha u_3=\nu\Delta\theta,
\end{aligned}
\right.\label{eq:A2}
\end{flalign}To describe the dynamics, we introduce the following mode decompositions: for a function $G=G(x,y,z)$, define its \textit{zero} and \textit{non-zero} modes by\[
G_0(y,z):=\frac{1}{2\pi}\int_{\mathbb{T}}G(x,y,z)\ dx,\qquad G_{\neq}(x,y,z):=G(x,y,z)-G_0(y,z),
\]When $G$ is defined on $\Omega_{\mathbb{T}}$, we further decompose\[
G_0(y,z)=\frac{1}{2\pi}\int_0^{2\pi}G(x,y,z)\ dx=\overline{G}(y,z)+\widetilde{G}(y),\qquad \widetilde{G}(y)=\frac{1}{4\pi^2}\int_{\mathbb{T}^2}G(x,y,z)\ dxdz.
\]We refer to $\overline{G}$ and $\widetilde{G}$ as the \textit{simple zero mode} and \textit{double zero mode}, respectively. For $G$ defined on $\Omega_{\mathbb{R}}$ or $\Omega_{\mathbb{T}}$, the corresponding Fourier transform $\widehat{G}$ is defined as\begin{align}
&\widehat{G}(k,\eta,l)=\int_{\mathbb{T}\times\mathbb{R}^2} G(x,y,z) e^{-ikx-i\eta y-ilz}\ dxdydz,\label{ourier1}\\&\widehat{G}(k,\eta,l)=\int_{\mathbb{T}\times\mathbb{R}\times\mathbb{T}} G(x,y,z) e^{-ikx-i\eta y-ilz}\ dxdydz,\label{ourier2}
\end{align}where $(k,\eta,l)\in (\mathbb{Z}\times\mathbb{R}\times\mathbb{R})$ and $(\mathbb{Z}\times\mathbb{R}\times\mathbb{Z})$ respectively.\vspace{6pt}

We next highlight two features of the system that are essential for establishing the main results.\vspace{8pt}

{ \normalsize\noindent\textbf{Additional coupling in the non-zero mode system.}} \vspace{4pt}

A standard approach to proving decay for the non-zero modes is to construct an energy functional based on suitably symmetrized variables, see, for instance, \cite{dolce2024stability,coti2025suppression,bianchini2022linear}. For the stably stratified Boussinesq equations linearized around the vertical Couette flow \cite{coti2025suppression} set on  $\mathbb{T}^2\times\mathbb{R}$, \[
V^*=(z,0,0),\qquad \eta^*=\alpha z,\qquad \alpha>0,
\]the pair $(q_3:=\Delta u_3,\theta)$ already forms a closed subsystem,\[
(\partial_t+z\partial_x)\begin{pmatrix}
    q_3\\\theta
\end{pmatrix}+\begin{pmatrix}
    -\nu&-\partial^2_x-\partial^2_y\\
    \alpha\Delta^{-1}&-\kappa
\end{pmatrix}\begin{pmatrix}
    q_3\\\theta
\end{pmatrix}=0.
\]This system is symmetrizable, and the decay estimates can be established by constructing an energy functional with a cross term. 

For the plane Couette flow $(y,0,0)$, however, the non-parallel alignment introduces an additional nonlocal coupling through the pressure, see (\ref{eq:A2}). Let us consider the slightly more general case, where rotation is also involved in the system. The (unperturbed) momentum equation in the inviscid case is\[
\partial_t V+(V\cdot\bm{\nabla})V+\widetilde{\gamma}\vec{\bm{e}}_3\times V+\bm{\nabla}p=\theta\vec{\bm{e}}_3.
\]One can derive a closed subsystem for $(u_3,\omega_3,\theta)$ when linearizing around the plane Couette flow:\begin{equation}
(\partial_t+y\partial_x)\begin{pmatrix}
    u_3\\\omega_3\\\theta
\end{pmatrix}+\begin{pmatrix}2\partial_{xy}\partial^2_z\Delta^{-1}\Delta_h^{-1}&\widetilde{\gamma}\partial_z\Delta^{-1}&-\Delta_h\Delta^{-1}\\
(1-\widetilde{\gamma})\partial_z&0&0\\
\alpha&0&0
\end{pmatrix}\begin{pmatrix}
    u_3\\\omega_3\\\theta
\end{pmatrix}-\begin{pmatrix}
    2\partial_{xxz}\Delta^{-1}\Delta^{-1}_h\omega_3\\0\\0
\end{pmatrix}=0.\label{coupledsystem}
\end{equation}When $\widetilde{\gamma}(\widetilde{\gamma}-1)>0$, the system can be symmetrized by certain change of variables and the term $2\partial_{xxz}\Delta^{-1}\Delta^{-1}_h\omega_3$ can be treated as an error term due to the fast decay property of the coefficient in the phase space. However, when $\widetilde{\gamma}=0$, $u_3$ is coupled with $\omega_3$ only through the error term $2\partial_{xxz}\Delta^{-1}\Delta^{-1}_h\omega_3$. The system is still symmetrizable but the symmetrization has the wrong scaling (see Section \ref{maindiff}). We therefore say system \eqref{coupledsystem} has a "bad" symmetrization structure when $\widetilde{\gamma}=0$. To address this issue, we introduce the following quantity\begin{equation}
K=\alpha\omega_3-\partial_z\theta,\label{LPV}
\end{equation}which satisfies the simple transport-diffusion equation\begin{equation}
\partial_tK+y\partial_xK=\nu\Delta K.\label{tde}
\end{equation}By the definition of $\omega_3$ and the divergence-free condition, we obtain a div--curl system for $(u_1,u_2)$,\begin{flalign}
\left\{
\begin{aligned}
&\partial_x u_2-\partial_y u_1=\frac{K+\partial_z\theta}{\alpha},\\
    &\partial_x u_1+\partial_y u_2=-\partial_z u_3.
\end{aligned}
\right.\label{divcurl1}
\end{flalign}This allows $u_1,u_2$ to be expressed in terms of $u_3,\theta$ and $K$, so that a closed system for $u_3,\theta$ can be derived (\ref{eq:B4}), to which a valid symmetrization can be made.\vspace{9pt}

{ \normalsize\noindent\textbf{A new lift-up mechanism.} }\vspace{4pt}

We next consider the simple zero-mode dynamics on $\Omega_{\mathbb{T}}$; the case of $\Omega_{\mathbb{R}}$ is similar. In the classical Navier Stokes setting (see, e.g. \cite{9e0bb666-1e4e-3aa2-87b0-f916c2336491}), the simple zero mode of $u_1$ is given by\[
\overline{u}_1=e^{\nu t\Delta}\left(\overline{u}_1(0)-t\overline{u}_2(0)\right).
\]Thus, the classical lift-up mechanism is generated by\[\mathcal{L}_{\text{NSE}}:= te^{\nu t\Delta}.\]Define the following $L^p$ space for functions with zero $z$-average: \[L^p_0:=\left\{f\in L^p(\mathbb{R}\times\mathbb{T}):\int_{\mathbb{T}} f(y,z)\ dz=0\right\},\]then we have \[\left\|\mathcal{L}_{\text{NSE}}\right\|_{L^2_0\rightarrow L^2_0}\lesssim\nu^{-1}e^{-c\nu t},\qquad \left\|\mathcal{L}_{\text{NSE}}\right\|_{L^2_0\rightarrow L^\infty_0}\lesssim\nu^{-1}e^{-c\nu t}.\]Consequently, $||\overline{u}_1||_{L^p}$ undergoes transient growth of size $\nu^{-1}$ for any $p\in[2,\infty]$ when $\nu\to0$. 

In the presence of vertical stratification, $\overline{u}_1$ is given explicitly by (\ref{eq:B9}). Unlike the classical NSE setting, the lift-up mechanism takes a different form, and it is caused by the following operator \begin{equation}
\mathcal{L}_{\text{Bou}}:=\alpha^{-\frac{1}{2}}\partial_z\partial_y^{-1}e^{\nu t\Delta}\left(e^{\pm\sqrt{\alpha}t\partial_y|\bm{\nabla}_{y,z}|^{-1}}-I_d\right).\label{Bouliftupope}
\end{equation}The anisotropic factor $\partial_z\partial_y^{-1}$ can produce transient growth. This mechanism is structurally different from $\mathcal{L}_{\text{NSE}}$ and can be suppressed under suitable assumptions on the initial data or by strong stratification. Moreover, $\mathcal{L}_{\text{Bou}}$ satisfies the following bounds on its operator norms:\[\left\|\mathcal{L}_{\text{Bou}}\right\|_{L^2_0\rightarrow L^2_0}\lesssim\nu^{-1}e^{-c\nu t},\qquad \left\|\mathcal{L}_{\text{Bou}}\right\|_{L^p_0\rightarrow L^q_0}\lesssim\nu^{-\frac{1}{2}}\alpha^{-\frac{1}{4}}e^{-c\nu t}.\]where\[
1\leq p\leq 2,\qquad \frac{1}{q}=\frac{1}{p}-\frac{1}{2}.
\]Thus, the transient growth is weaker than that generated by $\mathcal L_{\mathrm{NSE}}$. This type of lift-up can occur as long as the shear flow is not completely parallel to the stratification in the presence of stable stratification. We also remark that this type of lift-up mechanism can also occur in other wave-shear coupled systems such as the MHD system when an inverse derivative is combined with an oscillatory propagator (that is, transient growth due to low frequency in the Fourier variables). We will discuss these issues in Section \ref{weaklu}.\vspace{3pt}

\begin{theorem} (\textbf{Linear stability of the plane Couette flow.})\label{linear, Boussinesq} Assume that $\alpha>0$ and $\nu\in(0,1)$. Then, for $u,\theta$ satisfying (\ref{eq:A2}), we have:
\begin{itemize}
    \item \textbf{Linear inviscid damping and enhanced dissipation for the non-zero modes.} On $\Omega_{\mathbb{T}}$ or $\Omega_{\mathbb{R}}$, the following estimates hold.\begin{align}
||u_{1,\neq}||_{L^2}+||u_{2,\neq}||_{L^2}&\lesssim C_\alpha\frac{e^{-\frac{\nu t^3}{277}}}{\langle t\rangle}\left(\left\|\frac{\theta_{\neq}(0)}{\alpha^\frac{1}{2}}\right\|_{H^{\sigma+2}}+||u_{\neq}(0)||_{H^{\sigma+2}}\right),\label{eq:A3}\\
 \left\|\frac{\theta_{\neq}}{\alpha^\frac{1}{2}}\right\|_{L^2}+||u_{3,\neq}||_{L^2}&\lesssim C_\alpha e^{-\frac{\nu t^3}{277}}\left(\left\|\frac{\theta_{\neq}(0)}{\alpha^\frac{1}{2}}\right\|_{H^{\sigma}}+||u_{\neq}(0)||_{H^{\sigma}}\right),\label{eq:A4}
\end{align}where $C_\alpha=e^{c_1\alpha^{-1}}\left(1+\alpha^{-1}\right)$ for some $c_1>0$ and $\sigma\approx2.315$\footnote{Here $\sigma:=\frac{2(1-\lambda_m)}{\sqrt{1-\left(\frac{1+2\lambda_m}{2\sqrt{2}}\right)^2}}+\lambda_m+\frac{1}{2}$, where $\lambda_m$ is determined by the algebraic equation \eqref{minreg}, with $\lambda_m\approx 0.7845$. The constant $\frac{1}{277}$ here is also consequently determined for $\lambda=\lambda_m$.}.
\item \textbf{Conditional and weak lift-up of $u_{1,0}$ on $\Omega_{\mathbb{R}}$}. If $\int_{\mathbb{R}}(u_2,\theta)_{0}(0)\ dy=0,$  then $u_{1,0}$ does not experience transient growth,\begin{equation}
||u_{1,0}||_{L^\infty}\lesssim||u_{1,0}(0)||_{H^2}+\alpha^{-\frac{1}{2}}\left\|\left(u_2,\frac{\theta}{\alpha^\frac{1}{2}}\right)_0\right\|_{H^{2+}}+\alpha^{-\frac{1}{2}}\left\|\langle y\rangle^{\frac{3}{2}+}\langle\partial_z\rangle^{\frac{1}{2}+}\left(u_2,\frac{\theta}{\alpha^\frac{1}{2}}\right)_0\right\|_{L^2}.\label{9898}
\end{equation}
Otherwise, $u_{1,0}$ experiences transient growth,\[
||u_{1,0}||_{L^\infty}\lesssim\left\|(u_1,u_2)_0(0)\right\|_{L^\infty}+\nu^{-\frac{1}{2}}\alpha^{-\frac{1}{4}}\left\|\left(u,\frac{\theta}{\alpha^\frac{1}{2}}\right)_0(0)\right\|_{L^2}.
    \]
\item\textbf{Conditional and weak lift-up of $\overline{u}_1$ on $\Omega_{\mathbb{T}}$}. If $\int_{\mathbb{R}}(u_2,\theta)_{0}(0)\ dy=0$, then $\overline{u}_1$ does not experience lift-up,\begin{align}
||\overline{u}_{1}||_{L^2}&\lesssim e^{-\nu t}\left(||\overline{u}_{1}(0)||_{L^2}+\alpha^{-\frac{1}{2}}||\overline{u}_{2}(0)||_{L^2}+\alpha^{-\frac{1}{2}}\left\|\langle y\rangle^{\frac{3}{2}+}\partial_z\overline{\left(u_{2},\frac{\theta}{\alpha^\frac{1}{2}}\right)}(0)\right\|_{L^2}\right).\label{cancellation}
\end{align}Otherwise, $\overline{u}_1$ experiences transient growth, \begin{align}
||\overline{u}_1||_{L^2}&\lesssim e^{-\frac{\nu t}{2}}\left(||\overline{u}_{1}(0)||_{L^2}+\alpha^{-\frac{1}{2}}||\overline{u}_{2}(0)||_{L^2}+\nu^{-1}\left\|\overline{\left(u_2,\frac{\theta}{\alpha^\frac{1}{2}}\right)}(0)\right\|_{L^2}\right),\nonumber\\
||\overline{u}_1||_{L^\infty}&\lesssim  e^{-\frac{\nu t}{2}}\left(\left\|\overline{u}_1(0)\right\|_{H^2}+\alpha^{-\frac{1}{2}}\left\|\overline{u}_2(0)\right\|_{H^2}+\nu^{-\frac{1}{2}}\alpha^{-\frac{1}{4}}\left\|\overline{\left(u_2,\frac{\theta}{\alpha^\frac{1}{2}}\right)}(0)\right\|_{L^2}\right)\label{weakliftup2347}
\end{align}
\item \textbf{Double zero modes on $\Omega_{\mathbb{T}}$}. If     $\widetilde{u}(0),\widetilde{\theta}(0)\in L^2_y$, the explicit solution is given by\[
\left\{
\begin{aligned}
&\widetilde{u}_{1}=e^{\nu t\partial^2_y}\widetilde{u}_1(0,y),\qquad\widetilde{u}_{2}=0,\\
    &\widetilde{u}_{3}=e^{\nu t\partial^2_y}\left(\widetilde{u}_3(0,y)\cos(\sqrt{\alpha} t)+\frac{\widetilde{\theta}(0,y)}{\sqrt{\alpha}}\sin(\sqrt{\alpha}t)\right),\\
    &\widetilde{\theta}=e^{\nu t\partial^2_y}\left(-\sqrt{\alpha}\widetilde{u}_3(0,y)\sin(\sqrt{\alpha} t)+\widetilde{\theta}(0,y)\cos(\sqrt{\alpha}t)\right).
\end{aligned}
\right.
\]

\item \textbf{Dispersive estimate of the zero modes}.  On $\Omega_{\mathbb{R}}$, the following dispersive estimate holds\begin{align}
\left\|\left(u_2,u_3,\frac{\theta}{\alpha^\frac{1}{2}}\right)_0\right\|_{L^\infty}\lesssim \langle\alpha^\frac{1}{2}t\rangle^{-\frac{1}{2}}\left\|\left(u_2,u_3,\frac{\theta}{\alpha^\frac{1}{2}}\right)_0(0)\right\|_{W^{3,1}}\label{eq:D1}.
\end{align}On $\Omega_{\mathbb{T}}$, we have\begin{align*}
\left\|\left(\overline{u}_2,\overline{u}_3,\frac{\overline{\theta}}{\alpha^\frac{1}{2}}\right)\right\|_{L^\infty}\lesssim \langle\alpha^\frac{1}{2}t\rangle^{-\frac{1}{3}}e^{-\nu t}\left\|\left(\overline{u}_2,\overline{u}_3,\frac{\overline{\theta}}{\alpha^\frac{1}{2}}\right)(0)\right\|_{W^{3,1}}.
\end{align*}
\end{itemize}
\end{theorem}

\begin{remark} Apart from the Navier Stokes system with a specific rotation rate \cite{huang2024stabilitycouetteflow3d}, this appears to be the first result in three dimensions where both horizontal components of the velocity field exhibit damping. This behaviour is also in sharp contrast with that of the vertical Couette flow $(z,0,0)$, for which only $u_{3,\neq}$ undergoes inviscid damping and $\theta_{\neq}$ is mixing. To illustrate the geometric dependence, we consider Couette flow forming an arbitrary angle with the stratification direction in Section \ref{ang} and examine how the decay behaviour changes with the angle.\end{remark}\vspace{6pt}

{\large\noindent\textbf{"Bad" symmetrization structure.} }\vspace{2pt}

As a remark, we discuss the consequence of having "bad" symmetrization structure (e.g. no valid symmetrizer exists). One possible consequence is the loss of regularity for boundedness/decay. Let us consider the following system set on $\mathbb{R}^+\times\mathbb{R}^2$.\begin{align*}
    \partial_t\begin{pmatrix}
        a\\ b
    \end{pmatrix}+y\partial_x\begin{pmatrix}
        a\\ b
    \end{pmatrix}+M_i\begin{pmatrix}
        a\\ b
    \end{pmatrix}=0.
\end{align*}We set\[
M_1=\begin{pmatrix}
   0&\partial^5_x\Delta^{-1}\\
    \partial^5_x\Delta^{-1}&0
\end{pmatrix},\qquad M_2=\begin{pmatrix}
    0&\partial^5_x\Delta^{-1}\\
  0&0
\end{pmatrix}.
\]For the system associated with $M_1$, standard energy estimate gives\[
||a||^2_{L^2}+||b||^2_{L^2}=||a(0)||^2_{L^2}+||b(0)||^2_{L^2}.
\]For $M_2$, the system is non-symmetrizable. The explicit solution in Fourier variables reads\begin{align*}
   & \widehat{a}(t,k,\eta)=\widehat{a}_{\text{in}}(k,\eta+kt)-ik^3\widehat{b}_{\text{in}}(k,\eta+kt)\left(\tan^{-1}\left(\frac{\eta}{k}\right)-\tan^{-1}\left(\frac{\eta+kt}{k}\right)\right),\\
  &  \widehat{b}(t,k,\eta)=\widehat{b}_{\text{in}}(k,\eta+kt).
\end{align*}By Plancherel’s theorem, we see that the boundedness would require more derivatives than the symmetric situation. In particular, we have\[
||a||_{L^2}^2+||b||_{L^2}^2\lesssim||a_{\text{in}}||_{L^2}^2+||b_{\text{in}}||_{H^3}^2.
\]This suggests a possible mechanism for regularity loss in our setting. In our recent work \cite{yao2026pvwavedecomposition}, we established linear inviscid damping of $u_{1,\neq}$ and $u_{2,\neq}$ in the presence of rotation for initial data in $H^4$. This is a lower regularity requirement than the $H^{\approx4.315}$ assumption in Theorem \ref{linear, Boussinesq}.\vspace{8pt}

{\large\noindent\textbf{General case where $\nu\neq\kappa$.} }\vspace{2pt}

Having understood the simpler case $\nu=\kappa$, we now extend the preceding results to the case $\nu\neq\kappa$. The aim of this section is to establish decay estimates for the non-zero modes. 

No obvious nontrivial conservation law appears to persist. Nevertheless, by treating $K$ as a new unknown, the reformulated $(u_3,\theta)$ system still retains a favorable energy structure. This allows us to carry out the analysis whenever $\nu$ and $\kappa$ are comparable. The relative size of $\nu$ and $\kappa$ is naturally characterized by the Prandtl number $\text{Pr}:=\frac{\nu}{\kappa}$\footnote{In many important physical settings (e.g. in water and air), the Prandtl number $\text{Pr}\sim\mathcal{O}(1)$. }, a fundamental dimensionless parameter in geophysical flows \cite{ozsoy2021geophysical}. We therefore state the theorem in terms of $\text{Pr}$.
\begin{theorem}\label{Pran}For any $\nu,\kappa>0$ such that $\text{Pr}\in(15-4\sqrt{14},15+4\sqrt{14})\backslash\{1\}$, and an arbitrary constant $c_0\in\left(0,\frac{7}{9}\right)$, if $\alpha,\text{Pr}$ and $c_0$ satisfy the following relation\[
\Omega^2\leq\frac{32\left(\text{Pr}-c_0\min\left\{1,\text{Pr}\right\}\right)\left(1-c_0\min\left\{1,\text{Pr}\right\}\right)-\left(\text{Pr}+1\right)^2}{8(\text{Pr}-1)^2}\cdot\alpha
\]for some constant $\Omega>0$, then there exists a positive number $N=N_1+N_2\frac{\sqrt{\alpha}}{\Omega}$, where $N_1,N_2>0$ are universal constants\footnote{Independent of any parameter such as $\alpha,c_0$ and Pr.}, such that the following decay estimates hold for the non-zero modes.
\begin{align*}
    \left\|(u_1,u_2)_{\neq}\right\|_{L^2}&\lesssim\left(1+\frac{1}{\Omega}\right)\frac{e^{-\frac{c_0\mu t^3}{30}}}{\langle t\rangle}\mathcal{J}_{N+\frac{5}{2}},\qquad    \left\|u_{3,\neq}\right\|_{L^2}\lesssim e^{-\frac{c_0\mu t^3}{30}} \mathcal{J}_{N},\\
     \left\|\frac{\theta_{\neq}}{\alpha^\frac{1}{2}}\right\|_{L^2}&\lesssim e^{-\frac{c_0\mu t^3}{30}} \mathcal{J}_{N+1},\qquad 
     \left\|\frac{K_{\neq}}{\alpha}\right\|_{L^2}\lesssim e^{-\frac{c_0\mu t^3}{30}} \mathcal{J}_{N+3},
\end{align*}where $\mu=\min\{\nu,\kappa\}$ and $c>0$ is a universal constant such that\[
\mathcal{J}_N=e^{c\alpha^{-1}}\left(\left\|\frac{\theta_{\neq}(0)}{\alpha^\frac{1}{2}}\right\|_{H^N}+||u_{3,\neq}(0)||_{H^N}+\Omega\left\|\frac{K_{\neq}(0)}{\alpha}\right\|_{H^N}\right).
\]
\end{theorem} 

\begin{remark}Unlike the estimates (\ref{eq:A4}), the regularity required here increases with $\alpha$. Since stronger stratification is stabilizing, this loss is likely an artefact of the method rather than an intrinsic feature of the dynamics.\end{remark}

Theorem \ref{Pran} involves two free parameters, $c_0$ and $\Omega$. A smaller $c_0$ allows a wider range of admissible Prandtl numbers, at the expense of a weaker enhanced dissipation rate. The parameter $\Omega$ should be chosen at an intermediate scale, since both the estimates and the required regularity depend on both $\Omega$ and $\Omega^{-1}$. Moreover, by the homogeneity relation $K\sim \alpha\omega_3$ for large $\alpha$, $\Omega$ should be chosen independently of $\alpha$. For ease of understanding, one may simply take $\Omega$ to be an $\mathcal{O}(1)$ constant, for example $\Omega=1/5$.\vspace{8pt}

\subsubsection{\textbf{Potential vorticity formulation}}\label{sectionpv}
In the previous section, we identified the conserved quantity $K$ (\ref{LPV}) as the key ingredient in the energy estimate. The aim of this section is to relate this conserved quantity to the underlying physical mechanism.

In many geophysical fluid models—such as the Shallow water system, Quasi-geostrophic system and the Boussinesq system—the dynamics exhibit a remarkable feature that is absent in the classical Euler or Navier Stokes equations: there exist quantities that are materially transported by the flow and therefore conserved along fluid particle trajectories. 

The most fundamental of these quantities is potential vorticity (PV), originally introduced by Rossby and later generalized by Ertel. Physically, PV combines the rotational motion of a fluid parcel with the ambient stratification. It measures how the vorticity of a fluid element is constrained by the gradient of buoyancy or temperature. When a stratified fluid column is stretched, compressed, or tilted, its vorticity and stratification generally change, but these changes compensate in such a way that the potential vorticity remains (materially) conserved in the inviscid and adiabatic regime. This conservation law makes PV a natural variable for describing the slowly evolving, vortical component of geophysical flows and for distinguishing it from propagating wave motions. For further physical background and intuition, we refer to \cite{Vallis_2017}. 

Let us consider the 3D inviscid Boussinesq system,
\begin{flalign*}
\left\{
\begin{aligned}
& \partial_t V+(V\cdot\bm{\nabla})V=-\bm{\nabla}p+\theta\vec{\bm{e}}_3\\
&   \partial_t\theta+V\cdot\bm{\nabla}\theta=0,\qquad \bm{\nabla}\cdot V=0.
\end{aligned}
\right.
\end{flalign*}

This system admits the following conservation law\begin{equation}
q:=\bm{\nabla}\theta\cdot\bm{\omega},\qquad\Rightarrow\qquad \partial_tq+V\cdot\bm{\nabla}q=0.\label{PV}
\end{equation}

 Interestingly, the conserved quantity $K$ (\ref{LPV}) can also be interpreted through linearization of potential vorticity (\textcolor{D}{linearized PV}). Setting $V^*=y\vec{\bm{e}}_1$ and $\theta^*=\alpha z$, we have\[
q_l:=\bm{\nabla}\theta^*\cdot\omega+\bm{\nabla}\theta\cdot\omega^*=\alpha\omega_3-\partial_z\theta=K.
\]

The linearized PV provides a useful tool for studying a variety of related problems. In Section \ref{ch2}, we provide two examples of the use of PV and discuss further properties of it, including its analogy with the 2D Euler linearized around a plane shear flow, as well as its applications to related systems such as the compressible Navier Stokes. We will show that, if the 2D Euler is spectrally unstable for the profile $U(y)$, then the inviscid Boussinesq equations are also spectrally unstable for $U(y)$. We refer to \cite{yao2026pvwavedecomposition} for further applications of the (linearized) potential vorticity.  \vspace{4pt}

Nevertheless, the potential-vorticity formulation comes with an increased regularity requirement, since both $\omega$ and $\bm{\nabla}\theta$ contain one additional derivative. Although this is usually harmless in the linear analysis, it becomes more significant in nonlinear energy estimates. However, this loss is not solely caused by the introduction of PV: the symmetrizing structure of the original systems may itself require additional regularity. The PV formulation should therefore be viewed as a trade-off: it makes the derivative loss explicit, but in return restores a symmetrizable structure and a more favorable energy framework.

\vspace{6pt}

{\normalsize\noindent\textbf{(Boussinesq) PV as a 3D feature.} }\vspace{4pt}

A fundamental difference between the 2D stratified fluid (the evolution in the $y$-direction is neglected) and the 3D fluid is the presence of potential vorticity (PV). Indeed, in 2D, the vorticity reduces to\[
\bm{\omega}=(\partial_z u_1-\partial_x u_3)\vec{\bm{e}}_2,\qquad\Rightarrow\qquad  \bm{\omega}\cdot\bm{\nabla}\theta=0.
\]That is, the potential vorticity vanishes in 2D. Another consequence of the occurrence of PV concerns the dispersive nature of the linearized dynamics. Since the PV vanishes in the 2D stably stratified fluid, both $u$ and $\theta$ are \textit{fully dispersive}. In contrast, in 3D, the linearized potential vorticity is nontrivial and is given by $q_l\alpha^{-1}=\omega_3=\partial_x u_2-\partial_y u_1$. Recall the linearized system\begin{flalign*}
\left\{
\begin{aligned}
& \partial_t U+\partial_z\bm{\nabla}\Delta^{-1}\theta=\theta\vec{\bm{e}}_3,\\ 
&   \partial_t\theta+\alpha u_3=0.
\end{aligned}
\right.
\end{flalign*}The solutions can be computed explicitly. In particular, we see that\footnote{Here $|k,\eta|:=\sqrt{k^2+\eta^2}$ and $|k,\eta,l|:=\sqrt{k^2+\eta^2+l^2}$.}\begin{align*}
&\widehat{u}_1=\widehat{u}_1(0)+\frac{kl}{|k,\eta|^2}\widehat{u}_3(0)+a_1\sin\left(\sqrt{\alpha}t\frac{|k,\eta|}{|k,\eta,l|}\right)+a_2\cos\left(\sqrt{\alpha}t\frac{|k,\eta|}{|k,\eta,l|}\right),\\
&\widehat{u}_2=\widehat{u}_2(0)+\frac{\eta l}{|k,\eta|^2}\widehat{u}_3(0)+\frac{\eta}{k}a_1\sin\left(\sqrt{\alpha}t\frac{|k,\eta|}{|k,\eta,l|}\right)+\frac{\eta}{k}a_2\cos\left(\sqrt{\alpha}t\frac{|k,\eta|}{|k,\eta,l|}\right).
\end{align*}The first two terms in each expression are time-independent and therefore $u_1$ and $u_2$ are only \textit{partially dispersive}. Similarly, explicit formulas for $\widehat{u}_3$ and $\widehat{\theta}$ can be computed, and it can be shown that they are fully dispersive.

In our companion work \cite{yao2026pvwavedecomposition}, we identify nontrivial potential vorticity as a key obstruction to fully dispersive or purely oscillatory dynamics of the simple zero modes.

\subsubsection{\textbf{Applications of linearized PV I: Tilted Couette flow}}\label{Tilted}

Motivated by the inviscid damping phenomenon observed for the plane Couette flow in \eqref{eq:A3} and \eqref{eq:A4}, as well as by the new lift-up mechanism discussed in Section \ref{norot}, we study the linear dynamics around the \textit{tilted Couette flow} on $\Omega_{\mathbb{R}}$. Both phenomena are fundamentally different from the behavior observed in the vertical Couette flow case \cite{coti2025suppression}. More precisely, the background flow and stratification are given by\begin{equation}
V^*=(y\cos\phi+z\sin\phi,0,0),\qquad \eta^*=\alpha z,\qquad \phi\in\left[0,\frac{\pi}{2}\right].\label{tcf1}
\end{equation}In other words, we study the Couette flow whose shear direction forms an arbitrary angle with the stratification. This family interpolates between the plane Couette flow at $\phi=0$ and the vertical Couette flow at $\phi=\frac{\pi}{2}$.\begin{figure}[H]
\centering
\begin{minipage}[c]{0.55\textwidth}
In this setting, the linearized potential vorticity again provides a particularly useful simplification of the system. More precisely, it is given by\[
q_l:=\alpha\omega_3+(\sin\phi\partial_y-\cos\phi\partial_z)\theta.
\]$q_l$ satisfies the following purely transport equation\[
 \partial_tq_l+(y\cos\phi +z\sin\phi)\partial_xq_l=0.
\]Its solution can be computed explicitly as in the fully perpendicular case discussed in Section \ref{norot}. Thus, we can prove the following linear stability estimates.
\end{minipage}%
\hfill
\begin{minipage}[c]{0.40\textwidth}
    \centering
    \includegraphics[width=\linewidth]{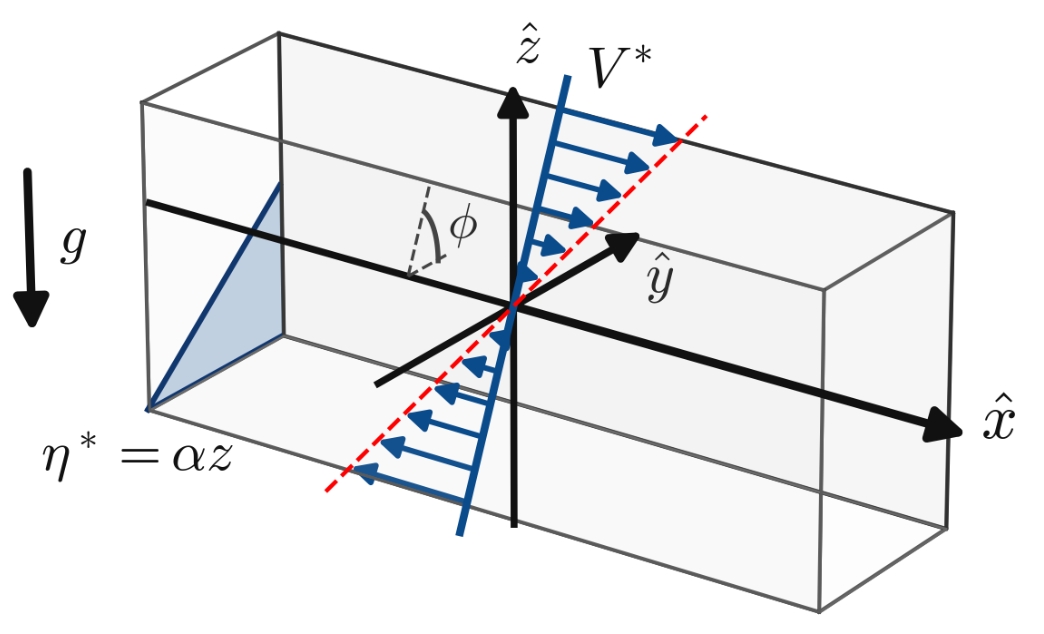}
    \caption{Tilted Couette flow}
\end{minipage}
\end{figure}

\begin{theorem}Let $(u,\theta)$ be the solution of the linearized 3D inviscid Boussinesq system $(\nu=\kappa=0)$ around (\ref{tcf1}) (solving (\ref{lineearizationaroundtcf})).\begin{itemize}
    \item \textbf{Inviscid damping of the non-zero modes}. For any $\phi\in\left(0,\frac{\pi}{2}\right)$, we have\begin{align*}
      \langle t\cos\phi\rangle  ||u_{1,\neq}||_{L^2}+||u_{2,\neq}||_{L^2}&\lesssim C_{\alpha,\phi}\left(||u_{\neq}(0)||_{H^{17}}+\left\|\frac{\theta_{\neq}(0)}{\alpha^\frac{1}{2}}\right\|_{H^{17}}\right),\\
         ||u_{3,\neq}||_{L^2}+\left\|\frac{\theta_{\neq}}{\alpha^\frac{1}{2}}\right\|_{L^2}&\lesssim C_{\alpha,\phi}\left(||u_{\neq}(0)||_{H^{15}}+\left\|\frac{\theta_{\neq}(0)}{\alpha^\frac{1}{2}}\right\|_{H^{15}}\right),
    \end{align*}where $C_{\alpha,\phi}:=e^{c\alpha^{-1/2}}\sec^{14}\phi$.
    \item \textbf{Asymptotic bound on $u_{2,\neq}.$} When $\phi\to0$, there holds the bound\[
    ||u_{2,\neq}||_{L^2}\lesssim  C_{\alpha,\phi}\max\left\{\frac{1}{\langle t\rangle},\tan\phi\right\}\left(||u_{\neq}(0)||_{H^{17}}+\left\|\frac{\theta_{\neq}(0)}{\alpha^\frac{1}{2}}\right\|_{H^{17}}\right).
    \]
    \item \textbf{Lift-up of $u_{1,0}$.} For any $\phi\in\left[0,\frac{\pi}{2}\right)$, $u_{1,0}$ experiences the lift-up effect, caused by the tilted operator $\cos\phi\cdot\mathcal{L}_{\text{Bou}}$.
\end{itemize}    
\end{theorem}

\noindent\textbf{Remark.} We do not attempt to optimize either the regularity assumptions for inviscid damping or the dependence of the decay estimate on $\sec\phi$, since the analysis is considerably more delicate than in the endpoint cases $\phi=0$ and $\phi=\frac{\pi}{2}$. The resulting damping behavior is rather unexpected: the decay rate degenerates as $\phi\to\frac{\pi}{2}$, and the damping disappears completely at $\phi=\frac{\pi}{2}$. Similarly, the lift-up effect persists whenever the shear direction is not parallel to the direction of stratification. This indicates that the orientation of the shear plays an essential role in the dynamics of stratified fluids.

\subsubsection{\textbf{Applications of linearized PV II: Description of leading order behaviour}}
As a simpler application, the linearized PV can be inverted, and can sometimes be used to describe the leading order dynamics. Recall the vertical Couette flow problem \cite{coti2024stability,coti2025suppression} (we set $\nu=\kappa=0$ here), the linearized PV is given by \[
q_l=\alpha\omega_3+\partial_y\theta.
\]The non-zero modes satisfy the following bound\[
||(u_1,u_2)_{\neq}||_{L^2}+\alpha^{-1/2}\langle t\rangle^\frac{3}{2}||u_{3,\neq}||_{L^2}+\langle t\rangle^\frac{1}{2}||\theta_{\neq}||\lesssim||(u,\alpha^{-1/2}\theta)_{\neq}(0)||_{H^3}.
\]Therefore, we have\[
\partial_x u_{2,\neq}-\partial_y u_{1,\neq}=\partial_x u_{2,\neq}(0,x-tz,y,z)-\partial_y u_{1,\neq}(0,x-tz,y,z)+\frac{1}{\alpha}\partial_y\theta_{\neq}(0,x-tz,y,z)+\mathcal{O}(t^{-1/2}).
\]Combining with the incompressibility, we obtain the leading order behaviour when $t\to\infty$:\begin{align*}
u_{1,\neq}&=-\partial_y\Delta^{-1}_{x,y}\left(\partial_xu_{2,\neq}(0,x-tz,y,z)-\partial_yu_{1,\neq}(0,x-tz,y,z)+\alpha^{-1}\partial_y\theta_{\neq}(0,x-tz,y,z)\right)+\mathcal{O}(t^{-1/2}),\\
u_{2,\neq}&=\partial_x\Delta^{-1}_{x,y}\left(\partial_xu_{2,\neq}(0,x-tz,y,z)-\partial_yu_{1,\neq}(0,x-tz,y,z)+\alpha^{-1}\partial_y\theta_{\neq}(0,x-tz,y,z)\right)+\mathcal{O}(t^{-1/2}).
\end{align*}Thus, although finding explicit solutions for $u_{1,\neq}$ and $u_{2,\neq}$ seems difficult, we are still able to obtain the $\mathcal{O}(1)$ terms at large time. Additionally, we see that the dynamics are being two-dimensionalized. For a further application, see \cite{yao2026pvwavedecomposition}, where we use PV to identify the leading-order behaviour of the temperature field around a vertical Couette flow in the presence of a background magnetic field.

\subsubsection{\textbf{Nonlinear stability of the plane Couette flow}}\label{thre}
Having established the linear theory, we now turn to the nonlinear problem. The key ingredient is the structural formulation associated with the rescaled linearized PV. It turns out that the dynamics are more interesting when the linear lift-up of $u_1$ is suppressed; we therefore start with this case. Introducing $\Xi=\partial_X U_2-\partial_{Y}^LU_1-\alpha^{-\frac{1}{2}}\partial_Z\Psi$ and $\Psi=\alpha^{-\frac{1}{2}}\Theta$ (see Section \ref{Notations} for the definitions of the functions denoted by capital letters), we prove the following nonlinear stability estimates when $\nu=\kappa$.

\begin{theorem}\label{nonlinear}(Nonlinear stability in the absence of linear lift-up effect.) Assume that $\alpha>0$, $s\geq4$ and $\nu\in(0,1)$. Suppose that\[
\int_{\mathbb{T}\times\mathbb{R}}u_2(0)\ dxdy=\int_{\mathbb{T}\times\mathbb{R}}\theta(0)\ dxdy=0
\]and\footnote{$C_\alpha=e^{c_1\alpha^{-1}}\left(1+\alpha^{-1}\right)$, as defined in Theorem \ref{linear, Boussinesq}.}\[
\left\|\left(u_{\text{in}},\frac{\theta_{\text{in}}}{\alpha^\frac{1}{2}}\right)\right\|_\mathcal{Y}\leq\varepsilon\lesssim C_\alpha^{-1}\nu^{1+\frac{\sqrt{2}}{3}},
\]where $||\cdot||_{\mathcal{Y}}$ is a weighted Sobolev norm, defined as\[
||G||_\mathcal{Y}:=||G||_{H^{s+1}\cap W^{s+4,1}}+\left\|\langle y\rangle^{\frac{3}{2}+}\partial_z\langle\bm{\nabla}\rangle^{s}G\right\|_{L^2},
\]then the following nonlinear estimates hold:\begin{itemize}
    \item \textbf{Inviscid damping and enhanced dissipation of the non-zero modes.} Let $\varpi>0$ be some small constant, then\[
   \langle t\rangle ||u_{1,\neq}||_{L^2}+\langle t\rangle||u_{2,\neq}||_{L^2}+||u_{3,\neq}||_{L^2}+\left\|\frac{\theta_{\neq}}{\alpha^\frac{1}{2}}\right\|_{L^2}\lesssim e^{-\varpi\nu^\frac{1}{3}t}\varepsilon,
    \]
    \item \textbf{Dispersive and $L^\infty$ bounds on simple zero modes.} There exists a number $\Upsilon\in\left[0,\frac{1}{3}\right]$ such that
\begin{align*}
&\left\|\overline{u}_2\right\|_{L^\infty}+\left\|\overline{u}_3\right\|_{L^\infty}+\left\|\frac{\overline{\theta}}{\alpha^\frac{1}{2}}\right\|_{L^\infty}\lesssim \frac{e^{-\nu t}}{\langle\alpha^\frac{1}{2}t\rangle^\frac{1}{3}}\varepsilon+\alpha^{-\frac{1}{6}}\nu^{-\frac{2}{3}}\varepsilon^2,\\
&\left\|\overline{u}_1\right\|_{L^\infty}\lesssim e^{-\nu t}\varepsilon+\left(\nu^{-1-\Upsilon}+\alpha^{-\frac{5}{12}}\nu^{-\frac{7}{6}}\right)\varepsilon^2.
\end{align*}    \item \textbf{Energy estimates}. 
    \begin{align*}
&\left\|\left(U,\Psi,\Xi\right)_{\neq}\right\|_{L^\infty_tH^s}+\nu^\frac{1}{2}\left\|\bm{\nabla}_L\left(U,\Psi,\Xi\right)_{\neq}\right\|_{L^2_tH^s}+\nu^\frac{1}{6}\left\|\left(U,\Psi,\Xi\right)_{\neq}\right\|_{L^2_tH^s}\lesssim\varepsilon,\\&
\left\|\overline{\left(U_2,U_3,\Psi\right)}\right\|_{L^\infty_tH^s}+\nu^\frac{1}{2}\left\|\bm{\nabla}_{y,z}\overline{\left(U_2,U_3,\Psi\right)}\right\|_{L^2_tH^s}\lesssim\varepsilon,\\
 &
\left\|\overline{U}_1\right\|_{L^\infty_tH^s}+\nu^\frac{1}{2}\left\|\partial_z\overline{U}_1\right\|_{L^2_tH^s}+\nu^{\frac{1}{2}-\Upsilon}\left\|\partial_y\overline{U}_1\right\|_{L^2_tH^s}\lesssim\nu^{-\Upsilon}\varepsilon,\\
&\left\|\widetilde{(U_1,U_3,\Psi)}\right\|_{L^\infty_tH^s}+\nu^\frac{1}{2}\left\|\partial_y\widetilde{(U_1,U_3,\Psi)}\right\|_{L^2_tH^s}\lesssim\varepsilon,
    \end{align*}
\end{itemize}provided that\footnote{For $\alpha\sim1$, one may take $\Upsilon=(2-\sqrt{2})/3$, so that the nonlinear energy bound for $\overline{u}_1$ allows transient growth of size $\nu^{-\Upsilon}$. For $\alpha\gtrsim\nu^{-4+2\sqrt{2}}$, one may take $\Upsilon=0$, and the bound becomes $O(\varepsilon)$ uniformly in time. Stronger stratification therefore suppresses the transient growth.}\begin{align*}
    \alpha\gtrsim 1+\nu^{6\Upsilon-\left(4-2\sqrt{2}\right)},
\end{align*}
\end{theorem}\begin{remark}
    The zero averages and finite weighted Sobolev norm assumptions on the initial are to suppress the linear lift-up effect. In this case, the wave-shear coupled lift-up effect due to $\mathcal{L}_{\text{Bou}}$ (\ref{Bouliftupope}) yields smaller transient growth of $\overline{u}_1$ of size $O\left(\nu^{-\Upsilon}\right)$ through the nonlinear interactions. In particular, the threshold for zero modes and non-zero modes are $\min\left\{\nu^{\frac{4}{3}-\Upsilon},\alpha^\frac{1}{6}\nu^{\frac{5}{3}-\Upsilon}\right\}$ and $\min\left\{\nu^{1+\frac{\sqrt{2}}{3}},\nu^{\frac{2}{3}+\frac{\sqrt{2}}{3}+\Upsilon}\right\}$, respectively.
\end{remark}

Without any additional assumptions on the initial data, including the finite weighted Sobolev norm, $\overline{u}_1$ exhibits the lift-up effect at both the linear and nonlinear levels. However, the $L^1$ assumption gives $\left\|\mathcal{L}_{\text{Bou}}\right\|_{L^1_0\rightarrow L^2_0}\lesssim\nu^{-\frac{1}{2}}\alpha^{-\frac{1}{4}}e^{-c\nu t}$, so that the transient growth of $\overline{u}_1$ is smaller.

\begin{theorem}\label{nonlinear1}(Nonlinear stability in the presence of linear lift-up effect.)
Assume that $\alpha>0$, $s\geq4$ and $\nu\in(0,1)$, and\[
\left\|\left(u_{\text{in}},\frac{\theta_{\text{in}}}{\alpha^\frac{1}{2}}\right)\right\|_{H^{s+1}\cap W^{s+4,1}}\leq\varepsilon\lesssim C_\alpha^{-1}\nu^{p},
\]where\footnote{That is, the threshold can be continuously improved due to stronger stratification.} \begin{align*}
    p=\frac{7}{6}+\frac{\sqrt{2}}{3},\qquad \text{if }\alpha\sim\mathcal{O}(1),\\
    p=1+\frac{\sqrt{2}}{3},\qquad \text{if }\alpha\gtrsim\nu^{-\frac{2}{3}}.
\end{align*}The bounds in Theorem \ref{nonlinear} still hold except the following\begin{itemize}
\item \textbf{$L^\infty$ bound on $\overline{U}_1$.} There holds
\begin{align*}
\left\|\overline{u}_1\right\|_{L^\infty}\lesssim \left(1+\nu^{-\frac{1}{2}}\alpha^{-\frac{1}{4}}\right) e^{-\frac{\nu t}{2}}\varepsilon+\nu^{-2}\varepsilon^2.
\end{align*}\item \textbf{Energy estimate of $\overline{U}_1$}. 
    \begin{align*}
\left\|\overline{U}_1\right\|_{L^\infty_tH^s}+\nu^\frac{1}{2}\left\|\bm{\nabla}_{y,z}\overline{U}_1\right\|_{L^2_tH^s}\lesssim\left(1+\nu^{-\frac{1}{2}}\alpha^{-\frac{1}{4}}\right)\varepsilon.
    \end{align*}
\end{itemize}
\end{theorem}
\begin{remark}
    The $L^1$ initial data not only yield the dispersive decay, but also weaken the transient growth due to the lift-up effect.
\end{remark}
\begin{remark}
For fixed $\alpha>0$, Theorems \ref{nonlinear} and \ref{nonlinear1} yield sufficient stability thresholds of order $\nu^{1+\sqrt{2}/3}$ and $\nu^{5/3+\sqrt{2}/3}$, respectively. These results can be compared with the parallel configuration studied in \cite{coti2024stability}, where the overall threshold improves from $\nu^{11/12}$ to $\nu^{8/9}$ under stronger, viscosity-dependent stratification assumptions. Another benchmark is the earlier $\nu^{3/2}$ threshold established for the classical 3D Navier Stokes equations in high Sobolev regularity \cite{9e0bb666-1e4e-3aa2-87b0-f916c2336491}. The persistence of nonlinear lift-up suggests that the optimal stability threshold may be smaller than that for the vertical Couette flow. 

To compare the underlying energy structures with the Navier--Stokes setting, we set aside the dispersive suppression of linear lift-up and the additional dispersion of the simple zero modes. A further difficulty is that a scaling-compatible symmetrizer is not directly available for the closed linearized subsystem in the variables $(\omega_3,u_3,\theta)$. Reformulation in terms of the linearized PV recovers a better energy structure at the cost of one additional derivative. As discussed in the previous section, this derivative requirement may contribute to the restrictive smallness conditions in our nonlinear estimates.

For comparison, in the parameter regime considered in \cite{huang2026stabilitythreshold3dboussinesq}, rotation suppresses lift-up and allows a suitable symmetrization of the corresponding $(\omega_3,u_3,\theta)$ subsystem, yielding a more favorable energy structure. A sufficient stability threshold of order $\nu^{14/15}$ is established in that setting. Whether the structural difficulty in the non-rotating problem constrains the optimal threshold or can be overcome by a more effective energy formulation (for the non-zero mode system) remains an open question.
\end{remark}

\subsection{\textbf{Notation}}\label{Notations}

\noindent\textbf{Fourier transform and Sobolev spaces}. For a function $G$ defined on $\Omega_{\mathbb{R}}$ or $\Omega_{\mathbb{T}}$, the corresponding Fourier transform $\widehat{G}$ is defined as (\ref{ourier1}) and (\ref{ourier2}). The inverse Fourier transform is then defined by\begin{align*}
   & G(x,y,z)=\frac{1}{8\pi^3}\sum_{(k,l)\in\mathbb{Z}^2}\int_{\mathbb{R}}\widehat{G}(k,\eta,l) e^{ikx+i\eta y+ilz}\ d\eta,\\
    &G(x,y,z)=\frac{1}{8\pi^3}\sum_{k\in\mathbb{Z}}\int_{\mathbb{R}^2}\widehat{G}(k,\eta,l) e^{ikx+i\eta y+ilz}\ d\eta dl.
\end{align*}We define the Japanese bracket $\langle\cdot\rangle$ by\[
\langle q_1,q_2,q_3,\cdots\rangle:=\sqrt{1+\sum_iq_i^2}.
\]The Sobolev norm $||\cdot||_{H^s}$ is then defined by\begin{align*}
    &||G||_{H^s(\Omega_{\mathbb{T}})}^2=\sum_{(k,l)\in\mathbb{Z}^2}\int_{\mathbb{R}} \langle k,\eta,l\rangle^{2s}|\widehat{G}|^2\ d\eta,\\
    &||G||_{H^s(\Omega_{\mathbb{R}})}^2=\sum_{k\in\mathbb{Z}}\int_{\mathbb{R}^2} \langle k,\eta,l\rangle^{2s}|\widehat{G}|^2\ d\eta dl.
\end{align*}In addition, for a general norm $X$, we write\[
||(A,B)||_{X}:=||A||_X+||B||_X
\]For a function $G=G(t,x,y,z)$ of spatial and temporal variables, we define the Banach space $L^p_t(a,b;X)$ by the norm\[
    ||G||_{L^p_t(a,b)X}=\left\|\left\|G\right\|_{X}\right\|_{L^p_t(a,b)}.
    \]
    
    In particular, if $(a,b)=(0,\infty)$, write it as $||G||_{L^p_tX}$.\vspace{7pt}

\noindent\textbf{Moving coordinates and differential operators}. Define\[
    X=x-ty,\qquad Y=y,\qquad Z=z
    \]The differential operators in this coordinate system are\begin{align}
       \partial^L_Y:=\partial_Y-t\partial_X,\qquad  \bm{\nabla}_L:=\begin{pmatrix}
            \partial_X\\\partial_Y-t\partial_X\\
            \partial_Z
        \end{pmatrix},\qquad \Delta_L:=\partial^2_X+(\partial_Y-t\partial_X)^2+\partial^2_Z.\label{notation1}
    \end{align} In addition to the standard Laplacian, we also define the horizontal and the anisotropic Laplacians\[
\Delta_h:=\partial^2_x+\partial^2_y,\qquad \Delta_{\beta}:=\partial^2_x+\partial^2_y+\beta\partial^2_z,\qquad \beta>0.
    \]In the moving frame, we write\[
    \Delta_H:=\partial^2_X+(\partial_Y-t\partial_X)^2,\qquad  \Delta_{\beta L}:=\Delta_H+\beta\partial^2_Z.
    \]We write $f$ and $p$ as the Fourier symbols of $-\Delta_H$ and $-\Delta_L$\[
    f:=k^2+(\eta-kt)^2,\qquad p:=f+l^2.
    \]The time derivative is denoted as\[
    \partial_tf:=-2k(\eta-kt).
    \]
    
     Apart from this, we use capital letters to denote the corresponding functions in the moving frame. For example $d(t,x,y,z)=D(t,X,Y,Z)$. In addition, we denote $\theta(t,x,y,z)=\Theta(t,X,Y,Z)$ and $K(t,x,y,z)=\mathcal{K}(t,X,Y,Z)$.\vspace{6pt}

\noindent\textbf{Miscellaneous}. \begin{itemize}
\item For two positive numbers $A$ and $B$, $A\lesssim B$ means $A\leq CB$ for some universal constant $C>0$. We write $A\approx B$ or $A\sim B$ if $A\lesssim B\lesssim A$.

    \item A superscript "$+$" denotes $+\delta$ for some sufficiently small $\delta>0$. For instance, $H^{s+}$ and $W^{k+,p}$ stand for $H^{s+\delta}$ and $W^{k+\delta,p}$, respectively.
    \item Approximation for $\min$ and $\max$: for $q>0$, we write $\min\{1,q\}\approx\frac{q}{\langle q\rangle}$ and $\max\{1,q\}\approx\langle q\rangle$. 
\end{itemize}

\subsection{\textbf{Main difficulties and proof strategy}}\label{maindiff}
Here we summarize the main obstacles and the organization of the article.

We first study the linear stability of the plane Couette flow in Section \ref{lsa}. As mentioned in \ref{norot}, we make use of the linearized PV $K$ (\ref{LPV}) to reformulate the $(u_3,\theta)$ system and make symmetrization possible. The zero modes can be solved explicitly in Fourier variables. To prove the decay of the non-zero modes, we focus on the $(u_{3,\neq},\theta_{\neq})$ system and recover $(u_{1,\neq},u_{2,\neq})$ by solving the div-curl system (\ref{divcurl1}). In this case, we introduce the following change of variables\[
\widehat{\Theta}_{\neq}=A_{\neq},\qquad \widehat{U}_{3,\neq}=\frac{1}{\alpha}\sqrt{\frac{z}{pf}}B_{\neq},\qquad z=2k^2p+\alpha f^2.
\]
Assuming, without loss of generality, that $A_{\neq}$ and $B_{\neq}$ are purely real or imaginary, respectively, the enhanced dissipation estimate then follows from the energy functional\[
\mathbb{E}_1:=|A_{\neq}|^2+|B_{\neq}|^2.
 \]The details of this linear argument are given in Appendix \ref{simpleproof}.

Although the above symmetrization is sufficient for the linear problem, it is not directly applicable to the nonlinear analysis in Section \ref{ch3}. The reason is that the change of variable costs too much regularity and the nonlinear energy estimate cannot be closed. Therefore, the main difficulty in the nonlinear problem is to construct a new symmetrization and a suitable energy functional. To this end, we introduce the following family of variables\begin{equation}
\widehat{\Theta}_{\neq}=\sqrt{\alpha}C\alpha^\frac{1}{4}\left(\frac{f}{p}\right)^\frac{\lambda}{2}\frac{p^\frac{1}{4}f^\frac{1}{4}}{z^\frac{1}{4}} A_{\neq},\qquad \widehat{U}_{3,\neq}=\frac{C}{\alpha^\frac{1}{4}}\left(\frac{f}{p}\right)^\frac{\lambda}{2}\frac{z^\frac{1}{4}}{f^\frac{1}{4}p^\frac{1}{4}}B_{\neq},\qquad z=2k^2p+\alpha f^2.\label{correctbariables}
 \end{equation}where $\lambda\in(0,1)$ is a free parameter and $C=C(\alpha,k,\eta,l,\lambda)>0$ is an arbitrary time-independent constant. At the linear level, different choices of \(C\) lead to equivalent energy estimates. We consider the energy functional\[
\mathbb{E}_2=|mA_{\neq}|^2+|mB_{\neq}|^2+|m\widehat{\mathcal{K}}_{\neq}|^2+2c(mA_{\neq}mB_{\neq}),
 \]where $m$ (\ref{multiplier}) is the Fourier multiplier used to control the vortex-stretching effect through enhanced dissipation and $c$ is another function whose upper bound depends on $\lambda$. The cross term is introduced to cancel some "bad" terms.
 
This construction is crucial for the nonlinear analysis: it both improves the regularity requirement for the decay estimates and allows the nonlinear energy estimates to be closed. To be more specific, we define\[
\mathcal{Z}=\mathcal{M}e^{\varpi\nu^\frac{1}{3}t}\langle \bm{\nabla}\rangle^s\mathcal{N}.
 \]Here $\mathcal{M}$ and $\mathcal{N}$ are some Fourier multipliers. We define the following energy functional\begin{equation}
 \mathbb{F}=\frac{1}{2}\left(||\mathcal{Z}A_{\neq}||_{L^2}^2+||\mathcal{Z}B_{\neq}||_{L^2}^2+||\mathcal{Z}\Xi_{\neq}||_{L^2}^2+\langle \mathcal{C}\mathcal{Z}A_{\neq},\mathcal{Z}B_{\neq}\rangle\right),\label{correctenergyfunctional}
 \end{equation}where $\mathcal{C}$ is the multiplier version of $c$. The choice of the coefficient $C(\alpha,k,\eta,l,\lambda)$ in the nonlinear energy estimate turns out to be very restricted. We show that the only case in which the energy estimate can be closed under certain appropriate assumptions is when $\lambda=\frac{1}{2}$. A reasonable choice of $C$ is then $C^{-1}=1+|l|^\frac{1}{2}/(\alpha^\frac{1}{4}|k|^\frac{1}{2})$. \vspace{4pt}

To prove the nonlinear stability result in Theorem \ref{nonlinear}, our strategy is to prove the following bootstrap argument\begin{theorem}\label{boot}Under the assumptions in Theorem \ref{nonlinear}, suppose that the following bounds hold:\begin{align*}
\left\|\mathcal{Z}F_{\neq}\right\|_{L^\infty_tL^2}^2+\nu\left\|\bm{\nabla}_L\mathcal{Z}F_{\neq}\right\|^2_{L^2_t L^2}+\left\|\sqrt{-\frac{\mathcal{N}'}{\mathcal{N}}}\mathcal{Z}F_{\neq}\right\|_{L^2_tL^2}^2&\leq 200\varepsilon^2,\qquad F\in\left\{U_1,U_2,A,B,\frac{K}{\alpha}\right\},\\
\left\|\overline{U}_{1}\right\|_{L^\infty_tH^s}^2+\nu^{1-2\Upsilon}\left\|\partial_y\overline{U}_{1}\right\|^2_{L^2_t H^s}+\nu\left\|\partial_z\overline{U}_{1}\right\|^2_{L^2_t H^s}&\leq 200\nu^{-2\Upsilon}\varepsilon^2,\\
\left\|\overline{F}\right\|_{L^\infty_tH^s}^2+\nu\left\|\bm{\nabla}\overline{F}\right\|^2_{L^2_t H^s}&\leq 200\varepsilon^2,\qquad F\in\left\{U_2,U_3,\frac{\Theta}{\sqrt{\alpha}}\right\},\\
\left\|\widetilde{F}\right\|_{L^\infty_tH^s}^2+\nu\left\|\bm{\nabla}\widetilde{F}\right\|^2_{L^2_t H^s}&\leq 200\varepsilon^2,\qquad F\in\left\{U_1,U_3,\frac{\Theta}{\sqrt{\alpha}}\right\}.
\end{align*}Then the same bounds hold with $200$ replaced by $100$. 
\end{theorem}
 To prove the Theorem \ref{nonlinear1}, the bootstrap assumption for $\overline{U}_1$ is set to be\[
 \left\|\overline{U}_1\right\|_{L^\infty_tH^s}^2+\nu\left\|\bm{\nabla}_{y,z}\overline{U}_1\right\|_{L^2_tH^s}^2\leq 200\left(1+\nu^{-1/2}\alpha^{-1/4}\right)^2\varepsilon^2.
 \]
 
 In fact, identifying the appropriate variables in (\ref{correctbariables}) and constructing the energy functional (\ref{correctenergyfunctional}) constitute the main technical difficulty of the nonlinear analysis. Once this nonlinear-compatible symmetrization is established, the remaining estimates can be handled within the weighted energy framework. In particular, the chosen variables avoid the delicate derivative distribution issues arising in nonlinear interactions, which are a central difficulty in several other three-dimensional Couette stability problems, such as \cite{coti2024stability,huang2026stabilitythreshold3dboussinesq}.\vspace{4pt}

Now we explain why \eqref{coupledsystem} does not have a valid symmetrizor. If we still want to apply the symmetrization method, by writing $\widehat{U}_3=bB$ and $\omega_3=c C$, $b$ and $c$ must satisfy  $b/c=\sqrt{2}|k|/\sqrt{pf}$, so that the "bad" term $BC$ in the energy estimate is canceled. However, as we have shown in Theorem \ref{linear, Boussinesq}, the correct scaling should be $U_3\sim \mathcal{O}(1)$ and $\omega_3\sim\mathcal{O}(1)$, which contradict the relation of $b/c$ chosen here. We remark that further decomposition of the unknowns as in \cite{yao2026pvwavedecomposition} may help to find out a better structure for mathematical analysis.\vspace{4pt}

In Section \ref{ch2}, we further investigate the properties and applications of the linearized PV formulation to related stability problems. In particular, we apply the same framework to the tilted Couette flow.

\section{\textbf{Linear stability of the plane Couette flow}}\label{ch1}
We start by analyzing the linearized problem. The linearized system reads\begin{flalign}
\left\{
\begin{aligned}
&\partial_t U+u_2\vec{\bm{e}}_1+y\partial_x U-\nu\Delta U-\theta\vec{\bm{e}}_3=\bm{\nabla}\Delta^{-1}\left[2\partial_x u_2-\partial_z\theta\right],\\
&\bm{\nabla}\cdot U=0,\\
    &\partial_t\theta+y\partial_x\theta+\alpha u_3=\kappa\Delta\theta.
\end{aligned}
\right.\label{eq:B2}
\end{flalign}In order to derive a suitable reformulation, we first set $\nu=\kappa=0$ for simplicity. Denote the vorticity as\[
\bm{\omega}=\bm{\nabla}\times U=\begin{pmatrix}
    \partial_y u_3-\partial_z u_2\\
    \partial_zu_1-\partial_x u_3\\
    \partial_x u_2-\partial_y u_1
\end{pmatrix}=\begin{pmatrix}
    \omega_1\\
    \omega_2\\
    \omega_3
\end{pmatrix}.
\]The linearized velocity equations read\begin{flalign*}
\left\{
\begin{aligned}
&\partial_t u_1+y\partial_x u_1=-u_2+ 2\partial^2_x\Delta^{-1} u_2-\partial_{xz}\Delta^{-1}\theta,\\
    &\partial_t u_2+y\partial_xu_2=2\partial_{xy}\Delta^{-1} u_2-\partial_{yz}\Delta^{-1}\theta,\\
    &\partial_t u_3+y\partial_xu_3=2\partial_{xz}\Delta^{-1} u_2+\Delta_h\Delta^{-1}\theta.
\end{aligned}
\right.
\end{flalign*}The linearized vorticity equations are
\begin{flalign*}
\left\{
\begin{aligned}
&\partial_t \omega_1+y\partial_x \omega_1=\partial_y\theta-\partial_x u_3,\\
    &\partial_t \omega_2+y\partial_x\omega_2=-\partial_x\theta-\partial_z u_2,\\
    &\partial_t \omega_3+y\partial_x\omega_3=-\partial_z u_3.
\end{aligned}
\right.
\end{flalign*}The continuity and temperature equations are\begin{flalign*}
\left\{
\begin{aligned}
&\partial_x u_1+\partial_y u_2+\partial_z u_3=0,\\
    &\partial_t \theta+y\partial_x\theta+\alpha u_3=0.
\end{aligned}
\right.
\end{flalign*}As discussed in \ref{norot} and \ref{maindiff}, the system does not have a good symmetrizing structure, and the analysis becomes more difficult. Nevertheless, we observe that the unknown $K:=\alpha\omega_3-\partial_z\theta$ satisfies the simple transport equation\[
\partial_tK+y\partial_x K=0.
\]It is then conserved in the moving frame. This quantity can be viewed as an external forcing in the system. With the help of $K$, we are able to derive closed systems for $(u_1,u_2)$ and $(u_3,\theta)$, respectively. Consider the pair $(q_1,\Gamma)$, where $q_1:=\Delta u_1$ and $\Gamma=\partial_x\omega_3$, they solve the following equations\begin{flalign}
\left\{
\begin{aligned}
&\partial_t \Gamma+y\partial_x\Gamma=\Delta_h\Delta^{-1}q_1+\partial^{-1}_x\partial_y\Gamma,\\
    &\partial_t q_1+y\partial_x q_1=-\partial^{-2}_x\Delta\Gamma+(2-\alpha)\Gamma-\partial^{-1}_x\partial_y q_1+\partial_x K
\end{aligned}
\right.\label{eq:B3}
\end{flalign}This is still complicated enough. However, note that we can actually solve for $u_1$ and $u_2$ in terms of other unknowns as they satisfy the following \textit{div-curl system}:\begin{flalign*}
\left\{
\begin{aligned}
&\partial_x u_2-\partial_y u_1=\frac{K+\partial_z\theta}{\alpha},\\
    &\partial_x u_1+\partial_y u_2=-\partial_z u_3.
\end{aligned}
\right.
\end{flalign*}On the other hand, using the above relation, we derive a new system for $(u_3,\theta):$\begin{flalign}
\left\{
\begin{aligned}
&\partial_t\theta+y\partial_x\theta=-\alpha u_3,\\
    &\partial_t u_3+y\partial_x u_3=2\partial_{xz}\Delta_h^{-1}\Delta^{-1}\left(\partial_x\left(\frac{K+\partial_z\theta}{\alpha}\right)-\partial_{zy}u_3\right)+\Delta_h\Delta^{-1}\theta.
\end{aligned}
\right.\label{eq:B4}
\end{flalign}This system is much simpler than (\ref{eq:B3}); notably, the interaction between $(u_1,u_2)$ and $(u_3,\theta)$ is only through the initial data.

We now study the decay properties of the non-zero modes of the velocity and temperature fields using this new system. Taking the Fourier transform of (\ref{eq:B3}) and (\ref{eq:B4}) in the moving frame, the system of equations reduces to
\begin{flalign}
\left\{
\begin{aligned}
&\partial_t \widehat{\Gamma}=\frac{f}{p}\widehat{Q}_1+\frac{\eta-kt}{k}\widehat{\Gamma},\\
    &\partial_t \widehat{Q}_1=-\frac{p}{k^2}\widehat{\Gamma}+(2-\alpha)\widehat{\Gamma}-\frac{\eta-kt}{k}\widehat{Q}_1+ik\widehat{\mathcal{K}},
\end{aligned}
\right.\label{eq:B5}
\end{flalign}while\begin{flalign}
\left\{
\begin{aligned}
&\partial_t\widehat{\Theta}=-\alpha\widehat{U}_3,\\
    &\partial_t \widehat{U}_3=-2\frac{kl}{pf}\left(\frac{ik\widehat{\mathcal{K}}-kl\widehat{\Theta}}{\alpha}+l(\eta-kt)\widehat{U}_3\right)+\frac{f}{p}\widehat{\Theta}.
\end{aligned}
\right.\label{eq:B6}
\end{flalign}Although the long-time behaviour of the solutions is not immediately apparent, the numerical simulations in Fig.\ref{numinvi1} suggest the following linear behaviour:\begin{equation}
\left|\widehat{U}_{1,\neq}\right|\lesssim t^{-1},\qquad \left|\widehat{U}_{2,\neq}\right|\lesssim t^{-1},\qquad \left|\widehat{U}_{3,\neq}\right|\sim 1,\qquad \left|\widehat{\Theta}_{\neq}\right|\sim 1.\label{predictionoflid}
\end{equation}We now show that the predicted decay rates agree with the observed ones.

\begin{figure}[H]
    \centering

    \begin{minipage}[t]{0.47\textwidth}
        \centering
        \includegraphics[width=\linewidth]{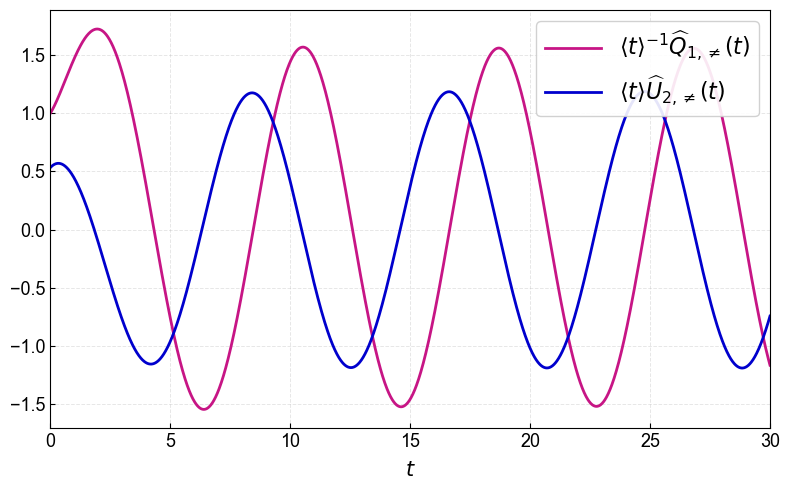}
        \par\vspace{-1.6ex}
       {\small \textbf{(a)}}
    \end{minipage}
     \hspace{0.008\textwidth}
    \begin{minipage}[t]{0.47\textwidth}
        \centering
        \includegraphics[width=\linewidth]{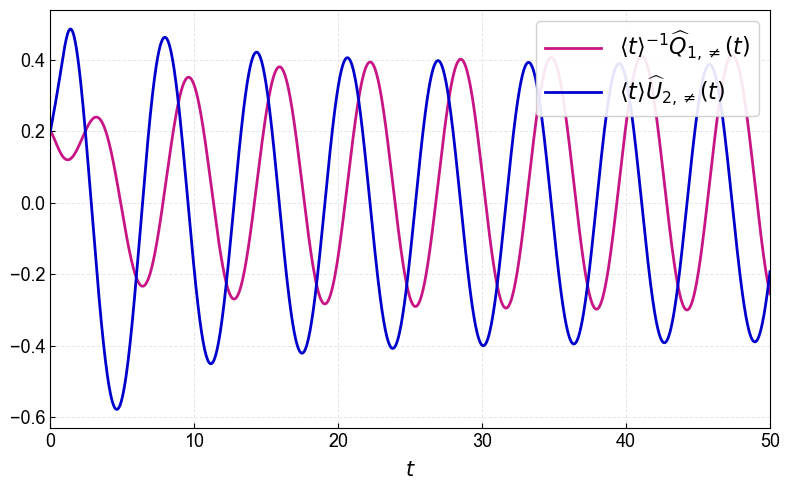}
        \par\vspace{-1.6ex}
         {\small\textbf{(b)}}
    \end{minipage}

   \par\vspace{-0.1ex}
    \begin{minipage}[t]{0.47\textwidth}
        \centering
        \includegraphics[width=\linewidth]{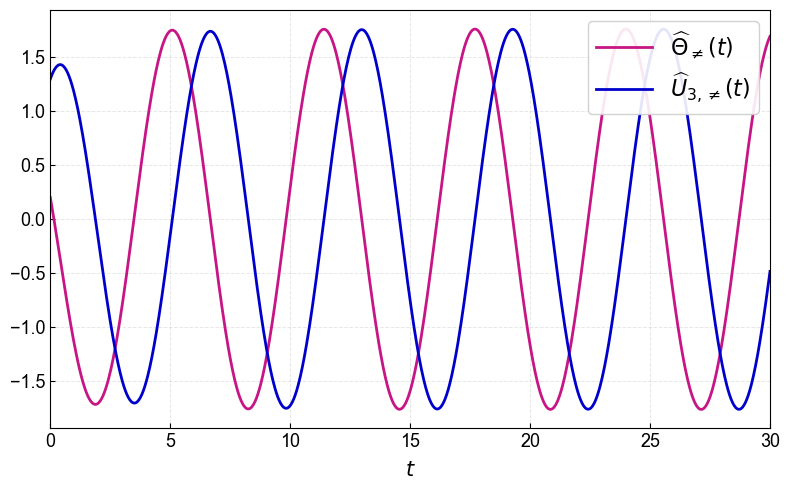}
        \par\vspace{-1.6ex}
         {\small\textbf{(c)}}
    \end{minipage}
     \hspace{0.008\textwidth}
    \begin{minipage}[t]{0.47\textwidth}
        \centering
        \includegraphics[width=\linewidth]{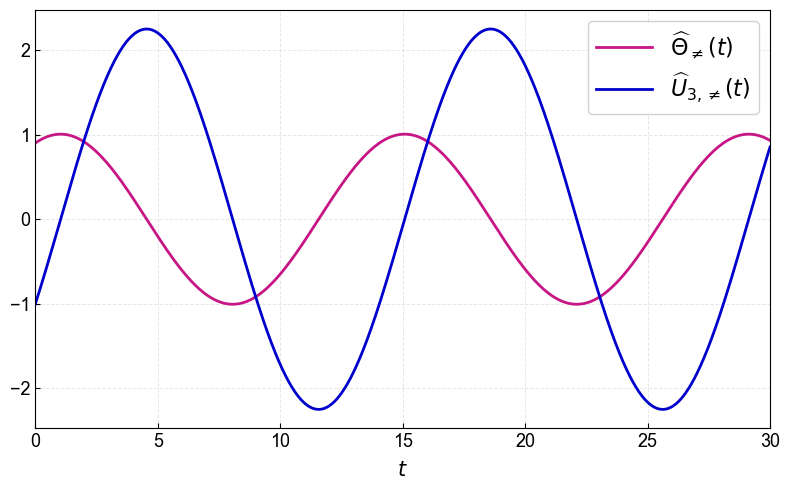}
        \par\vspace{-1.6ex}
         {\small\textbf{(d)}}
    \end{minipage}

    \caption{Numerical solutions of \eqref{eq:B5} and
    \eqref{eq:B6}. (a,b) Evolution of
    $\bigl(\langle t\rangle^{-1}\widehat{Q}_{1,\neq},
    \langle t\rangle\widehat{U}_{2,\neq}\bigr)$
    for \eqref{eq:B5} with nonzero
    $\widehat{\mathcal K}$, for
    $(k,l,\eta)=(1,1,-1)$ and $(1,0,1)$, respectively.
    (c,d) Evolution of
    $(\widehat{\Theta}_{\neq},\widehat{U}_{3,\neq})$
    for \eqref{eq:B6}, for
    $(k,l,\eta)=(1,0,-1)$ and $(2,1,0)$, respectively}
    \label{numinvi1}
\end{figure} 
\subsection{\textbf{Inviscid damping and enhanced dissipation}}\label{lsa}
In this section, we provide a rigorous proof of the prediction (\ref{predictionoflid}). The idea is to analyze the simpler system-$(U_3,\Theta)$ using the energy method.

Consider the linearized system in the viscous case with $\nu=\kappa$.\begin{flalign*}
\left\{
\begin{aligned}
&\partial_t\theta+y\partial_x\theta=-\alpha u_3+\nu\Delta\theta,\\
    &\partial_t u_3+y\partial_x u_3=2\partial_{xz}\Delta_h^{-1}\Delta^{-1}\left(\partial_x\left(\frac{K+\partial_z\theta}{\alpha}\right)-\partial_{zy}u_3\right)+\Delta_h\Delta^{-1}\theta+\nu\Delta u_3.
\end{aligned}
\right.
\end{flalign*}In the moving frame, we take the Fourier transform and perform the change of variables \[\widehat{\Theta}_{\neq}=\sqrt{\alpha}a(t)A,\qquad \widehat{U}_{3,\neq}=b(t)B,\qquad\widehat{\mathcal{K}}_{\neq}=\alpha\Xi.\]The system reads \begin{align*}
\partial_tA=&-\textcolor{D}{\sqrt{\alpha} \frac{b}{a}}B-\frac{a'}{a}A-\nu p A,\\
    \partial_t B=&-2\frac{ik^2l}{ bpf}\Xi+\sqrt{\alpha}\left(\textcolor{D}{\frac{2}{\alpha}\frac{k^2}{f}+\frac{f}{p}}-\frac{2}{\alpha}\frac{k^2}{p}\right)\frac{aA}{b}+\left(\frac{p'l^2}{pf}-\frac{b'}{b}\right)B-\nu p B.
\end{align*}To remove the slowly decaying terms in the system (or to partially symmetrize the system), we pick\begin{equation}
\frac{b}{a}=\frac{1}{\sqrt{\alpha}}\frac{\sqrt{z}}{\sqrt{pf}},\qquad z=2k^2p+\alpha f^2,\label{symmetrizer}
\end{equation}so that the three terms with slowly decaying coefficients can be canceled in the energy estimate. Choosing the two terms in $B$ is tricky: if we only choose $\frac{f}{p}$, due to the anisotropy, i.e. when $|l|\rightarrow\infty$, the energy has exponential growth in terms of $|l|$. If we include the term $-\frac{2}{\alpha}\frac{k^2}{p}$ as well, the energy functional which we will specify later requires $\alpha>c$ for some $c>0$ to ensure the coercivity. Since this term is integrable-in-time, we do not have to worry about it too much as it does not cause large growth as long as $\alpha$ is bounded away from zero.

To avoid cumbersome notation, we assume that $A$, $B$ and $i\Xi$ are purely real-valued and purely imaginary-valued functions respectively. This causes no loss of generality, since one may decompose the linear system into its real and imaginary parts and apply the energy estimates to each part separately. Now, we define the following energy functional with a cross term\[
\mathbb{E}=\frac{1}{2}\left[|A|^2+|B|^2+s(AB)\right],
\]where $s=s(t,k,\eta,l)$ is a bounded function which we are going to specify later. The purpose of introducing it is to allow for some cancellation; this is particularly useful for the nonlinear analysis. In fact, it suffices to set $s=0$ for the linear analysis, and the proof is much simpler, see Section \ref{isisi} for details.\vspace{2pt}

Taking the time derivative of $\mathbb{E}$, we see that\begin{align*}
    \partial_t\mathbb{E}&=-\frac{a'}{a}\left(|A|^2-|B|^2\right)+(1-\lambda)\frac{f'l^2}{pf}|B|^2-\frac{2}{\sqrt{\alpha}}\frac{k^2}{p}\frac{a}{b}AB-\frac{s}{\sqrt{\alpha}}\frac{k^2}{p}\frac{a}{b}|A|^2-\frac{2ik^2l}{pfb}\Xi B+\frac{s'}{2}AB\\
    &-\frac{\sqrt{\alpha}s}{2}\frac{b}{a}(|B|^2-|A|^2)+(1-\lambda)\frac{s}{2}\frac{f'l^2}{pf}AB-\frac{isk^2l}{pfb}\Xi A-\nu p|A|^2-\nu p |B|^2-\nu p sAB,
\end{align*}where we have chosen\[
\frac{a'}{a}=\lambda\left(\frac{f'l^2}{pf}\right)-\frac{b'}{b},\qquad \lambda\in[0,1].
\]The reason for such a choice is rather technical. Indeed, if $\lambda=1$, as we will discuss later, the coercivity cannot be guaranteed, i.e. it is possible that $\mathbb{E}<0$ for certain $A$ and $B$. In fact, in the $B$ equation, there is another type of slowly decaying coefficient $\frac{p'l^2}{pf}$, which is due to the vortex stretching effect. This coefficient causes some growth of the solution and this effect cannot be overcome using the symmetrization method. As a result, we use the cross term $s(AB)$ to allow some cancellation. Now we have obtained two equations for $a$ and $b$, solving them gives the following expressions\[
a=C\alpha^\frac{1}{4}\left(\frac{f}{p}\right)^\frac{\lambda}{2}\frac{p^\frac{1}{4}f^\frac{1}{4}}{z^\frac{1}{4}},\qquad b=C\left(\frac{f}{p}\right)^\frac{\lambda}{2}\frac{1}{\alpha^\frac{1}{4}}\frac{z^\frac{1}{4}}{f^\frac{1}{4}p^\frac{1}{4}},
\]where $C$ is some constant that may depend on $k,\eta,l$. To cancel the oscillatory terms $\left(|A|^2-|B|^2\right)$, we set\begin{equation}
\frac{s}{2}\cdot\alpha^\frac{1}{2}\cdot\frac{b}{a}=\frac{a'}{a}.\label{crossfunction}
\end{equation}This simplifies to\[
\frac{s}{2}\sqrt{\frac{z}{pf}}=\frac{\lambda}{2}\left(\frac{f'}{f}-\frac{p'}{p}\right)+\frac{1}{4}\left(\frac{p'}{p}+\frac{f'}{f}-\frac{z'}{z}\right).\]It is crucial to compute the upper bound for  $\frac{s}{2}$  to ensure the coercivity. In particular, we prove the following.

\begin{lemma}\label{upperboundofs} For any $\lambda\geq0$, denote $\zeta=\sup_{t,k\neq0,\eta,l}\left|\frac{s}{2}\right|$, then $\zeta$ can be computed explicitly\[
\zeta=\frac{1}{2\sqrt{2}}+\frac{\lambda}{\sqrt{2}}.
\]
\end{lemma}

\noindent\textit{Proof.} See Appendix \ref{appendix1}.

The coercivity holds true only if $\zeta<1$, which is when $\lambda<\frac{2\sqrt{2}-1}{2}\approx 0.91$. Therefore, it is not possible to cancel all the slowly decaying terms in form of $\frac{f'l^2}{pf}$ in the energy estimate. We now go back to the energy estimate, recall that $\partial_t\mathbb{E}$ satisfies\begin{align*}
    \partial_t\mathbb{E}&=(1-\lambda)\frac{f'l^2}{pf}|B|^2-\frac{2}{\sqrt{\alpha}}\frac{k^2}{p}\frac{a}{b}AB-\frac{s}{\sqrt{\alpha}}\frac{k^2}{p}\frac{a}{b}|A|^2+\frac{2ik^2l}{pfb}\Xi B+\frac{s'}{2}AB\\
    &+(1-\lambda)\frac{s}{2}\frac{f'l^2}{pf}AB+\frac{isk^2l}{pfb}\Xi A-\nu p|A|^2-\nu p |B|^2-\nu p sAB.
\end{align*}Since $2\zeta:=\sup_{k,\eta,l,t}\left|s\right|$, we then have\[
\frac{1}{2}\left(1-\zeta\right)\left[|A|^2+|B|^2\right]\leq\mathbb{E}\leq\frac{1}{2}\left(1+\zeta\right)\left[|A|^2+|B|^2\right].
\]Using this relation and the fact $\frac{1}{\sqrt{\alpha}}\frac{k^2}{p}\frac{a}{b}\leq\frac{k^2}{\alpha^\frac{1}{2}f}$, we have\begin{align*}
\partial_t\mathbb{E}+2\nu p\frac{1-\zeta}{1+\zeta}\mathbb{E}&\leq\frac{2}{1-\zeta}\left(\frac{2k^2}{\alpha^\frac{1}{2}f}+\frac{|s'|}{4}\right)\mathbb{E}+\frac{4k^2|l|}{pfb}\sqrt{\mathbb{E}}|\Xi|\\&+(1-\lambda)\frac{f'l^2}{pf}|B|^2+\frac{s(1-\lambda)}{2}\frac{f'l^2}{pf}(AB)\\
&\leq\frac{2}{1-\zeta}\left(\frac{2k^2}{\alpha^\frac{1}{2}f}+\frac{|s'|}{4}\right)\mathbb{E}+\frac{4k^2|l|}{pfb}\sqrt{\mathbb{E}}|\Xi|\\&+(1-\lambda)\frac{f'l^2}{pf}|B|^2+\frac{s(1-\lambda)}{2}\frac{f'l^2}{pf}(AB).
\end{align*}Since $s'$ has finitely many zeros, \[
\int_0^t|s'(\tau)|\ d\tau\leq\int_{\mathbb{R}}|s'(\tau)|\ d\tau<C
\]for some constant $C>0$ uniformly bounded in terms of $k,\eta$ and $l$. It remains to control the two terms in the last line of the inequality. The presence of these terms is due to the vortex stretching as just mentioned. One strategy to balance these terms is to define an additional multiplier $m$ such that $m\geq0$ and $m'\leq 0$. Let $\mathbb{F}=m^2\mathbb{E}$, we therefore obtain an inequality for $\mathbb{F}$,\begin{align*}
    \partial_t\mathbb{F}+2\nu p\frac{1-\zeta}{1+\zeta}\mathbb{F}&\leq2\frac{m'}{m}\mathbb{F}+(1-\lambda)\frac{f'l^2}{pf}|mB|^2+\frac{s(1-\lambda)}{2}\frac{f'l^2}{pf}(mA mB)\\
    &+\frac{2}{1-\zeta}\left(\frac{2k^2}{\alpha^\frac{1}{2}f}+\frac{|s'|}{4}\right)\mathbb{F}+\frac{4k^2|l|}{pfb}\sqrt{\mathbb{F}}|m\Xi|.
\end{align*}The terms on the second line are integrable and hence we do not need to treat them separately. For the first line, notice that the viscous term can dominate these terms, as long as\[
\nu p\gg\frac{|f'|l^2}{pf}.
\]This condition can be relaxed to \[
\nu f\gtrsim \frac{|k|}{f^\frac{1}{2}}\qquad\Rightarrow\qquad \left|\frac{\eta}{k}-t\right|\gtrsim\nu^{-\frac{1}{3}}.\]Thus we only need to consider the case $\left|\frac{\eta}{k}-t\right|\lesssim\nu^{-\frac{1}{3}}$. We set\begin{flalign}
\frac{m'}{m}=\left\{
\begin{aligned}
&-\varrho_1\frac{f'l^2}{pf}
&t\in\left[\frac{\eta}{k},\frac{\eta}{k}+\frac{20}{\nu^\frac{1}{3}}\right],
\\
&\varrho_2\frac{f'l^2}{pf}&t\in\left[\frac{\eta}{k}-\frac{20}{\nu^\frac{1}{3}},\frac{\eta}{k}\right], \\
  &0&\left|\frac{\eta}{k}-t\right|>\frac{20}{\nu^\frac{1}{3}},
\end{aligned}
\right.\label{multiplier}
\end{flalign}where $\varrho_1,\varrho_2\geq0$ are some constants satisfying\begin{equation}
\boxed{\varrho_1\geq\frac{1-\lambda}{2}\left(1+\frac{1}{\sqrt{1-\zeta^2}}\right),\qquad \varrho_2\geq\frac{1-\lambda}{2}\left(-1+\frac{1}{\sqrt{1-\zeta^2}}\right)}\qquad \forall\lambda\in \left[0,\frac{2\sqrt{2}-1}{2}\right].\label{twocons}
\end{equation}The initial data of $m$ will also be specified later. Now we consider the following stretching terms\begin{equation}
\frac{m'}{m}|mA|^2+\frac{m'}{m}|mB|^2+s\frac{m'}{m}(mAmB)+(1-\lambda)\frac{f'l^2}{pf}|mB|^2+(1-\lambda)\frac{s}{2}\frac{f'l^2}{pf}(mAmB)\label{streching balance}
\end{equation}It suffices to consider the case in which $\left|\frac{\eta}{k}-t\right|\leq\frac{20}{\nu^\frac{1}{3}}$. There are two cases to consider\begin{itemize}
    \item When $f'\geq0$, for (\ref{streching balance}) to be non-positive, we need\[
    -\varrho_1 x^2+s\left(\frac{1-\lambda}{2}-\varrho_1\right)x+(-\varrho_1+1-\lambda)\leq0,
    \]where $x=\frac{mA}{mB}$. To ensure that the above inequality holds for any $x\in\mathbb{R}$, we require
    \[
    \varrho_1\geq\left(\frac{1}{2}+\frac{1}{\sqrt{4-s^2}}\right)(1-\lambda).
    \]
    \item When $f'\leq0$, the non-positivity of (\ref{streching balance}) now requires\[
    \varrho_2 x^2+\left(\frac{1-\lambda}{2}+\varrho_2\right)sx+\varrho_2+1-\lambda\geq0
    \]Similarly, this yields\[
    \varrho_2\geq\left(-\frac{1}{2}+\frac{1}{\sqrt{4-s^2}}\right)(1-\lambda).
    \]
\end{itemize}We now study the multiplier $m$. In order to simplify the analysis, we set $m=m_1m_2$, with each of them corresponding to different parts of $m$: \begin{flalign}
\frac{m'_1}{m_1}=\left\{
\begin{aligned}
&-\varrho_1\frac{f'l^2}{pf}
&t\in\left[\frac{\eta}{k},\frac{\eta}{k}+\frac{20}{\nu^\frac{1}{3}}\right],
\\
  &0&t\notin\left[\frac{\eta}{k},\frac{\eta}{k}+\frac{20}{\nu^\frac{1}{3}}\right].
\end{aligned}
\right.,\qquad \frac{m'_2}{m_2}=\left\{
\begin{aligned}
&\varrho_2\frac{f'l^2}{pf}
&t\in\left[\frac{\eta}{k}-\frac{20}{\nu^\frac{1}{3}},\frac{\eta}{k}\right],
\\
  &0&t\notin\left[\frac{\eta}{k}-\frac{20}{\nu^\frac{1}{3}},\frac{\eta}{k}\right],
\end{aligned}
\right.
\end{flalign}with the initial data\begin{align}
    m_1(0,k,\eta,l)=\left(\frac{k^2+400k^2\nu^{-\frac{2}{3}}}{k^2+400k^2\nu^{-\frac{2}{3}}+l^2}\cdot\frac{k^2+l^2}{k^2}\right)^{\varrho_1},\nonumber\\
    m_2(0,k,\eta,l)=\left(\frac{k^2+400k^2\nu^{-\frac{2}{3}}}{k^2+400k^2\nu^{-\frac{2}{3}}+l^2}\cdot\frac{k^2+l^2}{k^2}\right)^{\varrho_2}.\label{initialdata229}
\end{align}The explicit solutions are computed as follows:
\begin{align*}
    m_1=\left\{
\begin{aligned}
&\left(\frac{k^2+400k^2\nu^{-\frac{2}{3}}}{k^2+400k^2\nu^{-\frac{2}{3}}+l^2}\cdot\frac{k^2+l^2}{k^2}\right)^{\varrho_1}
&\frac{\eta}{k}+\frac{20}{\nu^\frac{1}{3}}<0,
\\
  &\left(\frac{k^2+\eta^2}{k^2+\eta^2+l^2}\cdot\frac{k^2+l^2}{k^2}\right)^{\varrho_1}&\frac{\eta}{k}<0<\frac{\eta}{k}+\frac{20}{\nu^\frac{1}{3}}<t,\\
  & \left(\frac{p}{f}\cdot\frac{k^2+\eta^2}{k^2+\eta^2+l^2}\cdot\frac{k^2+l^2}{k^2}\cdot\frac{k^2+400k^2\nu^{-\frac{2}{3}}}{k^2+400k^2\nu^{-\frac{2}{3}}+l^2}\right)^{\varrho_1}&\frac{\eta}{k}<0<t<\frac{\eta}{k}+\frac{20}{\nu^\frac{1}{3}},\\
  & 1&0<\frac{\eta}{k}<\frac{\eta}{k}+\frac{20}{\nu^\frac{1}{3}}<t,\\
  & \left(\frac{p}{f}\cdot\frac{k^2+400k^2\nu^{-\frac{2}{3}}}{k^2+400k^2\nu^{-\frac{2}{3}}+l^2}\right)^{\varrho_1}&0<\frac{\eta}{k}<t<\frac{\eta}{k}+\frac{20}{\nu^\frac{1}{3}},\\
  &\left(\frac{k^2+400k^2\nu^{-\frac{2}{3}}}{k^2+400k^2\nu^{-\frac{2}{3}}+l^2}\cdot\frac{k^2+l^2}{k^2}\right)^{\varrho_1} &t<\frac{\eta}{k}.
\end{aligned}
\right.
\end{align*}Similarly,
\begin{align*}
m_2=\left\{
\begin{aligned}
&\left(\frac{k^2+400k^2\nu^{-\frac{2}{3}}}{k^2+400k^2\nu^{-\frac{2}{3}}+l^2}\cdot\frac{k^2+l^2}{k^2}\right)^{\varrho_2}
&\frac{\eta}{k}<0,
\\
  &\left(\frac{f}{p}\cdot\frac{(k^2+l^2)^2}{k^4}\cdot\frac{k^2+400k^2\nu^{-\frac{2}{3}}}{k^2+400k^2\nu^{-\frac{2}{3}}+l^2}\right)^{\varrho_2}&\frac{\eta}{k}-\frac{20}{\nu^\frac{1}{3}}<0<\frac{\eta}{k}<t,\\
  & \left(\frac{f}{p}\cdot\frac{k^2+l^2}{k^2}\cdot\frac{k^2+\eta^2+l^2}{k^2+\eta^2}\frac{k^2+400k^2\nu^{-\frac{2}{3}}}{k^2+400k^2\nu^{-\frac{2}{3}}+l^2}\right)^{\varrho_2}&\frac{\eta}{k}-\frac{20}{\nu^\frac{1}{3}}<0<t<\frac{\eta}{k},\\
  & 1&0<\frac{\eta}{k}-\frac{20}{\nu^\frac{1}{3}}<\frac{\eta}{k}<t,\\
  & \left(\frac{f}{p}\cdot\frac{k^2+l^2}{k^2}\right)^{\varrho_2}&0<\frac{\eta}{k}-\frac{20}{\nu^\frac{1}{3}}<t<\frac{\eta}{k},\\
  &\left(\frac{k^2+400k^2\nu^{-\frac{2}{3}}}{k^2+400k^2\nu^{-\frac{2}{3}}+l^2}\cdot\frac{k^2+l^2}{k^2}\right)^{\varrho_2} &t<\frac{\eta}{k}-\frac{20}{\nu^\frac{1}{3}}.
\end{aligned}
\right.
\end{align*}It turns out that $m_1$ and $m_2$ enjoy the following properties due to monotonicity\begin{align*}
1\lesssim m_1\lesssim\nu^{-\frac{2}{3}\varrho_1},&\qquad 1\lesssim m_2\lesssim\nu^{-\frac{2}{3}\varrho_2},\\
m_1+\frac{m_1(0)}{m_1}\lesssim p^{\varrho_1}(0),&\qquad m_2+\frac{m_2(0)}{m_2}\lesssim p^{\varrho_2}(0).
\end{align*}As a result, $m$ satisfies the following bound.\[
1\lesssim m\lesssim \nu^{-\frac{2}{3}\varrho_1-\frac{2}{3}\varrho_2},\qquad m+\frac{m(0)}{m}\lesssim  p^{\varrho_1+\varrho_2}(0).
\]

With this definition of the Fourier multiplier, the energy inequality now becomes\[
\partial_t\mathbb{F}+\frac{9}{5}\nu p\frac{1-\zeta}{1+\zeta}\mathbb{F}\leq\frac{2}{1-\zeta}\left(\frac{k^2}{\alpha^\frac{1}{2}f}+\frac{|s'|}{4}\right)\mathbb{F}+\frac{4k^2|l|}{pfb}\sqrt{\mathbb{F}}|m\Xi|
\]Using change of variables $\mathbb{F}e^{\frac{9}{5}\nu\frac{1-\zeta}{1+\zeta}\int_0^tp\ ds }\to\mathbb{F}$ and $\Xi  e^{\frac{9}{10}\nu\frac{1-\zeta}{1+\zeta}\int_0^tp\ ds }\to\Xi$, we see that the following bound holds\[
\mathbb{F}^\frac{1}{2}\leq \mathbb{F}^\frac{1}{2}(0) e^{\frac{1}{2}\int_0^t h(s)\ ds}+\frac{1}{2}e^{\frac{1}{2}\int_0^t h(s)\ ds}\int_0^t \frac{k^2|l|}{bpf}|m\Xi|e^{-\frac{1}{2}\int_0^s h(\tau)\ d\tau},
\]with $h=\frac{2}{1-\zeta}\left(\frac{k^2}{\alpha^\frac{1}{2}f}+\frac{|s'|}{4}\right)$. $\int_0^\infty h<C$ for some universal constant $C>0$ when $\alpha$ is fixed, therefore,\[
\mathbb{F}^\frac{1}{2}\lesssim \mathbb{F}^\frac{1}{2}(0) +|m(0)\Xi(0)|\int_0^t \frac{k^2|l|}{bpf}\ ds
\]Since $b=C\left(\frac{f}{p}\right)^\frac{\lambda}{2}\frac{z^\frac{1}{4}}{\alpha^\frac{1}{4}f^\frac{1}{4}p^\frac{1}{4}}$, we compute that\begin{align*}
\int_0^t\frac{k^2|l|}{bpf}\ ds&=\frac{\alpha^\frac{1}{4}}{C}\int_0^t\left(\frac{p}{f}\right)^\frac{\lambda}{2}\frac{k^2|l|}{p^\frac{3}{4}f^\frac{3}{4}z^\frac{1}{4}}\ ds\\
&\lesssim\frac{1}{C}\int_0^t\left(\frac{p}{f}\right)^\frac{\lambda}{2}\frac{k^2}{f^{\frac{5}{4}}p^\frac{1}{4}}\ ds\\
&\lesssim\frac{1}{C}\left\{
\begin{aligned}
&\langle l\rangle^{\lambda-\frac{1}{2}}
&\lambda\geq\frac{1}{2},
\\
  &1&\lambda\leq\frac{1}{2}.
\end{aligned}
\right.
\end{align*}

As a result, we have \begin{align*}
\frac{|m(t)\widehat{\Psi}_{\neq}|}{a(t)}+\frac{|m(t)\widehat{U}_{3,\neq}|}{b(t)}\lesssim\frac{|m(0)\widehat{\Psi}_{\neq}(0)|}{a(0)}+\frac{|m(0)\widehat{U}_{3,\neq}(0)|}{b(0)}\\+|m(0)|\left[|\widehat{U}_{\neq}(0)|+\alpha^{-\frac{1}{2}}|\widehat{\Psi}_{\neq}(0)|\right]\left\{
\begin{aligned}
&p^{\frac{\lambda}{2}+\frac{1}{4}}(0)
&\lambda\geq\frac{1}{2},\\
  &p^{\frac{1}{2}}(0)&\lambda\leq\frac{1}{2}.
\end{aligned}
\right.
\end{align*}Notice that the following bounds hold for $a$ and $b$\begin{align*}
a&\lesssim C\alpha^\frac{1}{4}\frac{f^{\frac{1}{4}+\frac{\lambda}{2}}}{p^{\frac{\lambda}{2}-\frac{1}{4}}z^\frac{1}{4}}\lesssim \left\{
\begin{aligned}
&C
&\lambda\geq\frac{1}{2},\\
  &C\langle l\rangle^{\frac{1}{2}-\lambda}&\lambda\leq\frac{1}{2}.
\end{aligned}
\right.,\\
b&\lesssim\frac{C}{\alpha^\frac{1}{4}}\frac{z^\frac{1}{4}f^{\frac{\lambda}{2}-\frac{1}{4}}}{p^{\frac{1}{4}+\frac{\lambda}{2}}}\lesssim C,\\
a^{-1}&\lesssim\frac{1}{C\alpha^\frac{1}{4}}\frac{z^\frac{1}{4}p^{\frac{\lambda}{2}-\frac{1}{4}}}{f^{\frac{1}{4}+\frac{\lambda}{2}}}\lesssim\frac{1}{C\alpha^\frac{1}{4}}\frac{|k|^\frac{1}{2}p^\frac{1}{4}+\alpha^\frac{1}{4}f^\frac{1}{2}}{f^{\frac{1}{4}+\frac{\lambda}{2}}} p^{\frac{\lambda}{2}-\frac{1}{4}}\lesssim\frac{1}{C}(1+l^2)^\frac{\lambda}{2},\\
b^{-1}&\lesssim\frac{\alpha^\frac{1}{4}}{C}\frac{p^{\frac{1}{4}+\frac{\lambda}{2}}}{z^\frac{1}{4}f^{\frac{\lambda}{2}-\frac{1}{4}}}\lesssim\frac{1}{C}(1+l^2)^{\frac{1}{4}+\frac{\lambda}{2}}.
\end{align*}As a result, we can obtain the following estimate\begin{align*}
    \left\|\frac{\theta}{\alpha^\frac{1}{2}}\right\|_{L^2}\lesssim\left\{
\begin{aligned}
& \left\|\frac{\theta}{\alpha^\frac{1}{2}}\right\|_{H^{2(\varrho_1+\varrho_2)+\lambda}}+||u_{3,\neq}(0)||_{H^{2(\varrho_1+\varrho_2)+\lambda+\frac{1}{2}}}+\left\|\left(u_{\neq},\frac{\theta}{\alpha^\frac{1}{2}}\right)_{\neq}(0)\right\|_{H^{2(\varrho_1+\varrho_2)+\lambda+\frac{1}{2}}}
&\lambda\geq\frac{1}{2},\\
  & \left\|\frac{\theta}{\alpha^\frac{1}{2}}\right\|_{H^{2(\varrho_1+\varrho_2)+\frac{1}{2}}}+||u_{3,\neq}(0)||_{H^{2(\varrho_1+\varrho_2)+1}}+\left\|\left(u_{\neq},\frac{\theta}{\alpha^\frac{1}{2}}\right)_{\neq}(0)\right\|_{H^{2(\varrho_1+\varrho_2)+\frac{3}{2}-\lambda}}&\lambda\leq\frac{1}{2}.
\end{aligned}
\right.
\end{align*}Similarly, $||u_{3,\neq}||_{L^2}$ satisfies\[
||u_{3,\neq}||_{L^2}\lesssim \left\|\frac{\theta(0)}{\alpha^\frac{1}{2}}\right\|_{H^{2(\varrho_1+\varrho_2)+\lambda}}+||u_{3,\neq}(0)||_{H^{2(\varrho_1+\varrho_2)+\lambda+\frac{1}{2}}}+\left\{
\begin{aligned}
&\left\|\left(u_{\neq},\frac{\theta}{\alpha^\frac{1}{2}}\right)_{\neq}(0)\right\|_{H^{2(\varrho_1+\varrho_2)+\lambda+\frac{1}{2}}}
&\lambda\geq\frac{1}{2},\\
  &\left\|\left(u_{\neq},\frac{\theta}{\alpha^\frac{1}{2}}\right)_{\neq}(0)\right\|_{H^{2(\varrho_1+\varrho_2)+1}}&\lambda\leq\frac{1}{2}.
\end{aligned}
\right.
\]The enhanced dissipation is due to the factor\footnote{Notice that $\int_0^t k^2+(\eta-ks)^2+l^2\ ds\geq\frac{t^3}{12}$.} $e^{-\frac{9}{10}\frac{1-\zeta}{1+\zeta}\nu\int_0^tp\ ds}$. It remains to choose $\lambda$ such that the regularity is minimized. Here we set \[
\varrho_1=\frac{1-\lambda}{2}\left(1+\frac{1}{\sqrt{1-\zeta^2}}\right),\qquad \varrho_2=\frac{1-\lambda}{2}\left(-1+\frac{1}{\sqrt{1-\zeta^2}}\right).
\]The regularity is minimized when $\lambda>\frac{1}{2}$, the critical value $\lambda_m$ solves the following equation \begin{equation}
5-6\lambda_m=4\left(\frac{7}{8}-\frac{\lambda_m}{2}-\frac{\lambda^2_m}{2}\right)^\frac{3}{2},\qquad \lambda_m\in\left(\frac{1}{2},1\right).\label{minreg}
\end{equation}As can be seen in Figure \ref{fig9}, the root is $\lambda_m\approx 0.7845$, we therefore obtain the following decay estimate\begin{align*}
  ||u_{1,\neq}||_{L^2}+||u_{2,\neq}||_{L^2}&\lesssim C_\alpha\frac{e^{-\frac{1}{277}\nu t^3}}{\langle t\rangle}\left(\left\|\frac{\theta_{\neq}(0)}{\alpha^\frac{1}{2}}\right\|_{H^{\sigma+2}}+||u_{\neq}(0)||_{H^{\sigma+2}}\right),\\
 \left\|\frac{\theta_{\neq}}{\alpha^\frac{1}{2}}\right\|_{L^2}+||u_{3,\neq}||_{L^2}&\lesssim C_\alpha e^{-\frac{1}{277}\nu t^3}\left(\left\|\frac{\theta_{\neq}(0)}{\alpha^\frac{1}{2}}\right\|_{H^{\sigma}}+||u_{\neq}(0)||_{H^{\sigma}}\right),
\end{align*}where $C_\alpha=e^{\frac{c_1}{\alpha}}\left(1+\alpha^{-1}\right)$ for some $c_1>0$ and $\sigma:=\frac{2(1-\lambda_m)}{\sqrt{1-\left(\frac{1+2\lambda_m}{2\sqrt{2}}\right)^2}}+\lambda_m+\frac{1}{2}\approx2.315$.\begin{figure}[H]
    \centering
    \includegraphics[width=0.66\textwidth]{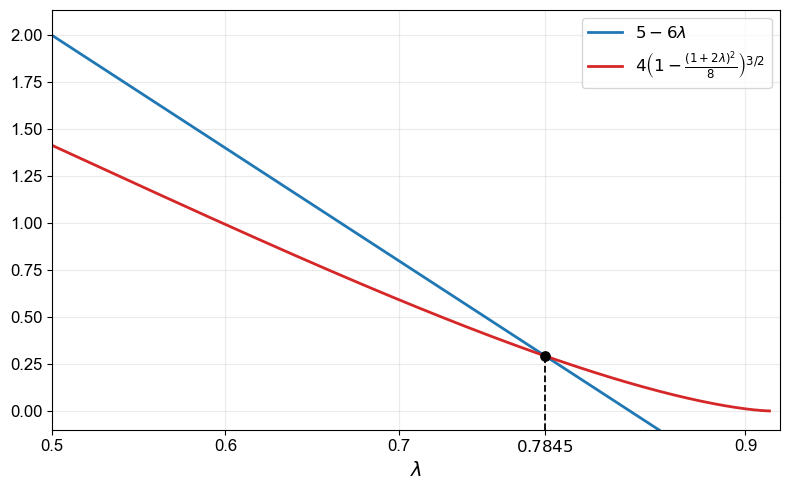}
    \caption{Numerical determination of the root of equation \eqref{minreg}}
    \label{fig9}
\end{figure} For the bound on $(u_{1,\neq},u_{2,\neq})$, we have used the fact\[
|\widehat{U}_{1,\neq}|+|\widehat{U}_{2,\neq}|\lesssim\frac{p(0)}{\langle t\rangle}\left(|\widehat{U}_{\neq}(0)|+|\widehat{\Theta}_{\neq}(0)|+|\widehat{U}_{3,\neq}|+|\widehat{\Theta}_{\neq}|\right),
\]which is direct by solving the div-curl system (\ref{divcurl1}).

\begin{remark} In other 3D Couette flow problems, the Fourier multipliers used, when needed, to control the vortex stretching through viscosity, are different from the one used here. To be specific, in other problems, the multipliers are in form of (see for instance \cite{zelati2025stabilityviscousthreedimensionalrotating,9e0bb666-1e4e-3aa2-87b0-f916c2336491,huang2026stabilitythreshold3dboussinesq})\begin{flalign*}
\frac{M'}{M}=\left\{
\begin{aligned}
&-c\frac{|p'|}{p}
&\left|\frac{\eta}{k}-t\right|\leq\frac{c}{\nu^\frac{1}{3}},
\\
  &0&\left|\frac{\eta}{k}-t\right|>\frac{c}{\nu^\frac{1}{3}}.
\end{aligned}
\right.\qquad\text{or}\qquad \frac{M'}{M}=\left\{
\begin{aligned}
&-c\frac{p'}{p}
&t\in\left[\frac{\eta}{k},\frac{\eta}{k}+\frac{c}{\nu^\frac{1}{3}}\right],
\\
  &0&t\notin\left[\frac{\eta}{k},\frac{\eta}{k}+\frac{c}{\nu^\frac{1}{3}}\right].
\end{aligned}
\right.
\end{flalign*}for some constant $c>0$. Here we mention two differences between $M$ and $m$ used here.
\begin{itemize}
    \item The multiplier $M$ satisfies a commutator-type estimate (uniform in $k$), namely
    \[
    M(t,k,\eta,l)\lesssim\langle \eta-\eta',l-l'\rangle^{c'}M(t,k,\eta',l'),
    \qquad c'\in[0,2].
    \]
    However, $m$ does not satisfy a similar estimate. This is entirely due to the anisotropy- for example, when $\frac{\eta}{k}<0<t<\frac{\eta}{k}+\frac{20}{\nu^{\frac{1}{3}}}$, the presence of a single $k^2$ cannot be used to control $\langle\eta-\eta',l-l'\rangle$ or $\langle\eta',l'\rangle$. Therefore, there is no such type of commutator estimate (with small $c'$). In the nonlinear analysis, we can therefore only use the viscosity to absorb the term $m$.
    \item The time taken for the viscosity to overcome the vortex stretching is still $c\nu^{-\frac{1}{3}}+\frac{\eta}{k}$, even though the derivative $\frac{|f'|l^2}{pf}$ seems to be much smaller than $\frac{p'}{p}$. In particular, we need to find a minimal time $t^*\geq\frac{\eta}{k}$ such that\[
\nu p\geq c\frac{|f'|l^2}{pf}.
\]Setting $k=1$, we see that the requirement now becomes\[
\left(1+(\eta-t)^2\right)\left(1+(\eta-t)^2+l^2\right)^2\geq c|\eta-t|l^2\nu^{-1}.
\]Taking $l\approx\frac{1}{\nu^\frac{1}{3}}$, we see that it is necessary to have $|\eta-t|\gtrsim\frac{1}{\nu^\frac{1}{3}}$, which is the critical time.
\end{itemize}\end{remark}

\subsection{\textbf{The wave-shear coupled lift-up effect of the zero modes}}\label{weaklu}

We now turn to the zero modes. It suffices to study the inviscid problem with $\nu=\kappa=0$. The linearized (simple) zero mode system reads\[\left\{
\begin{aligned}
&\partial_t\overline{u}_1=-\overline{u}_2,\\
    &\partial_t\overline{u}_2=-\partial_{yz}\Delta^{-1}_{y,z}\overline{\theta},\\
    &\partial_t \overline{u}_3=\partial^2_y\Delta^{-1}_{y,z}\overline{\theta},\\
    &\partial_t\overline{\theta}=-\alpha \overline{u}_3.
\end{aligned}
\right.\]Taking the Fourier transform, we can compute the exact solutions directly when $\frac{\eta}{l}\neq0$:\begin{flalign}
\left\{
\begin{aligned}
& \widehat{\overline{u}}_{1}=\widehat{\overline{u}}_{1}(0)+\frac{2l}{\alpha\eta}\widehat{\overline{\theta}}(0)\sin^2\left(\frac{\sqrt{\alpha}t\eta}{2\sqrt{\eta^2+l^2}}\right)-\frac{\sqrt{\eta^2+l^2}}{\sqrt{\alpha}\eta}\widehat{\overline{u}}_2(0)\sin\left(\frac{\sqrt{\alpha}t\eta}{\sqrt{\eta^2+l^2}}\right),\\&
    \widehat{\overline{u}}_{2}=-\frac{l}{\sqrt{\alpha}\sqrt{\eta^2+l^2}} \widehat{\overline{\theta}}(0)\sin\left(\frac{\sqrt{\alpha}t\eta}{\sqrt{\eta^2+l^2}}\right)+ \widehat{\overline{u}}_{2}(0)\cos\left(\frac{\sqrt{\alpha}t\eta}{\sqrt{\eta^2+l^2}}\right),\\&
     \widehat{\overline{u}}_{3}= \widehat{\overline{u}}_{3}(0)\cos\left(\frac{\sqrt{\alpha}t\eta}{\sqrt{\eta^2+l^2}}\right)+\frac{\eta}{\sqrt{\alpha}\sqrt{\eta^2+l^2}} \widehat{\overline{\theta}}(0)\sin\left(\frac{\sqrt{\alpha}t\eta}{\sqrt{\eta^2+l^2}}\right),\\&
     \widehat{\overline{\theta}}= \widehat{\overline{\theta}}(0)\cos\left(\frac{\sqrt{\alpha}t\eta}{\sqrt{\eta^2+l^2}}\right)-\sqrt{\alpha}\frac{\eta \widehat{\overline{u}}_{3}(0)-l \widehat{\overline{u}}_{2}(0)}{\sqrt{\eta^2+l^2}}\sin\left(\frac{\sqrt{\alpha}t\eta}{\sqrt{\eta^2+l^2}}\right).
\end{aligned}
\right.\label{eq:B9}
\end{flalign}In the exact formula of $ \widehat{\overline{u}}_{1}$ (\ref{eq:B9}), there is a singular term $\eta^{-1}$, which can cause some growth of the solution. Specifically, we have\[
\left|\frac{l}{\eta}\sin\left(\frac{\sqrt{\alpha}t\eta}{\sqrt{\eta^2+l^2}}\right)\right|,\ \left|\frac{l}{\eta}\left(1-\cos\left(\frac{\sqrt{\alpha}t\eta}{\sqrt{\eta^2+l^2}}\right)\right)\right|\lesssim\sqrt{\alpha} t,\qquad \text{as }\ \frac{\eta}{l}\rightarrow0.
\]

This phenomenon is known as the \textit{lift-up} effect. It can lead to large transient growth when the viscosity is small and algebraic growth in the inviscid case. \vspace{8pt}

\noindent\textbf{Shear flow in 3D Euler.} Let us recall the 3D Euler linearized around the shear flow $U(y)\vec{\bm{e}}_1$, the simple zero modes satisfy the non-normal system\[
\partial_t\begin{pmatrix}
   \overline{u}_1\\\overline{u}_2\\\overline{u}_3
\end{pmatrix}=\begin{pmatrix}
    0&-U'(y)&0\\
    0&0&0\\
    0&0&0
\end{pmatrix}\begin{pmatrix}
   \overline{u}_1\\\overline{u}_2\\\overline{u}_3
\end{pmatrix}.
\]It is clear that $\overline{u}_2$ and $\overline{u}_3$ are both conserved. $\overline{u}_1$ can be solved as follows\[
\overline{u}_1=\overline{u}_1(0)-tU'(y)\overline{u}_2(0).
\]Whenever $U'(y)\overline{u}_2(0)\not\equiv0$, the streamwise component exhibits linear growth in time.\vspace{9pt}

\noindent\textbf{Plane shear flow in 3D inviscid Boussinesq.} 
We now consider the inviscid Boussinesq system. Since $\overline{u}_3$ can be recovered by incompressibility, the linearization of the $(\overline{u}_1,\overline{u}_2,\overline{\theta})$ system around the plane shear flow $U(y)\vec{e}_1$ is formulated as\[
\partial_t\begin{pmatrix}
   \overline{u}_1\\\overline{u}_2\\\overline{\theta}
\end{pmatrix}=\begin{pmatrix}
    0&-U'(y)&0\\
    0&0&-\partial_{yz}\Delta^{-1}_{y,z}\\
    0&\alpha\partial^{-1}_z\partial_y&0
\end{pmatrix}\begin{pmatrix}
   \overline{u}_1\\\overline{u}_2\\\overline{\theta}
\end{pmatrix}.
\]The subsystem for $(\overline{u}_2,\overline{\theta})$ evolves independently of $\overline{u}_1$, while the full system retains a non-normal structure. For $\alpha>0$, the explicit solution is\begin{align*}
\overline{u}_1
=&\overline{u}_1(0)
-U'(y)\mathcal{L}^{-1}\Bigl(
\sin(t\mathcal{L})\overline{u}_2(0)-\partial_{yz}\Delta^{-1}_{y,z}\mathcal{L}^{-1}
\left(1-\cos(t\mathcal{L})\right)\overline{\theta}(0)
\Bigr),\\
\overline{u}_2
=&\cos(t\mathcal{L})\overline{u}_2(0)
-\partial_{yz}\Delta^{-1}_{y,z}\mathcal{L}^{-1}
\sin(t\mathcal{L})\overline{\theta}(0),\\
\overline{\theta}
=&\cos(t\mathcal{L})\overline{\theta}(0)
-\mathcal{L}\sin(t\mathcal{L})\overline{u}_3(0)+\alpha\partial_{yz}\Delta^{-1}_{y,z}\mathcal{L}^{-1}
\sin(t\mathcal{L})\overline{u}_2(0).
\end{align*}where $\mathcal{L}:=\sqrt{\alpha}|\partial_y||\bm{\nabla}_{y,z}|^{-1}$. The lift-up effect is caused by the following operator\begin{equation}
\mathcal{L}_l:=-\alpha^{-\frac{1}{2}}U'(y)\partial_y^{-1}\partial_z\left(e^{\pm it\mathcal{L}}-I_d\right).\label{liftup123}
\end{equation}The growth is associated with the regime of small internal-wave frequencies. When $\alpha=0$, the internal-gravity wave disappears and the explicit solution is\[
\overline{\theta}=\overline{\theta}(0),\qquad \overline{u}_2=\overline{u}_2(0)-t\partial_{yz}\Delta^{-1}_{y,z}\overline{\theta}(0),\qquad \overline{u}_1=\overline{u}_1(0)-tU'(y)\overline{u}_2(0)+\frac{t^2}{2}U'(y)\partial_{yz}\Delta^{-1}_{y,z}\overline{\theta}(0).
\]The instability is due to the explicit growth factors $t$ and $t^2$, which produce algebraic growth similar to that in the Navier Stokes setting. We refer to the corresponding mechanism for $\alpha>0$ as the \textit{wave-shear coupled lift-up effect}.

\vspace{9pt}

\noindent\textbf{Vertical shear flow in 3D inviscid Boussinesq.} As a comparison, if we instead consider the parallel shear flow $U(z)\vec{\bm{e}}_1$, the linearization of $(\overline{u}_1,\overline{u}_3,\overline{\theta})$ is given by\[
\partial_t\begin{pmatrix}
   \overline{u}_1\\\overline{u}_3\\\overline{\theta}
\end{pmatrix}=\begin{pmatrix}
    0&-U'(z)&0\\
    0&0&\partial_{y}^2\Delta^{-1}_{y,z}\\
    0&-\alpha&0
\end{pmatrix}\begin{pmatrix}
   \overline{u}_1\\\overline{u}_3\\\overline{\theta}
\end{pmatrix}.
\]The corresponding explicit solution is
\begin{align*}
\overline{u}_1
&=\overline{u}_1(0)
-U'(z)\Bigl[
\sin(t\mathcal{L})\Bigl(
\alpha^{-1}\mathcal{L}\overline{u}_3(0)
-\partial_{yz}\Delta^{-1}_{y,z}\mathcal{L}^{-1}
\overline{u}_2(0)\Bigr)
+\alpha^{-1}\left(1-\cos(t\mathcal{L})\right)
\overline{\theta}(0)
\Bigr],\\
\overline{u}_3
&=\cos(t\mathcal{L})\overline{u}_3(0)
+\alpha^{-1}\mathcal{L}\sin(t\mathcal{L})
\overline{\theta}(0),\\
\overline{\theta}
&=\cos(t\mathcal{L})\overline{\theta}(0)
-\mathcal{L}\sin(t\mathcal{L})\overline{u}_3(0)+\alpha\partial_{yz}\Delta^{-1}_{y,z}\mathcal{L}^{-1}
\sin(t\mathcal{L})\overline{u}_2(0).
\end{align*}Incompressibility removes the apparent inverse-frequency singularity. More generally, consider tilted shear flows of the form \[V^*=\left(U(y)\cos\phi+U(z)\sin\phi\right)\vec{\bm{e}}_1,\qquad\phi\in[0,\pi/2].\]The potential lift-up effect is generated by the following operator\[
\mathcal{L}_{ll}:=-\alpha^{-\frac{1}{2}}\cos\phi \cdot U'(y)\partial_y^{-1}\partial_z\left(e^{\pm it\mathcal{L}}-I_d\right).
\]This contribution vanishes at $\phi=\pi/2$, consistently with the suppression of lift-up in the parallel configuration. We will discuss the tilted Couette flow in Section \ref{ang} as an example.\vspace{5pt}

\begin{remark}
    If $U'(y)\in L^\infty$, then due to Lemma \ref{suppofl3}, the following bounds on the transient growth hold for any $\phi\in[0,\pi/2)$:\begin{align*}
    \left\|e^{\nu t\Delta}\mathcal{L}_{ll}\overline{g}\right\|_{L^2}&\lesssim\cos\phi\cdot\nu^{-1}||U'(y)||_{L^\infty_y}||\overline{g}||_{L^2},\\
    \left\|e^{\nu t\Delta}\mathcal{L}_{ll}\overline{g}\right\|_{L^\infty}&\lesssim\cos\phi\cdot\nu^{-\frac{1}{2}}\alpha^{-\frac{1}{4}}||U'(y)||_{L^\infty_y}||\overline{g}||_{L^2}.
    \end{align*}
\end{remark}

In the rest of this section, we study quantitatively the lift-up effect due to the operator \eqref{liftup123} in the Couette flow case.\vspace{6pt}

\noindent\textbf{Bounds on the transient growth}\vspace{4pt}

We start with the $\Omega_{\mathbb{R}}$ case. Setting $U(y)=y$, the zero mode $u_{1,0}$ can be solved explicitly\[u_{1,0}=e^{\nu t\Delta}\left(u_{1,0}(0)-tu_{2,0}(0)\right).
\]Recall the following $L^p-L^q$ estimate of the heat kernel on $\mathbb{R}^2$, \[
\left\|e^{\nu t\Delta}f\right\|_{L^p}\lesssim\frac{1}{(\nu t)^{\frac{1}{q}-\frac{1}{p}}}\left\|f\right\|_{L^q}.
\]Thus, the only available $t$-independent $L^p-L^q$ estimate is the following
\begin{equation}
||u_{1,0}||_{L^\infty}\lesssim||u_{1,0}(0)||_{L^\infty}+\nu^{-1}||u_{2,0}(0)||_{L^1}.\label{nse}
\end{equation}

If we revisit the exact solution (\ref{eq:B9}), the growth-in-time function $\sqrt{\alpha}t$ of $u_{1,0}$ can be interpreted in Fourier space as the singular limit as $ \eta l^{-1}\to 0$. This is fundamentally different from the Navier--Stokes case discussed above. In fact, this type of lift-up effect produces smaller transient growth when measuring it in different norms, and can even be completely suppressed. To see the differences, we first prove the following estimate for the lift-up operator $\mathcal{L}_{\text{Bou}}$ (\ref{Bouliftupope}).

\begin{lemma}\label{nonsup1} The following estimate holds
\begin{align*}
    \left\|\partial_y^{-1} \partial_z e^{\nu t \Delta}\left(e^{\sqrt{\alpha}t\mathcal{R}_y}-I_d\right)f\right\|_{L^\infty(\mathbb{R}^2)}\lesssim\nu^{-\frac{1}{2}}\alpha^{\frac{1}{4}}||f||_{L^2}.
\end{align*}\end{lemma}

\noindent\textsc{Proof}. Taking the Fourier transform, we see that\begin{align*}
    \left|\partial_y^{-1} \partial_z e^{\nu t \Delta}\left(e^{\sqrt{\alpha}t\mathcal{R}_y}-I_d\right)f\right|&\lesssim\left|\int_{\mathbb{R}^2}\frac{l}{\sqrt{\eta^2+l^2}} e^{-\nu t(\eta^2+l^2)}\frac{\left(e^{i\sqrt{\alpha}t\frac{\eta}{\sqrt{\eta^2+l^2}}}-1\right)}{\frac{\eta}{\sqrt{\eta^2+l^2}}}\widehat{f}e^{i\eta y+ilz}\ d\eta dl\right|\\
    &\lesssim\left\|\widehat{f}\right\|_{L^2}\left\|\frac{l}{\eta}e^{-\nu t(\eta^2+l^2)}\left(e^{i\sqrt{\alpha}t\frac{\eta}{\sqrt{\eta^2+l^2}}}-1\right)\right\|_{L^{2}}.
\end{align*}Denote the second norm as $I$, we see that\begin{align*}
  I^{2}&\lesssim \int_{\mathbb{R}^2}e^{-2\nu t(\eta^2+l^2)}\left|\frac{l}{\eta}\right|^{2}\left|e^{i\sqrt{\alpha}t\frac{\eta}{\sqrt{\eta^2+l^2}}}-1\right|^{2}\ d\eta dl\\
  &\lesssim \int_0^{2\pi}\int_0^\infty r\tan^2\theta\left|e^{i\sqrt{\alpha}t\cos\theta}-1\right|^{2}e^{-2\nu r^2t}\ drd\theta,
\end{align*}where we have used the polar coordinates $\eta=r\cos\theta$ and $l=r\sin\theta$. We have\[
\int_0^\infty re^{-c\nu r^2 t}\ dr\lesssim \nu^{-1}t^{-1}.
\]Using the change of variable $x=\cos\theta$, we have\[
\int_0^{2\pi}\tan^2\theta\left|e^{i\sqrt{\alpha}t\cos\theta}-1\right|^2\ d\theta\lesssim\int_0^1\frac{\min^2\{1,\sqrt{\alpha}tx\}}{x^2}\ dx\lesssim\int_0^{\frac{1}{\sqrt{\alpha}t}}+\int_{\frac{1}{\sqrt{\alpha}t}}^1\frac{\min^2\{1,\sqrt{\alpha}tx\}}{x^2}\ dx\lesssim I_1+I_2.
\]We have\[
I_1\lesssim\sqrt{\alpha}t,\qquad I_2\lesssim\sqrt{\alpha}t.
\]As a result, we have\[
I^2\lesssim\nu^{-1}\sqrt{\alpha}
\]\hfill \qedsymbol

From the above estimate, we see that, if $u_0(0),\theta_0(0)\in L^{2}\cap L^\infty$, then $u_{1,0}$ satisfies the bound\[
||u_{1,0}||_{L^\infty}\lesssim\left(1+\alpha^{-\frac{1}{4}}\nu^{-\frac{1}{2}}\right)\left\|\left(u,\frac{\theta}{\alpha^\frac{1}{2}}\right)_0(0)\right\|_{L^\infty\cap L^{2}}.
\]That is, the transient growth in terms of $\nu$ is weakened and may even be suppressed by strong stratification.
\vspace{5pt}

\noindent\textbf{Suppression of the lift-up effect}\vspace{5pt}

On the other hand, since this lift-up is due to the singular behaviour at $\frac{\eta}{l}\rightarrow0$, it is possible to overcome this singular behaviour completely by imposing an additional assumption on the initial data. Here we present one possible method.\vspace{5pt}

The key observation is the following

\begin{lemma}\label{supofl1} Assume that $f:\mathbb{R}^2\rightarrow\mathbb{R}$ satisfies $\int_{\mathbb{R}}f(y,z)\ dy=0$, then\[
||\partial^{-1}_yf||_{L^\infty}\lesssim \left\|\langle\bm{\nabla}\rangle^{1+}f\right\|_{L^2}+\left\|\langle y\rangle^{\frac{3}{2}+}\langle\partial_z\rangle^{\frac{1}{2}+}f\right\|_{L^2}.
\]\end{lemma}

\noindent\textsc{Proof}. We have\begin{align*}
    |\partial^{-1}_yf|\leq\int_{\mathbb{R}^2}\frac{|\widehat{f}|}{|\eta|}\ d\eta dl=I_A+I_B.
\end{align*}Here $A:=\left\{(\eta,l)\in\mathbb{R}^2:|\eta|\leq 1\right\}$ and $B=A^c$. It is clear that \[
I_B\lesssim\int_{\mathbb{R}^2}\left|\widehat{\langle\bm{\nabla}\rangle^{1+}}f\right|\cdot\frac{1}{|1,\eta,l|^{1+}}\ d\eta dl\lesssim\left\|\langle\bm{\nabla}\rangle^{1+}f\right\|_{L^2}.
\]By the zero-$y$-average assumption, we have\[
\widehat{f}_{\eta,l}=\int_{\mathbb{R}}\widehat{f}_le^{-i\eta y}\ dy=\int_{\mathbb{R}}\widehat{f}_l\left(e^{-i \eta y}-1\right)\ dy.
\]Thus, we have\[
\left|\frac{\widehat{f}}{\eta}\right|\leq\int\left|\widehat{\left(yf\right)_l}\right|\ dy\lesssim\left\|\widehat{\left(y\langle y\rangle^{\frac{1}{2}+}f\right)_l}\right\|_{L^2_y}=: C\cdot\left\|g\right\|_{L^2_y}(l).
\]Consequently,\[
I_A\lesssim\int ||g||_{L^2_y}(l)\ dl\lesssim\int\frac{\langle l\rangle^{\frac{1}{2}+}||g||_{L^2_y}}{\langle l\rangle^{\frac{1}{2}+}}\ dl\lesssim\left(\int_{\mathbb{R}}\langle l\rangle^{1+}\int |g|^2(y,l)\ dydl\right)^\frac{1}{2}\lesssim\left\|\langle\partial_z\rangle^{\frac{1}{2}+}\langle y\rangle^{\frac{3}{2}+}f\right\|_{L^2_{y,z}}.
\]\hfill \qedsymbol\vspace{4pt}

 As a result, we see that $u_{1,0}$ satisfies the following estimate\begin{align*}
||u_{1,0}||_{L^\infty}\lesssim||u_{1,0}(0)||_{H^2}+\alpha^{-\frac{1}{2}}\left\|\left(u_2,\frac{\theta}{\alpha^\frac{1}{2}}\right)_0\right\|_{H^{2+}}+\alpha^{-\frac{1}{2}}\left\|\langle y\rangle^{\frac{3}{2}+}\langle\partial_z\rangle^{\frac{1}{2}+}\left(u_2,\frac{\theta}{\alpha^\frac{1}{2}}\right)_0\right\|_{L^2}.
\end{align*}In addition, one can show that if weaker spatial decay is assumed for $f$, then there can be some small growth of the solution. For instance, if we can only control $\left|\partial_y\right|^{-1/2}f$, the solution can grow like $t^{1/2}e^{-\nu t}\lesssim\nu^{-1/2}e^{-\nu t}$.\vspace{9pt}

\noindent\textbf{Lift-up on $\Omega_{\mathbb{T}}$}\vspace{5pt}

On $\Omega_{\mathbb{T}}$,  the frequency $l$ is discrete, and the dynamics are significantly different when $l=0$ or $l\neq0$, which correspond to the double zero and simple zero modes, respectively. We first consider the simple zero modes. Suppose that $f$ has zero $z$-average, we then observe that.\begin{align*}
     \left\|\partial_y^{-1} \partial_z e^{\nu t \Delta}\left(e^{\sqrt{\alpha}t\mathcal{R}_y}-I_d\right)f\right\|_{L^2(\mathbb{R}\times\mathbb{T})}^2&\approx\sum_{l\neq0}\int_\mathbb{R} \frac{l^2}{\eta^2}e^{-2\nu t(\eta^2+l^2)}\left|\widehat{f}\right|^2\left|e^{i\sqrt{\alpha}t\frac{\eta}{\sqrt{\eta^2+l^2}}}-1\right|^2\ d\eta\\
     &\lesssim\sum_{l\neq0}\int_\mathbb{R}\alpha t^2 e^{-2\nu t(\eta^2+l^2)}\left|\widehat{f}\right|^2\ d\eta\\
    &\lesssim\sum_{l\neq0}\int_\mathbb{R}\frac{\alpha e^{-c\nu t}}{\nu^2(\eta^2+l^2)^2}\left|\widehat{f}\right|^2\ d\eta\\
     &\lesssim\alpha \nu^{-2}e^{-c\nu t}||\Delta^{-1}_{y,z}f||_{L^2}^2.
\end{align*}We then see that $\overline{u}_1$ experiences lift-up, with the transient growth $||\overline{u}_1||_{L^2}\lesssim\nu^{-1}$. Now we show that, as in the $\Omega_{\mathbb{R}}$ case, the $L^p$ transient growth is weaker when $p>2$. In particular, we have the following result\begin{lemma}\label{suppofl3} Suppose that $f$  has zero-$z$ average. Then the following estimate holds\[
   \left\|\partial_y^{-1} \partial_z e^{\nu t \Delta}\left(e^{\sqrt{\alpha}t\mathcal{R}_y}-I_d\right)f\right\|_{L^\infty(\mathbb{R}\times\mathbb{T})}\lesssim  e^{-c\nu t}\nu^{-\frac{1}{2}}\alpha^{\frac{1}{4}}||f||_{L^{2}(\mathbb{R}\times\mathbb{T})}
\]\end{lemma}

\noindent\textit{Proof}. See Appendix \ref{coro345}.\vspace{3pt}

By contrast, the lift-up of $\overline{u}_1$ in the classical Navier Stokes is due to the operator $\mathcal{L}_{\text{NSE}}=te^{\nu t\Delta}$. On $\mathbb{R}\times\mathbb{T}$, this operator has the norm $||\mathcal{L}_{\text{NSE}}||_{L^2_0\rightarrow L^2_0}\lesssim\nu^{-1}e^{-c\nu t}$ for any $p\geq2$. By interpolation, we see that for any $p>2$, $||\mathcal{L}_{\text{Bou}}||_{L^2_0\rightarrow L^p_0}\lesssim\nu^{-\frac{2+p}{2p}}\alpha^{-\frac{p-2}{4p}}e^{-c\nu t}$, indicating that this lift-up can be suppressed when the stratification is strong enough.\vspace{4pt}

Similarly, the lift-up effect can also be suppressed on $\Omega_{\mathbb{T}}$. We prove the following $L^2$ and $L^\infty$ bounds.

\begin{lemma}\label{suppofl2} Let $f$ be a function with zero $y$- and $z$- averages, the following bounds hold\begin{align*}
\left\|\partial_y^{-1}\partial_z e^{\nu \Delta t}\left(e^{\sqrt{\alpha}t\mathcal{R}_y}-I_d\right)f\right\|_{L^2}&\lesssim e^{-\nu t}\left(\left\|\langle y\rangle^{\frac{3}{2}+\mu}\partial_zf\right\|_{L^2}+e^{-\nu t}||\partial_z f||_{L^2}\right),\\
\left\|\partial_y^{-1}\partial_z e^{\nu \Delta t}\left(e^{\sqrt{\alpha}t\mathcal{R}_y}-I_d\right)f\right\|_{L^\infty}&\lesssim e^{-\nu t}\left(\left\|\langle y\rangle^{\frac{3}{2}+\mu}|\partial_z|^{\frac{3}{2}+\mu}f\right\|_{L^2}+e^{-\nu t}\left\||\partial_z|^{\frac{3}{2}+\mu}\langle \partial_y\rangle^{\frac{1}{2}+\mu}f\right\|_{L^2}\right)
\end{align*}for any arbitrarily small $\mu>0$.\end{lemma}

\noindent\textsc{Proof}. For simplicity, it suffices to prove the estimate without $\partial_z$. For the $L^2$ bound, we expand it again\begin{align*}
\left\|\partial_y^{-1}e^{\nu\Delta t}\left(e^{\sqrt{\alpha} t\mathcal{R}_y}-I_d\right)f\right\|_{L^2}^2&\lesssim\sum_{l\neq0}\int_{\mathbb{R}}\frac{|\widehat{f}|^2}{\eta^2}e^{-2\nu t(\eta^2+l^2)}\ d\eta\\
&\lesssim e^{-2\nu t}\sum_{l\neq0}\int_{\mathbb{R}}\frac{|\widehat{f}|^2}{\eta^2}e^{-2\nu t\eta^2}\\
&\lesssim e^{-2\nu t}\sum_{l\neq0}\underbrace{\int_{[-1,1]}}_{I_1}+\underbrace{\int_{[-1,1]^c}}_{I_2}\frac{|\widehat{f}|^2}{\eta^2}e^{-2\nu t\eta^2}
\end{align*}Define $\widehat{f}_\eta$ and $\widehat{f}_l$ by\begin{align*}
&\widehat{f}_l=\int_{\mathbb{T}}f(y,z)e^{-ilz} \ dz,\qquad\quad\widehat{f}_\eta=\int_{\mathbb{R}}f(y,z)e^{-i\eta y} \ dy,\\&\mathcal{F}_z\left(\widehat{f}_\eta\right)=\widehat{f}(\eta,l),\qquad \widehat{f}_\eta(z)=\sum_{l\neq0}\widehat{f}(\eta,l)e^{ilz}.
\end{align*}The zero $y$- average assumption implies\[
\widehat{f}(\eta,l)=\int_{\mathbb{R}} \widehat{f}_l e^{-i\eta y}-\widehat{f}_l \ dy.
\]Thus,\[
\frac{|\widehat{f}|}{|\eta|}\lesssim ||y\widehat{f}_l||_{L^1_y}\lesssim\left\|\langle y\rangle^{\frac{3}{2}+\mu}\widehat{f}_l\right\|_{L^2_y}.
\]As a result, we have\begin{align*}
    \sum_{l\neq0}I_1\lesssim\sum_{l\neq 0}\left\|\langle y\rangle^{\frac{3}{2}+\mu}\widehat{f}_l\right\|_{L^2_y}^2=\left\|\langle y\rangle^{\frac{3}{2}+\mu}f\right\|_{L^2}^2.
\end{align*}$I_2$ can be similarly bounded\[
\sum_{l\neq0}I_2\lesssim e^{-2\nu t}\sum_{l\neq0}\int_{\mathbb{R}}|\widehat{f}|^2\lesssim e^{-2\nu t}||f||_{L^2}^2.
\]\vspace{8pt}

For the $L^\infty$ bound, we have\begin{align*}
    \left\|\partial_y^{-1}e^{\nu\Delta t}\left(e^{\sqrt{\alpha}t\mathcal{R}_y}-I_d\right)f\right\|_{L^\infty}\lesssim e^{-\nu t}\sum_{l\neq0}\int_{\mathbb{R}}\frac{|\widehat{f}|}{|\eta|} e^{-\nu \eta^2 t}\ d\eta\lesssim e^{-\nu t}\sum_{l\neq0}\underbrace{\int_{[-1,1]}}_{I_1}+\underbrace{\int_{[-1,1]^c}}_{I_2}\frac{|\widehat{f}|}{|\eta|}e^{-\nu \eta^2 t}\ d\eta.
\end{align*}The method is similar to the one above. For $I_1$, we have\begin{align*}
  \sum_{l\neq0}I_1\lesssim\sum_{l\neq 0}\left\|\langle y\rangle^{\frac{3}{2}+\mu}\widehat{f}_l\right\|_{L^2_y}\lesssim\sum_{l\neq 0} |l|^{-\frac{1}{2}-\mu}\left\|\langle y\rangle^{\frac{3}{2}+\mu}\widehat{|\partial_z|^{\frac{1}{2}+\mu}f}_l\right\|_{L^2_y}\lesssim\left\|\langle y\rangle^{\frac{3}{2}+\mu}|\partial_z|^{\frac{1}{2}+\mu}f\right\|_{L^2}.
\end{align*}For $I_2$, we see that\begin{align*}
\sum_{l\neq0}I_2&\lesssim e^{-\nu t}\sum_{l\neq0}\int|\widehat{f}|\lesssim e^{-\nu t}\sum_{l\neq0}\int\frac{|\widehat{\langle \partial_y\rangle^{\frac{1}{2}+\mu} f}|}{\langle\eta\rangle^{\frac{1}{2}+\mu}}\lesssim e^{-\nu t}\sum_{l\neq0}\left\|\widehat{\langle \partial_y\rangle^{\frac{1}{2}+\mu} f}\right\|_{L^2_\eta}\\
&\lesssim  e^{-\nu t}\sum_{l\neq0}|l|^{-\frac{1}{2}-\mu}|l|^{\frac{1}{2}+\mu}\left\|\widehat{\langle \partial_y\rangle^{\frac{1}{2}+\mu} f}\right\|_{L^2_\eta}\lesssim e^{-\nu t}\left\||\partial_z|^{\frac{1}{2}+\mu}\langle \partial_y\rangle^{\frac{1}{2}+\mu}f\right\|_{L^2}.\end{align*}\hfill \qedsymbol\vspace{3pt}

\noindent\textbf{Presence of Wave-shear coupled lift-up effect in related models.}

We discuss some examples where this wave-shear coupled lift-up effect can occur. For the 3D MHD system linearized around the Couette flow $U^*=(y,0,0)$ on $\Omega_{\mathbb{T}}$ with a constant background magnetic field $B^*=\alpha(\sigma,0,1)$ \cite{liss2020sobolev}, assuming zero viscosity and resistivity, and letting $z^{\pm}:=e^{\alpha t(\sigma\partial_x+\partial_z)}(u\mp b)$, $\widehat{z}_0^{\pm,1}$ are solved as follows.\[
\widehat{z}_0^{\pm,1}=\widehat{z}_0^{\pm,1}(0)-\frac{e^{\pm i\alpha lt}}{\alpha  l}\sin(\alpha lt)\widehat{z}_0^{\mp,2}(0).
    \]Since $z\in\mathbb{T}$, $|l|\geq1$ for simple zero modes and the solution is bounded. If the perturbation is posed on $\Omega_{\mathbb{R}}$, $l$ is continuous. As $l\to 0$, we have $l^{-1}\sin(al t)\sim at$, that is, transient growth still occurs in the zero modes. However, it is not the classical lift-up mechanism, since it is generated by the singular operator $|\partial_z|^{-1}\sin(\alpha t|\partial_z|)$. This lift-up effect is therefore similar to the one identified in the Boussinesq case here through \eqref{Bouliftupope}. In particular, it induces weaker transient growth in the $L^p$ norm for $p>2$ and can be suppressed either by a sufficiently strong background magnetic field or under suitable assumptions on the initial data. Moreover, such type lift-up can also appear in the Boussinesq-MHD system and can be suppressed in the rotating Boussinesq system in the presence of rotation, see \cite{yao2026pvwavedecomposition} for more details. This suggests that operators such as \eqref{Bouliftupope} arise more generally in wave-shear coupled systems, particularly when continuous low frequencies are present. \vspace{8pt}

Having understood the simple zero modes, we now turn to the double zero modes. The linearized system reads
\[\left\{
\begin{aligned}
&\partial_t\widetilde{u}_{1}=-\widetilde{u}_{2}+\nu\partial^2_y\widetilde{u}_1,\\
    &\partial_t\widetilde{u}_{2}=-\partial_y\widetilde{p}+\nu\partial^2_y\widetilde{u}_2,\\
    &\partial_t \widetilde{u}_{3}=\widetilde{\theta}+\nu\partial^2_y\widetilde{u}_3,\\
    &\partial_t\widetilde{\theta}=-\alpha \widetilde{u}_{3}+\nu\partial^2_y\widetilde{\theta}.
\end{aligned}
\right.\]The incompressibility condition now reduces to\[
\partial_y\widetilde{u}_2=0,\qquad\Rightarrow\qquad \widetilde{u}_2(t,y)=\widetilde{u}_2(t).
\]As we require $\widetilde{u}_2\in L^2_y$, the only possibility is that $\widetilde{u}_2=0$. As a result, the explicit solution is given by\[
\left\{
\begin{aligned}
&\widetilde{u}_{1}=e^{\nu t\partial^2_y}\widetilde{u}_1(0,y),\\
    &\widetilde{u}_{2}=0,\\
    &\widetilde{u}_{3}=e^{\nu t\partial^2_y}\left(\widetilde{u}_3(0,y)\cos(\sqrt{\alpha} t)+\frac{\widetilde{\theta}(0,y)}{\sqrt{\alpha}}\sin(\sqrt{\alpha}t)\right),\\
    &\widetilde{\theta}=e^{\nu t\partial^2_y}\left(-\sqrt{\alpha}\widetilde{u}_3(0,y)\sin(\sqrt{\alpha} t)+\widetilde{\theta}(0,y)\cos(\sqrt{\alpha}t)\right).
\end{aligned}
\right.
\]The explicit formula shows that the double zero modes do not experience the lift-up effect.

\subsection{\textbf{Dispersive estimates for the zero modes}}\label{dis1}
In the preceding section, we solved the zero-mode system (\ref{eq:B9}). As is clear from the solution, $u_{2,0}$, $u_{3,0}$, and $\theta_0$ (or the simple zero modes on $\Omega_{\mathbb{T}}$) are oscillatory and therefore enjoy dispersive decay in physical space, while the leading-order dynamics of $u_{1,0}$ (or $\overline{u}_1$ when $\mathbb{D}=\mathbb{T}$) exhibit no decay in the absence of viscosity. Such decay estimates have been studied in many previous works; see, for example \cite{li2025transitionthresholdnavierstokescoriolishigh,coti2024stability}.  We reproduce here the proof of dispersive estimate when $\mathbb{D}=\mathbb{T}$ for the reader's convenience and defer the proof for $\mathbb{D}=\mathbb{R}$ to the Appendix \ref{APPENDIS1}.\vspace{6pt}

\begin{lemma}\label{firstdis} Suppose that $f$ has zero $z$-average, then the following dispersive estimate holds\[
\left\|e^{t\partial_y|\bm{\nabla}_{y,z}|^{-1}}f\right\|_{L^\infty(\mathbb{R}\times\mathbb{T})}\lesssim\frac{1}{\langle t\rangle^\frac{1}{3}}||f||_{W^{3,1}(\mathbb{R}\times\mathbb{T})}
\]\end{lemma}

\noindent\textsc{Proof}. Boundedness is clear. Therefore, we only consider the case $t\geq 1$. First of all, we have \begin{gather*}
f(y,z)=\frac{1}{2\pi}\sum_{l\neq0}\widehat{f}_l(y,l)e^{ilz},\qquad \hat{f}_l(y,l)=\int_{\mathbb{T}}f(y,z)e^{-ilz}\ dz,\\\widehat{f}(\eta,l)=\int_{\mathbb{R}} \widehat{f}_l(y,l)e^{-i\eta y}\ dy.
\end{gather*}Define the operator $\mathcal{L}:=\partial_y|\bm{\nabla}_{y,z}|^{-1}$ and the following Littlewood-Paley decomposition\begin{flalign*}
\phi\in C^\infty(\mathbb{R}),\qquad \phi(\eta)=\left\{
\begin{aligned}
& 1&|\eta|\leq 1\\
& 0&|\eta|>\frac{4}{3}
\end{aligned}
\right.,\qquad \varphi(\eta)=\phi\left(\frac{\eta}{2}\right)-\phi(\eta),\qquad \varphi_j(\eta)=\varphi(2^{-j}\eta).
\end{flalign*}Denote the projection operator as $\mathbb{P}_j\psi:=\mathbb{F}^{-1}\left(\varphi_j(\eta)\widehat{\psi}\right)$. Thus, we have\begin{align}
\left\|e^{t\mathcal{L}}\widehat{f}_l\right\|_{L^\infty_y}&\leq\sum_{j\in\mathbb{Z}}\left\|\mathcal{F}^{-1}_y\left(e^{it\frac{\eta}{\sqrt{\eta^2+l^2}}}\varphi_j\varphi_j^1\widehat{f}(\eta,l)\right)\right\|_{L^\infty_y}\nonumber\\
    &\lesssim\sum_{j\in\mathbb{Z}}\left\|\mathcal{F}_y^{-1}\left(e^{it\frac{\eta}{\sqrt{\eta^2+l^2}}}\varphi_j\right)\right\|_{L^\infty}\left\|\mathbb{P}^1_j\widehat{f}_l(\cdot,l)\right\|_{L^1_y}.\label{asodyasodasoudsd}
\end{align}The functions $\varphi^1_j$ and $\varphi_j$ are chosen to have similar supports such that $\varphi_j=\varphi_j\varphi_j^1$. Consider the first term in (\ref{asodyasodasoudsd}), i.e.\[
\mathcal{F}_y^{-1}\left(e^{t\hat{\mathcal{L}}}\varphi_j\right)=\int_{\mathbb{R}}e^{it\frac{\eta}{\sqrt{\eta^2+l^2}}}\varphi_j(\eta) e^{i\eta y}\ d\eta
\]It remains to estimate the integral\[
I=\int_{\mathbb{R}} e^{i\eta y+i\frac{t\eta}{\sqrt{\eta^2+l^2}}}\varphi_j(\eta)\ d\eta,
\]Using the change of variable $\eta=l\zeta$ and assuming without loss of generality that $l>0$, we have\[
I=l\int_{\mathbb{R}}e^{it\Phi(\zeta)}\varphi\left(2^{-j}l\zeta\right)\ d\zeta,\qquad \Phi=\frac{\zeta}{\sqrt{\zeta^2+1}}+\Omega\zeta,\qquad \Omega=\frac{ly}{t}.
\]The derivatives of $\Phi$ are given by\[
\frac{d\Phi}{d\zeta}=\frac{1}{\left(\zeta^2+1\right)^\frac{3}{2}}+\Omega,\qquad\frac{d^2\Phi}{d\zeta}=\frac{-3\zeta}{\left(\zeta^2+1\right)^\frac{5}{2}}.
\]Notice that, the zero of $\Phi''$ is at $\zeta=0$. This can lead to some difficulty when we consider the low frequency with $j<0$. Instead, we split the frequencies into the following parts\begin{align}
\left\|e^{t\mathcal{L}}\widehat{f}_l\right\|_{L^\infty_y}&\leq\sum_{j\in\mathbb{Z}}\left\|\mathcal{F}^{-1}_y\left(e^{it\frac{\eta}{\sqrt{\eta^2+l^2}}}\varphi_j\varphi_j^1\widehat{f}(\eta,l)\right)\right\|_{L^\infty_y}\nonumber\\
    &\lesssim\sum_{j\geq j_0}\left\|\mathcal{F}_y^{-1}\left(e^{it\frac{\eta}{\sqrt{\eta^2+l^2}}}\varphi_j\right)\right\|_{L^\infty}\left\|\mathbb{P}^1_j\widehat{f}_l(\cdot,l)\right\|_{L^1_y}+\left\|\mathcal{F}_y^{-1}\left(e^{it\frac{\eta}{\sqrt{\eta^2+l^2}}}\phi_{j_0}^1(\eta)\right)\right\|_{L^\infty}||\widehat{f}_l||_{L^1_y}.
\end{align}Here $\phi_{j_0}^1$ has the similar support as $\phi_{j_0}:=\sum_{j<j_0}\varphi_j$ for some universal integer $j_0$ (e.g. $j_0=-1$). The low frequency part ($\phi^1_{j_0}$) has finite support, say $[-C,C]$. We consider the following integral,\[
\overline{I}:=l\int e^{it\Phi(\zeta)}\phi^1_{j_0}(l\zeta)\ d\zeta.
\]We split the integral $\overline{I}$ into $\overline{I}_1+\overline{I}_2$ such that \[
\overline{I}_1=l\int_{B_\epsilon(0)}e^{it\Phi(\zeta)}\phi^1_{j_0}(l\zeta)\ d\zeta.
\]Simple size estimate gives $|\overline{I}_1|\lesssim |l|\epsilon$. Since $||\phi^1_{j_0}||_{W^{1,1}}\lesssim 1$, by the van der Corput lemma (see \cite{stein1993harmonic}) and the fact that $|\Phi''|\gtrsim \epsilon$, we have\[
|\overline{I}_2|\lesssim|l|(\epsilon t)^{-1/2}.
\]Pick $\epsilon=t^{-1/3}$, we see that\[
\left\|\mathcal{F}_y^{-1}\left(e^{it\frac{\eta}{\sqrt{\eta^2+l^2}}}\phi_{j_0}\right)\right\|_{L^\infty}\left\|\widehat{f}\right\|_{L^1_y}\lesssim|l|t^{-1/3}||\widehat{f}_l||_{L^1_y}.
\]For the high frequency part, via a similar argument in \cite{coti2024stability}, we obtain\[
|I|\lesssim |l|t^{-1/3}\bm{1}_{2^j\leq |l|}+2^{2j}|l|^{-1}t^{-1/2}\bm{1}_{2^j>|l|},
\]the latter bound comes from the fact that $|\Phi''|\gtrsim |\zeta|^{-4}$ for large $\zeta$. We then obtain the following bound\begin{align*}
\left\|e^{t\mathcal{L}}\widehat{f}_l\right\|_{L^\infty_y}&\lesssim \sum_{j\geq j_0,2^j\leq |l|}\frac{|l|}{t^\frac{1}{3}}\left\|\mathbb{P}^1_j\widehat{f}_l(\cdot,l)\right\|_{L^1_y}+\sum_{j\geq j_0,2^j>|l|}\frac{2^{2j}}{|l|t^\frac{1}{2}}\left\|\mathbb{P}^1_j\widehat{f}_l(\cdot,l)\right\|_{L^1_y}+\frac{|l|}{t^\frac{1}{3}}||\widehat{f}_l||_{L^1_y}\\&\lesssim\frac{|l|}{t^\frac{1}{3}}||\widehat{f}_l||_{W^{\varepsilon,1}_y}+\frac{1}{|l|t^\frac{1}{2}}||\widehat{f}_l||_{W^{2+\varepsilon,1}_y}.
\end{align*}Summing over $l$, we have\[
\sum_{l\neq0}\frac{|l|}{t^\frac{1}{3}}||\widehat{f}_l||_{W^{\varepsilon,1}_y}+\frac{1}{|l|t^\frac{1}{2}}||\widehat{f}_l||_{W^{2+\varepsilon,1}_y}\lesssim\frac{1}{t^\frac{1}{3}}\sum_{l\neq0}|l|^{2+\varepsilon}|l|^{-1-\varepsilon}||\widehat{f}_l||_{W^{\varepsilon,1}_y}+|l|^{\varepsilon}|l|^{-1-\varepsilon}||\widehat{f}_l||_{W^{2+\varepsilon,1}_y}\lesssim\frac{1}{t^\frac{1}{3}}||f||_{W^{3,1}}.
\]\hfill \qedsymbol\vspace{3pt}

Based on this estimate, the $L^\infty$ decay of the zero modes of $u_{2,3}$ and $\theta$ follows readily:\begin{align*}
    \left\|\left(u_2,u_3,\frac{\theta}{\alpha^\frac{1}{2}}\right)_0\right\|_{L^\infty}\lesssim\frac{e^{-\nu t}}{\langle\alpha^\frac{1}{2}t\rangle^\frac{1}{3}}  \left\|\left(u_2,u_3,\frac{\theta}{\alpha^\frac{1}{2}}\right)_0(0)\right\|_{W^{3,1}}.
\end{align*}With minor modifications to the proof, one can verify that the following estimate also holds.

\begin{corollary}\label{fdis} For any $a\in\mathbb{Z}^+$, there holds the following estimate for zero-z-average function $f$,\[
\left\||\partial_z|^a|\bm{\nabla}_{y,z}|^{-a}e^{t\mathcal{L}}f\right\|_{L^\infty}+\left\||\partial_y|^a|\bm{\nabla}_{y,z}|^{-a}e^{t\mathcal{L}}f\right\|_{L^\infty}\lesssim\frac{1}{\langle t\rangle^\frac{1}{3}}||f||_{W^{3,1}}.
\]  
\end{corollary}

\subsection{\textbf{Decay of non-zero modes at a non-unit Prandtl number}}\label{pr}

With the special case $\nu=\kappa$ being understood, we try to derive the corresponding estimates in the more general setting where $\nu\neq \kappa$.

The full linearized system now reads
\begin{flalign*}
\left\{
\begin{aligned}
&\partial_t\theta+y\partial_x\theta=-\alpha u_3+\kappa\Delta\theta,\\
    &\partial_t u_3+y\partial_x u_3=2\partial_{xz}\Delta_H^{-1}\Delta^{-1}\left(\partial_x\left(\frac{K+\partial_z\theta}{\alpha}\right)-\partial_{zy}u_3\right)+\Delta_H\Delta^{-1}\theta+\nu\Delta u_3,\\
    &\partial_tK+y\partial_x K=(\nu-\kappa)\Delta\partial_z\theta+\nu\Delta K.
\end{aligned}
\right.
\end{flalign*}Unlike the case $\nu=\kappa$, the presence of $(\nu-\kappa)\partial_z\Delta\theta$ destroys the conservation of $\mathcal{K}$. Taking the Fourier transform in the moving frame, we see that\begin{flalign*}
\left\{
\begin{aligned}
&\partial_t\widehat{\Theta}=-\alpha \widehat{U}_3-\kappa p\widehat{\Theta},\\
    &\partial_t \widehat{U}_3=-2\frac{kl}{pf}\left(\frac{ik\widehat{\mathcal{K}}-kl\widehat{\Theta}}{\alpha}+l(\eta-kt)\widehat{U}_3\right)+\frac{f}{p}\widehat{\Theta}-\nu p\widehat{U}_3,\\
    &\partial_t\widehat{\mathcal{K}}=-il(\nu-\kappa)p\widehat{\Theta}-\nu p\widehat{\mathcal{K}}.
\end{aligned}
\right.
\end{flalign*}The key step is to set the following change of variables\[
\widehat{\mathcal{K}}_{\neq}=i\langle l\rangle\alpha cC,\qquad \widehat{\Theta}_{\neq}=\sqrt{\alpha}aA,\qquad \widehat{U}_{3,\neq}=bB,
\]where $a(t),b(t),c(t)$ are the functions of $t,k,\eta,l$ to be determined. The new system then reads\begin{flalign*}
\left\{
\begin{aligned}
&\partial_tA=-\sqrt{\alpha}\frac{b}{a} B-\kappa pA-\frac{a'}{a}A,\\
    &\partial_t B=-2\frac{kl}{bpf}\left(-k\langle l\rangle cC-\frac{kl}{\sqrt{\alpha}}aA+bl(\eta-kt)B\right)+\frac{\sqrt{\alpha}af}{bp}A-\nu pB-\frac{b'}{b}B,\\
    &\partial_tC=-\frac{l}{\langle l\rangle}\frac{\nu-\kappa}{\sqrt{\alpha}}p\frac{a}{c}A-\nu pC-\frac{c'}{c}C.
\end{aligned}
\right.
\end{flalign*}Now we pick\[
\frac{b}{a}=\frac{\sqrt{z}}{\sqrt{\alpha}\sqrt{pf}},\qquad \frac{a'}{a}=\lambda\left(\frac{f'l^2}{pf}\right)-\frac{b'}{b},
\]where $z=2k^2p+\alpha f^2$. Now we consider the energy functional\[
\mathbb{E}=\frac{1}{2}\left[|A|^2+|B|^2+s(AB)+|C|^2\right], 
\]where $s$ was defined in the previous section. We see that\begin{align}
    \partial_t\mathbb{E}=&-\kappa p|A|^2-\nu p|C|^2-\nu p|B|^2-\frac{c'}{c}|C|^2-\frac{l}{\langle l\rangle}\frac{\nu-\kappa}{\sqrt{\alpha}}\frac{a}{c}p AC-\frac{s(\nu+\kappa)}{2}p(AB)\label{g1}\\
    +&(1-\lambda)\frac{f'l^2}{pf}|B|^2-\frac{2}{\sqrt{\alpha}}\frac{k^2}{p}\frac{a}{b}AB+\frac{2ck^2l\langle l\rangle}{bpf}BC+\frac{s'}{2}AB+\frac{s(1-\lambda)}{2}\frac{f'l^2}{pf}AB\\
    +&\frac{sck^2l\langle l\rangle}{bpf}AC-\frac{s}{\sqrt{\alpha}}\frac{k^2}{p}\frac{a}{b}|A|^2.\nonumber
\end{align}Set $a=\Omega c$ for some constant $\Omega>0$. Then the first line, excluding the $\frac{c'}{c}C$ term, divided by $p$, can be bounded by \begin{align*}
    &-\kappa |A|^2-\nu |C|^2-\nu |B|^2+\zeta(\nu+\kappa)|AB|+\frac{\Omega}{\sqrt{\alpha}}|\nu-\kappa||AC|\\
    \leq&-\left(\nu-\frac{\lambda_1}{2\xi_1}\right)|C|^2-\left(\nu-\frac{\lambda_2}{2\xi_2}\right)|B|^2-\left(\kappa-\frac{\lambda_1\xi_1}{2}-\frac{\lambda_2\xi_2}{2}\right)|A|^2.
\end{align*}To prove the enhanced dissipation, the parameters should be chosen to satisfy\begin{align*}
    \nu-\frac{\lambda_1}{2\xi_1}\geq c_0\mu,\qquad     \nu-\frac{\lambda_2}{2\xi_2}\geq c_0\mu,\qquad \kappa-\frac{\lambda_1\xi_1}{2}-\frac{\lambda_2\xi_2}{2}\geq c_0\mu,\qquad c_0\in(0,1),
\end{align*}where $\lambda_1=\frac{\Omega|\nu-\kappa|}{\sqrt{\alpha}}$, $\lambda_2=\zeta(\nu+\kappa)$ and $\mu=\min\{\nu,\kappa\}$. Rearranging, this would require\begin{equation}
\lambda_1^2+\lambda_2^2\leq4(\nu-c_0\mu)(\kappa-c_0\mu).\label{AMGMbound}
\end{equation}To ensure $\Omega^2>0$, we need\[
4(\nu-c_0\mu)(\kappa-c_0\mu)>\lambda_2^2.
\]We solve this inequality in terms of the Prandtl number, i.e. $\text{Pr}=\frac{\nu}{\kappa}$.\begin{itemize}
    \item When $\text{Pr}\geq1$, i.e. $\mu=\kappa$, then the above yields\begin{align*}
    &\text{Pr}>\max\left\{1,\frac{2}{\zeta^2}(1-c_0)-1-\frac{2}{\zeta}\sqrt{\frac{(1-c_0)^2}{\zeta^2}-(1-c_0^2)}\right\},\\
  & \text{Pr}<\frac{2}{\zeta^2}(1-c_0)-1+\frac{2}{\zeta}\sqrt{\frac{(1-c_0)^2}{\zeta^2}-(1-c_0^2)}.
    \end{align*}In the case $\text{Pr}\leq 1$, we obtain the following by symmetry\begin{align*}
    &\frac{1}{\text{Pr}}>\max\left\{1,\frac{2}{\zeta^2}(1-c_0)-1-\frac{2}{\zeta}\sqrt{\frac{(1-c_0)^2}{\zeta^2}-\left(1-c_0^2\right)}\right\},\\
  & \frac{1}{\text{Pr}}<\frac{2}{\zeta^2}(1-c_0)-1+\frac{2}{\zeta}\sqrt{\frac{(1-c_0)^2}{\zeta^2}-\left(1-c_0^2\right)}.
    \end{align*}
\end{itemize}The upper bound for $\text{Pr}$ is then maximized when $\zeta$ is minimized (and the lower bound is minimized when $\zeta$ is minimized). By Lemma \ref{upperboundofs}, we see that the smallest $\zeta$ is achieved when $\lambda=0$ with $\zeta=\frac{1}{2\sqrt{2}}$. Therefore, we shall take $\lambda=0$ from now on. In this case, the possible pairs of $(c_0,\text{Pr})$ are plotted in Figure \ref{fig2}. As $c_0\rightarrow0$, the upper and lower bounds for $\text{Pr}$ tend to be $15+4\sqrt{14}$ and $15-4\sqrt{14}$, respectively. Furthermore, to ensure that the square root is real, it is necessary to have $c_0<\frac{7}{9}$.\vspace{2pt}

We now compute the further restriction due to (\ref{AMGMbound}), which is the following\[
\Omega^2\leq\frac{32\left(\text{Pr}-c_0\min\left\{1,\text{Pr}\right\}\right)\left(1-c_0\min\left\{1,\text{Pr}\right\}\right)-\left(\text{Pr}+1\right)^2}{8(\text{Pr}-1)^2}\cdot\alpha.
\]

\begin{figure}[H]
    \centering
    \begin{minipage}[b]{0.65\textwidth}
        \centering
        \includegraphics[width=\textwidth]{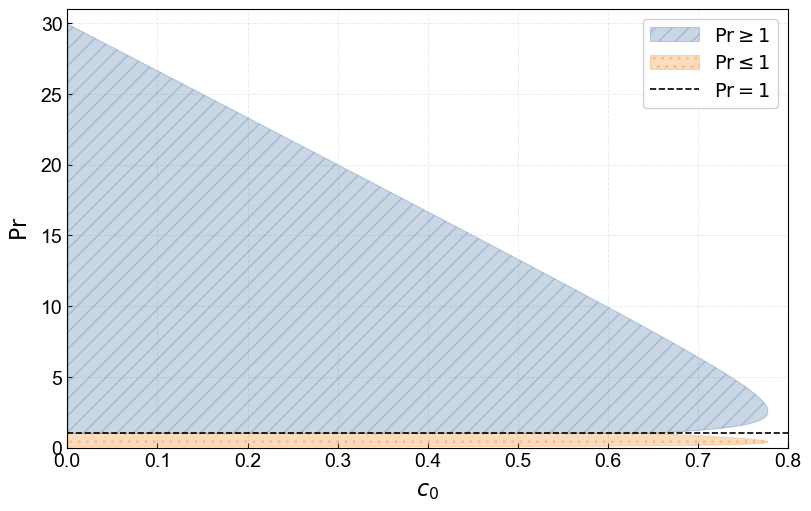}
        
    \end{minipage}
     \caption{Feasible region for $(c_0,\text{Pr})$ when $\lambda=0$}\label{fig2}
\end{figure}We can now proceed to derive the further energy estimates. Notice that\[
\left|\frac{2ck^2l\langle l\rangle}{bpf}BC+\frac{sck^2l\langle l\rangle}{bpf}AC\right|\leq\frac{\sqrt{2}}{\Omega}\frac{a}{b}\frac{k^2l\langle l\rangle}{pf}\left(\sqrt{2}|BC|+|AC|\right)\leq\frac{c_1}{\Omega}\frac{a}{b}\frac{k^2l\langle l\rangle}{pf} \mathbb{E},
\]where $c_1>0$ is a constant. We have\[
\frac{a}{\sqrt{\alpha}b}\frac{k^2l\langle l\rangle}{pf}=\frac{k^2|l|\langle l\rangle}{\sqrt{zpf}}.
\]It is non-integrable and it can lead to exponential growth in terms of $l$. However, the crucial observation is that the growth in terms of $l$ can be treated as a loss of regularity.

\begin{lemma}
For any $k\in\mathbb{Z}\backslash\{0\}$, the  following estimate holds\[
e^{\int_0^t\frac{k^2|l|\langle l\rangle}{\sqrt{zpf}}\ ds}\lesssim\langle l\rangle^{\sqrt{2}}.
\]
\end{lemma}

\noindent\textsc{Proof}. The bound is trivial when $l=0$, we therefore assume that it is not zero. We start by observing that\[
\frac{k^2|l|\langle l\rangle}{\sqrt{zpf}}\leq\frac{1}{\sqrt{2}}\frac{|k||l|\langle l\rangle}{p\sqrt{f}}.
\]Thus the integral can be bounded by\[
I=\int_{0}^t\frac{k^2|l|\langle l\rangle}{\sqrt{zpf}}\ ds\leq\int_{\mathbb{R}}\frac{1}{\sqrt{2}}\frac{|k||l|\langle l\rangle}{\sqrt{k^2+(\eta-kt)^2}(k^2+(\eta-kt)^2+l^2)}\ dt.
\]Using the change of variable $\eta-kt=kx$, we have\begin{align*}
I&\leq\frac{1}{\sqrt{2}}\int_{\mathbb{R}}\frac{|l|\langle l\rangle}{\sqrt{1+x^2}}\cdot\frac{1}{(1+x^2)+l^2}\ dx\\
&=\sqrt{2}\int_{0}^{\delta|l|}+\int_{\delta|l|}^\infty\frac{|l|\langle l\rangle}{1+l^2+x^2}\cdot\frac{1}{\sqrt{1+x^2}}\ dx\\
&:=\sqrt{2}(I_1+I_2).
\end{align*}For $I_2$, we bound it in the following way\[
I_2\leq\int_{\delta|l|}^\infty\frac{|l|\langle l\rangle}{x^3}\ dx=\frac{\langle l\rangle}{2\delta^2|l|}.
\]For $I_1$, we proceed in the following way\[
I_1\leq\int_0^{\delta|l|}\frac{1}{\sqrt{1+x^2}}\ dx=\ln(\delta|l|+\sqrt{1+\delta^2l^2}).
\]Setting $\delta^2=\frac{\langle l\rangle}{|l|}$ yields the desired estimate.\hfill \qedsymbol\vspace{5pt}

Therefore, we see that\[
e^{\int\frac{c_1}{\Omega}\frac{a}{b}\frac{k^2l\langle l\rangle}{pf}}=e^{\frac{c_1\sqrt{\alpha}}{\Omega}\int\frac{k^2 l\langle l\rangle}{\sqrt{zpf}}}\lesssim \langle l\rangle^{c_1\sqrt{2}\frac{\sqrt{\alpha}}{\Omega}}.
\]Lastly, we introduce the same Fourier multiplier $m$ (\ref{multiplier}), with different $\varrho_1$ and $\varrho_2$, to be determined below. Recall that the energy inequality (\ref{g1}) now contains a new term $-\frac{c'}{c}|C|^2.$ Indeed, this term can be split into two parts:\begin{align*}
    \frac{c'}{c}=-\frac{f'}{4}\cdot\frac{2\alpha l^2f^2-l^2z}{zpf}+\frac{f'}{2}\cdot\frac{k^2fp}{fpz}=I+II.
\end{align*}Since \[
|II|\leq\frac{|k|^3}{\alpha f^\frac{3}{2}},
\]we see that $II$ is uniformly integrable. For $I$, we have\[
|I|=\left|\frac{f'l^2}{4fp}-\frac{f'}{2}\cdot\frac{\alpha l ^2f}{zp}\right|\leq\frac{3|f'|l^2}{4pf}.
\]As a result, for $\mathbb{F}=m^2\mathbb{E}$ where $m$ was defined in  \eqref{multiplier}, we have the bound\begin{align*}
    \partial_t\mathbb{F}\leq-c_0\mu p\mathbb{F}+\frac{m'}{m}|mC|^2+\frac{3|f'|l^2}{4pf}|mC|^2+C_1\left(\frac{1}{\sqrt{\alpha}}\frac{k^2}{p}\frac{a}{b}+|s'|+\frac{|k|^3}{\alpha f^\frac{3}{2}}\right)\mathbb{F}+C_2\frac{\sqrt{\alpha}}{\Omega}\frac{k^2\langle l^2\rangle}{\sqrt{zpf}}\mathbb{F},
    \end{align*}where we have chosen $\varrho_1=\frac{1}{2}\left(1+\sqrt{\frac{8}{7}}\right)$ and $\varrho_2=\frac{3}{4}$ so that $\frac{m'}{m}+\frac{3|f'|l^2}{4pf}\leq0$. After integration, we obtain the following bound\[
\mathbb{F}\lesssim e^{\frac{c_2}{\alpha}}\langle l\rangle^{c_3\frac{\sqrt{\alpha}}{\Omega}}e^{-\frac{c_0\mu t^3}{15}}\mathbb{F}(0),
\]where $c_2$ and $c_3$ are universal constants. Using the fact that\begin{align*}
    a\leq C\langle l\rangle^\frac{1}{2},\qquad a^{-1}\lesssim C^{-1}\left(1+\alpha^{-\frac{1}{4}}\right),\\
    b\lesssim C(1+\alpha^{-\frac{1}{4}}),\qquad b^{-1}\lesssim C^{-1}\langle l\rangle^{-\frac{1}{2}},
\end{align*}taking $C=1$, we obtain the following estimates:\begin{align*}
\left|(\widehat{U}_1,\widehat{U}_2)_{\neq}\right|^2 &\lesssim \left(1+\frac{1}{\Omega^2}\right)\frac{1}{\langle t\rangle^2}p^\frac{5}{2}(0)e^{c_2\alpha^{-1}}\langle l\rangle^{c_3\frac{\sqrt{\alpha}}{\Omega}}e^{-\frac{c_0\mu t^3}{15}}\left(\left|\frac{\widehat{\Theta}_{\neq}(0)}{\alpha^\frac{1}{2}}\right|^2+\frac{|\widehat{U}_{3,\neq}(0)|^2}{\langle l\rangle}+\frac{\Omega^2}{\alpha^2}\frac{|\widehat{K}_{\neq}(0)|^2}{\langle l\rangle^2}\right)\\
    \left|\widehat{U}_{3,\neq}\right|^2 &\lesssim e^{c_2\alpha^{-1}}\langle l\rangle^{c_3\frac{\sqrt{\alpha}}{\Omega}}e^{-\frac{c_0\mu t^3}{15}}\left(\left|\frac{\widehat{\Theta}_{\neq}(0)}{\alpha^\frac{1}{2}}\right|^2+\frac{|\widehat{U}_{3,\neq}(0)|^2}{\langle l\rangle}+\frac{\Omega^2}{\alpha^2}\frac{|\widehat{K}_{\neq}(0)|^2}{\langle l\rangle^2}\right),\\
    \left|\frac{\widehat{\Theta}_{\neq}}{\alpha^\frac{1}{2}}\right|^2 &\lesssim |l|e^{c_2\alpha^{-1}}\langle l\rangle^{c_3\frac{\sqrt{\alpha}}{\Omega}}e^{-\frac{c_0\mu t^3}{15}}\left(\left|\frac{\widehat{\Theta}_{\neq}(0)}{\alpha^\frac{1}{2}}\right|^2+\frac{|\widehat{U}_{3,\neq}(0)|^2}{\langle l\rangle}+\frac{\Omega^2}{\alpha^2}\frac{|\widehat{K}_{\neq}(0)|^2}{\langle l\rangle^2}\right),\\
     \left|\frac{\widehat{\mathcal{K}}_{\neq}}{\alpha}\right|^2 &\lesssim \frac{1}{\Omega^2}|l|^3e^{c_2\alpha^{-1}}\langle l\rangle^{c_3\frac{\sqrt{\alpha}}{\Omega}}e^{-\frac{c_0\mu t^3}{15}}\left(\left|\frac{\widehat{\Theta}_{\neq}(0)}{\alpha^\frac{1}{2}}\right|^2+\frac{|\widehat{U}_{3,\neq}(0)|^2}{\langle l\rangle}+\frac{\Omega^2}{\alpha^2}\frac{|\widehat{K}_{\neq}(0)|^2}{\langle l\rangle^2}\right).
\end{align*}As a result, we obtain the inviscid damping and the enhanced dissipation estimates in Theorem \eqref{Pran}.

\section{\textbf{Applications of the linearized potential vorticity}}\label{ch2}
As briefly discussed in Section \ref{sectionpv},  potential
vorticity is a fundamental Lagrangian invariant in geophysical
fluid dynamics.  In this section, we show that the linearized
potential vorticity provides a natural organizing variable for
perturbations of three-dimensional stratified shear flows. Consider again the 3D inviscid Boussinesq system\begin{flalign*}
\left\{
\begin{aligned}
& \partial_t V+(V\cdot\bm{\nabla})V=-\bm{\nabla}P+\eta\vec{\bm{e}}_3\\
   & \bm{\nabla}\cdot V=0,\\
&   \partial_t\eta+V\cdot\bm{\nabla}\eta=0,
\end{aligned}
\right.
\end{flalign*}We consider the more general steady state\begin{equation}
V^*=U(y,z)\vec{\bm{e}}_1,\qquad \eta^*=\alpha z,\qquad q^*=-\alpha\partial_y U(y,z).\label{general}
\end{equation}The linearized PV around this steady state is given by\[
q_l:=\bm{\nabla}(\alpha z)\cdot\begin{pmatrix}
    \partial_y u_3-\partial_z u_2\\
    \partial_zu_1-\partial_x u_3\\
    \partial_x u_2-\partial_y u_1
\end{pmatrix}+\bm{\nabla}\theta\cdot\begin{pmatrix}
    0\\\partial_zU\\-\partial_y U
\end{pmatrix}=\alpha\omega_3-\partial_y U\partial_z\theta+\partial_zU\partial_y\theta.
\]The linearized conservation law (\ref{PV}) then reduces to\begin{equation}
\partial_t  q_l+U(y,z)\partial_x q_l+u\cdot\bm{\nabla}q^*=0.\label{governingequoflpv}
\end{equation}

Using the incompressibility together with the definition of the linearized PV, we can solve for $u_2$ in terms of $q_l$, $u_3$ and $\theta$,\[
u_2=\Delta_h^{-1}\left(\frac{\partial_xq_l+\partial_y U\partial_{xz}\theta-\partial_z U\partial_{xy}\theta}{\alpha}-\partial_{yz}u_3\right).
\]Substituting this identity into (\ref{governingequoflpv}) yields
\begin{equation}
\boxed{\partial_tq_l+U\partial_xq_l-\partial^2_y U\partial_x\Delta^{-1}_hq_l=\alpha u_3\partial_{yz}U+\partial^2_y U\Delta^{-1}_h\left(\partial_y U\partial_{xz}\theta-\partial_z U\partial_{xy}\theta-\alpha\partial_{yz}u_3\right).}\label{structuredPV}
\end{equation}

When $U=U(y)$ is independent of $z$, the operator on the
left-hand side of (\ref{structuredPV}) is precisely the linearized
two-dimensional Euler around the shear profile $U(y)$.
Indeed, let $V^*=U(y)\vec{\bm{e}}_1$, and $\omega_3^*=-U'(y)$, the linearized 2D Euler (in vorticity form) is\[
\partial_t\omega+(U\partial_x-U''\partial_x\Delta^{-1})\omega=0.
\]We see that this coincides with the left-hand side of the $q_l$ equation\footnote{One can similarly derive a 2D forced Rayleigh equation by taking $E:=\Delta^{-1}_h q_l$.}. When the forcing is zero, we recover the standard linearized 2D Euler. Taking $l=0$, we see that if the 2D Euler is spectrally unstable for $U$, then so does the inviscid Boussinesq.
\vspace{3pt}

Moreover, from (\ref{structuredPV}), we see that the dynamics of the linearized PV is particularly simple when $\partial^2_y U=\partial_{yz}U=0$, this corresponds to profiles of the form\begin{equation}U(y,z)=ay+b(z)\label{specialss}\end{equation}where $b$ is an arbitrary function. These profiles consist of
a horizontal Couette component coupled with an arbitrary vertical
shear profile. In this case,\[
\partial_tq_l+U\partial_x q_l=0,\qquad \Rightarrow\qquad q_l(t,x,y,z)=q_l(0,x-tU,y,z).
\]That is, the linearized PV is simply transported by the background flow. It is therefore natural to first examine a representative member of
this class, and then to ask how much of this transport structure
survives for profiles lying just beyond it. This leads us to the
following two model flows.
\begin{itemize}
    \item \textit{The tilted Couette flow.} Let\begin{equation}
    U(y,z)=(y\cos\phi+z\sin\phi,0,0),\qquad \phi\in\left[0,\frac{\pi}{2}\right].\label{tcf}
    \end{equation}This represents a Couette flow whose shear direction forms an arbitrary angle with the stratification direction. Moreover, this is of the form $U=ay+b(z)$, so that the linearized PV can be solved explicitly. We will discuss the linear stability of this flow in Section \ref{ang}.
    \item \textit{The streak flow.} Let\[
    U(y,z)=yz.
    \]For this profile, the linearized PV equation cannot be solved explicitly, but it retains a relatively simple structure. In particular, the potential vorticity solves\[
\partial_tq_l+U\partial_xq_l=\alpha u_3.
\]In this case, we see that $q_l+\theta$ is linearly conserved, which implies a further non-trivial linear conserved quantity $\alpha\omega_3-z\partial_z\theta+y\partial_y\theta+\theta$.  
\end{itemize}

These examples illustrate that linearized potential vorticity can be
a powerful variable for organizing perturbation dynamics, especially
when its evolution equation possesses a closed or explicitly solvable
transport structure. In this case, we can derive the following reformulated $(u_3,\theta)$ system\begin{flalign*}
\left\{
\begin{aligned}
&\partial_t\theta+U\partial_x\theta=-\alpha u_3,  \\
&\partial_t u_3+U\partial_x u_3=\Delta_h\Delta^{-1}\theta+2\partial_{xz}\Delta^{-1}\left(\partial_y U\Delta_h^{-1}\left(\frac{\partial_xq_l+\partial_y U\partial_{xz}\theta-\partial_z U\partial_{xy}\theta}{\alpha}-\partial_{yz}u_3\right)+\partial_z U u_3\right).
\end{aligned}
\right.
\end{flalign*}It is therefore natural to focus on this system only, and the corresponding div-curl system for $(u_1,u_2)$ is given by\begin{flalign*}
\left\{
\begin{aligned}
&\partial_x u_2-\partial_y u_1=\frac{q_l-\partial_z U\partial_y\theta+\partial_y U\partial_z\theta}{\alpha},\\
    &\partial_x u_1+\partial_y u_2=-\partial_z u_3.
\end{aligned}
\right.
\end{flalign*}

Linearized potential vorticity has also played an important role in the
mathematical literature. We mention two representative examples.\vspace{5pt}

 \textsc{\textbf{1}}. In the nonlinear analysis in \cite{coti2024stability} of the zero modes of $u_1$ around the $(z,0,0)$ Couette flow, the following quantity was introduced\[
V:=-\alpha u_1+\theta
\]This is in fact the zero mode of the linearized PV. Indeed, the full linearized PV is given by\[
q_l=\alpha\omega_3+\partial_y\theta.
\]The zero mode of $q_l$ degenerates to\[
q_{l,0}=-\alpha\partial_y u_1+\partial_y\theta,
\]and we see that $\partial^{-1}_yq_l=V$.\vspace{3pt}

\textsc{\textbf{2}}. In \cite{antonelli2021linear}, the stability of Couette flow in 2D isentropic compressible fluid was studied. In this case, the potential vorticity is \[q=\frac{\omega}{\rho}.\]The corresponding steady vorticity and density are\[
\omega^*=-1,\qquad \rho^*=1.
\]The linearized PV in this case becomes\[
q_l=-1+\rho+\omega,
\]which plays a crucial role in the analysis of linear stability. (The constant $1$ can be neglected.) \vspace{4pt}

These examples illustrate that the linearized potential vorticity can be
a powerful tool for analyzing the dynamics of perturbations, particularly
when its evolution equation admits a simple or explicitly solvable
structure. In the remainder of this section, we present a further example that demonstrate how this formulation can be used to investigate linear
stability in more complicated setting.
\subsection{\textbf{Tilted Couette flow}}\label{ang}
In this section, we provide another application of the linearized PV to clarify the relationship between the plane and vertical Couette flows. Motivated by results obtained in Sections \ref{lsa} and \ref{weaklu}, in which the inviscid damping and lift-up patterns are very different from the vertical Couette flow setting \cite{coti2025suppression}, we consider the linearized stability of the following steady state \[
V^*=(\cos\phi y+\sin\phi z,0,0)^T,\qquad \theta^*=\alpha z,\qquad\alpha>0,\qquad \phi\in\left[0,\frac{\pi}{2}\right]
\]The perturbation is posed on $\Omega_{\mathbb{R}}$. This steady state is of the form (\ref{specialss}), and it can be regarded as the intermediate state between the two Couette flows as $\phi$ varies from $0$ to $\pi/2$. In this setting, the linearized system is formulated as follows\begin{flalign}
\left\{
\begin{aligned}
&\partial_t u+\mathcal{L} u+(\cos\phi u_2+\sin\phi u_3)\vec{\bm{e}}_1+\bm{\nabla}p=\theta\vec{\bm{e}}_3,\\
    &\partial_t\theta+\mathcal{L}\theta=-\alpha u_3,\\
    &\bm{\nabla}\cdot u=0.
\end{aligned}
\right.\label{lineearizationaroundtcf}
\end{flalign}Here $\mathcal{L}=(y\cos\phi +z\sin\phi)\partial_x$ is the shear operator. The pressure can be solved via the incompressibility,\[
 p=\Delta^{-1}\left(-2(\cos\phi\partial_x u_2+\sin\phi\partial_x u_3)+\partial_z\theta\right).
\]

\noindent\textbf{Non-zero modes.}\vspace{3pt}

As in the analysis of the plane Couette flow in Section \ref{lsa}, $(u_3,\theta)$ does not automatically form a closed system, a direct symmetrization method cannot be applied without reformulation. In this setting, the linearized PV reads\[
K_\phi=\alpha\omega_3+\mathcal{L}_1\theta,
\]where $\mathcal{L}_1=\sin\phi\partial_y-\cos\phi\partial_z$. One can verify that $K_\phi$ solves the transport-diffusion equation\[
\partial_t K_\phi+\mathcal{L}K_\phi=0.
\]Again, we can use $K_\phi$ to reformulate the $(u_{3,\neq},\theta_{\neq})$ system, as well as recover $(u_{1,\neq},u_{2,\neq})$ by solving the following div-curl system\begin{align}
    \left\{
\begin{aligned}
&\partial_x u_2-\partial_y u_1=\frac{K_\phi-\mathcal{L}_1\theta}{\alpha},\\
&\partial_x u_1+\partial_y u_2=-\partial_z u_3
\end{aligned}
\right.\qquad\Rightarrow\qquad  \left\{
\begin{aligned}
&\Delta_h u_1=-\partial_{xz} u_3-\frac{\partial_y K_\phi-\partial_y\mathcal{L}_1\theta}{\alpha},\\
&\Delta_h u_2=-\partial_{yz} u_3+\frac{\partial_x K_\phi-\partial_x\mathcal{L}_1\theta}{\alpha}
\end{aligned}
\right.\label{generaldivcurl}
\end{align}We define the new moving-frame coordinates $X=x-ty\cos\phi -tz\sin\phi,$ $Y=y,$ $ Z=z.$ 
The differential operators in this frame are given by\begin{align*}
&\partial_y=\partial_{Y}-t\cos\phi\partial_X=:\partial_Y^L,\qquad \partial_z=\partial_Z-t\sin\phi\partial_X=:\partial_Z^L\\
&\Delta_{H}=\partial^2_X+(\partial_Y-t\cos\phi \partial_X)^2,\quad\Delta_{L}=\partial^2_X+(\partial_Y-t\cos\phi \partial_X)^2+(\partial_Z-t\sin\phi \partial_X)^2.
\end{align*}We thus derive the following reformulated $(u_3,\theta)$ system in the moving frame,\begin{align}
   \left\{
\begin{aligned}
&\partial_tU_3-2\partial_Z^L\Delta_L^{-1}\left[\cos\phi\partial_X\left(-\partial_{YZ}^L\Delta^{-1}_H U_3+\frac{\partial_X\Delta^{-1}_H\mathcal{K}_\phi}{\alpha}-\frac{\partial_X\Delta^{-1}_H\mathcal{L}_1\Theta}{\alpha}\right)+\sin\phi \partial_X U_3\right]=\Delta_H\Delta_L^{-1}\Theta,\\
&\partial_t \Theta=-\alpha U_3.
\end{aligned}
\right.\label{newsystem}
\end{align}Define the following Fourier multipliers\begin{align*}
    f=k^2+(\eta-kt\cos\phi)^2,\qquad p=f+(l-kt\sin\phi)^2.
\end{align*}Consider the change of variables\[
\widehat{\Theta}_{\neq}=A,\qquad \widehat{U}_{3,\neq}=bB.
\]The Fourier-transformed system (\ref{newsystem}) becomes\begin{align}
   \left\{
\begin{aligned}
&\partial_tB=-\frac{b'}{b}B+\left(\frac{f'}{f}-\frac{p'}{p}\right) B-\frac{2i}{\alpha}\frac{k^2(l-kt\sin\phi)\cos\phi}{bpf}\widehat{\mathcal{K}}_\phi\\&\qquad-\frac{2k(l-kt\sin\phi)\cos\phi}{\alpha bpf}\left(k(\eta-kt\cos\phi)\sin\phi-k(l-kt\sin\phi)\cos\phi\right)A +\frac{f}{bp}A\\
&\partial_t A=-\alpha bB.
\end{aligned}
\right.\label{Fouriernewsystem}
\end{align}We set\[
b=\frac{1}{\alpha}\sqrt{\frac{z}{pf}},\qquad z=\alpha f^2+2k^2 p\cos^2\phi,
\]so that a valid symmetrization can be made. The reason for this choice of multiplier $b$ is the same as that for (\ref{symmetrizer}). Consider the energy functional $\mathbb{E}_3:=\frac{1}{2}\left[|A|^2+|B|^2\right]$, we have\begin{align*}
    \partial_t\mathbb{E}_3&=\left(\frac{3f'}{2f}-\frac{z'}{2z}-\frac{p'}{2p}\right)|B|^2-\frac{2ik^2\cos\phi(l-kt\sin\phi)}{\sqrt{fpz}}\widehat{\mathcal{K}}_{\phi,\neq} B+\frac{\alpha\sqrt{pf}}{\sqrt{z}}\left(-\frac{2k^2\cos^2\phi}{\alpha p}-\frac{1}{2\alpha}\frac{f'g'}{pf}\right)AB\\
    &=F_1+F_2+F_3.
\end{align*}(Here $g:=(l-kt\sin\phi)^2.$) We can bound each of the terms as follows\[
|F_3|\leq\frac{k^2}{\sqrt{\alpha}f}\mathbb{E}_3,\qquad |F_2|\leq\frac{k^2}{\sqrt{\alpha}f}\left(\left|\widehat{\mathcal{K}}_{\phi,\neq}\right|^2+\mathbb{E}_3\right)
\]$F_1$ can cause some regularity loss. We thus bound it as\[
|F_1|\leq\left|\frac{d}{dt}\ln\left(\frac{f^3}{pz}\right)\right|\mathbb{E}_3
\]We see that\begin{equation}
\int_0^t\left|\frac{d}{dt}\ln\left(\frac{f^3}{pz}\right)\right|\ ds\leq (N+1)\sup_{t_1,t_2\in\mathbb{R}}\left|\ln\left(\frac{f^3(t_1)}{p(t_1)z(t_1)}\cdot\frac{p(t_2)z(t_2)}{f^3(t_2)}\right)\right|,\label{NNNNN}
\end{equation}where $N$ is the maximal number of zeros of $\frac{3f'}{2f}-\frac{z'}{2z}-\frac{p'}{2p}$. After some computation, one can verify that $N\leq 6$. Besides, observe that\[
\left|\frac{f^3}{pz}\right|\leq\frac{1}{\alpha},\qquad \left|\frac{pz}{f^3}\right|\lesssim\left|\frac{k^2p+\alpha f^2}{f^2}\cdot\frac{p}{f}\right|\lesssim\alpha\frac{p}{f}+\frac{p^2}{f^2}.
\]Notice that\[
\frac{p}{f}=1+\frac{\left(l-\left(kt-\frac{\eta}{\cos\phi}+\frac{\eta}{\cos\phi}\right)\sin\phi\right)^2}{k^2+(\eta-kt\cos\phi)^2}\lesssim1+\frac{(l-\tan\phi\eta)^2+\tan^2\phi(\eta-kt)^2}{k^2+(\eta-kt)^2}\lesssim\sec^2\phi\cdot p(0).
\]As a result, we obtain the following estimate\[
e^{\int_0^t\left|\frac{d}{dt}\ln\left(\frac{f^3}{pz}\right)\right|\ ds}\lesssim\sec^{28}\phi\cdot p^{14}(0).
\]Consequently,\[
\mathbb{E}_3\lesssim  C_{\alpha,\phi}^2 p^{14}(0)\left(\mathbb{E}_3(0)+\left|\widehat{\mathcal{K}}_{\phi,\neq}\right|^2\right),\qquad C_{\alpha,\phi}:=e^{\frac{c}{\sqrt{\alpha}}}\sec^{14}\phi .
\]Note that\[
b(t)\lesssim\frac{\sqrt{1+\alpha}}{\alpha},\qquad b^{-1}(0)\lesssim \sqrt{\alpha}p^\frac{1}{2}(0),
\]we obtain\[
||u_{3,\neq}||_{L^2}+\left\|\frac{\theta_{\neq}}{\alpha^\frac{1}{2}}\right\|_{L^2}\lesssim  C_{\alpha,\phi}\left(||u_{\neq}(0)||_{H^{15}}+\left\|\frac{\theta_{\neq}(0)}{\alpha^\frac{1}{2}}\right\|_{H^{15}}\right).
\]This estimate implies that, when the shear direction is not parallel to the stratification direction, i.e. when $\phi\neq \pi/2$, the vertical velocity and the temperature components do not undergo inviscid damping. This observation is supported by the numerical simulations shown in Figure~\ref{fig3}. As $\phi\to\pi/2$, the decay pattern changes qualitatively. More precisely, in agreement with our analysis---apart from the possibly artefactual factor $\sec^{14}\phi$, $u_{3,\neq}$ and $\theta_{\neq}$ decay only at the endpoint $\phi=\pi/2$, whereas they exhibit oscillatory behavior for $\phi\neq\pi/2$.

\begin{figure}[H]
    \centering

    \begin{minipage}[t]{0.47\textwidth}
        \centering
        \includegraphics[width=\linewidth]{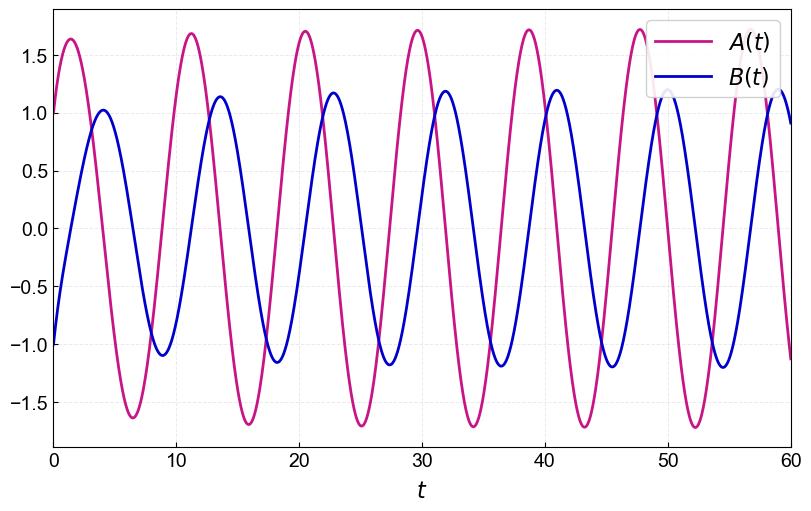}
        \par\vspace{-1.6ex}
        {\small\textbf{(a)}}
    \end{minipage}
    \hspace{0.008\textwidth}
    \begin{minipage}[t]{0.47\textwidth}
        \centering
        \includegraphics[width=\linewidth]{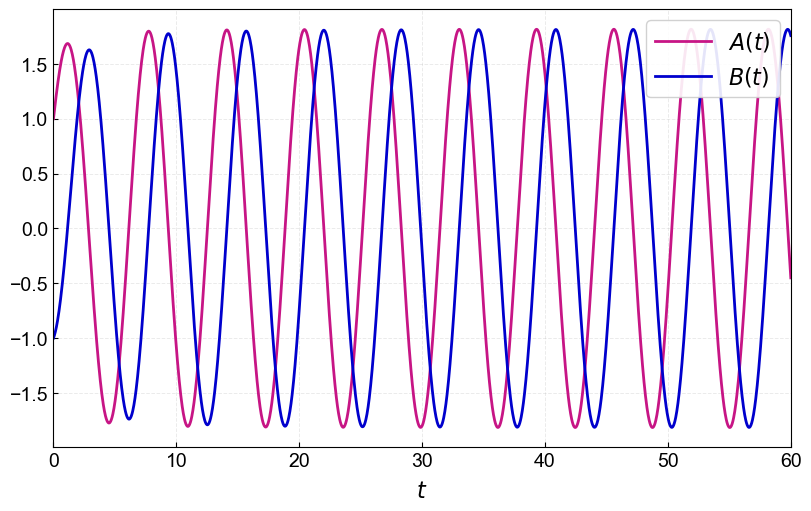}
        \par\vspace{-1.6ex}
        {\small\textbf{(b)}}
    \end{minipage}

    \par\vspace{-0.1ex}

    \begin{minipage}[t]{0.47\textwidth}
        \centering
        \includegraphics[width=\linewidth]{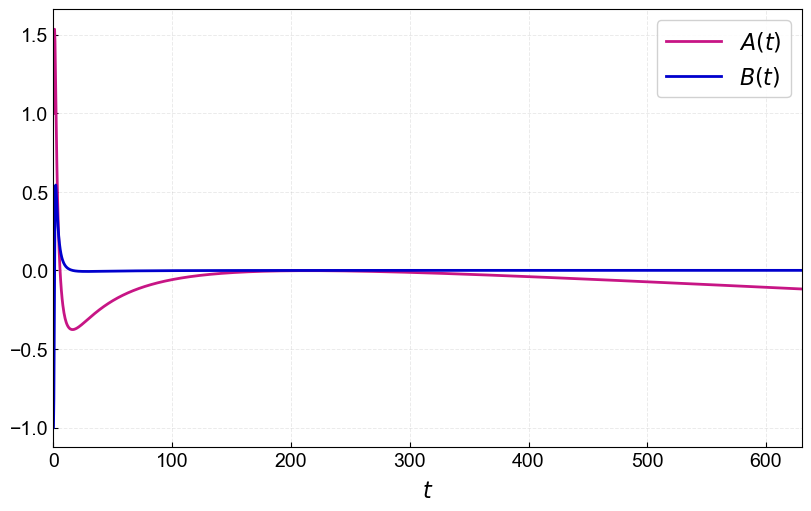}
        \par\vspace{-1.6ex}
        {\small\textbf{(c)}}
    \end{minipage}
    \hspace{0.008\textwidth}
    \begin{minipage}[t]{0.47\textwidth}
        \centering
        \includegraphics[width=\linewidth]{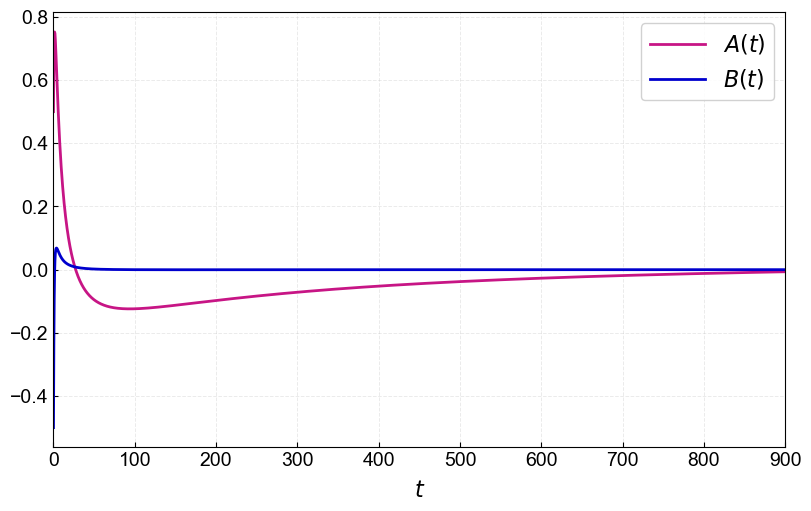}
        \par\vspace{-1.6ex}
        {\small\textbf{(d)}}
    \end{minipage}

    \caption{Numerical solutions for
    $(A,B)$
    when $\alpha=1$. Top left:
    $(k,\eta,l,\phi)=(-1,0,1,\pi/4)$; top right:
    $(k,\eta,l,\phi)=(1,-3,2,0)$; bottom left:
    $(k,\eta,l,\phi)=(2,1,1,\pi/2-0.001)$; bottom right:
    $(k,\eta,l,\phi)=(1,0,1,\pi/2)$}
    \label{fig3}
\end{figure}For $u_{1,\neq}$ and $u_{2,\neq}$, we use the formula \eqref{generaldivcurl}. $u_{1,\neq}$ experiences inviscid damping due to the bound\[
|\widehat{U}_{1,\neq}|\lesssim\frac{p(0)}{\langle t\cos\phi\rangle}\left[|\widehat{U}_{3,\neq}|+|\widehat{\mathcal{K}}_{\phi,\neq}(0)|+|\widehat{\Theta}_{\neq}|\right].
\]For $u_{2,\neq}$, notice that $\widehat{\partial_{YZ}^L U_{3,\neq}}=-(\eta-t\cos\phi)(l-t\sin\phi)\widehat{U}_{3,\neq}$, which means that the inviscid damping cannot be obtained by this formulation. Indeed, we verify by numerical simulation (see below) that, for any $\phi\in(0,\pi/2]$, $u_{2,\neq}$ does not experience inviscid damping.\begin{figure}[H]
    \centering
    \begin{minipage}[b]{0.49\textwidth}
        \centering
        \includegraphics[width=\textwidth]{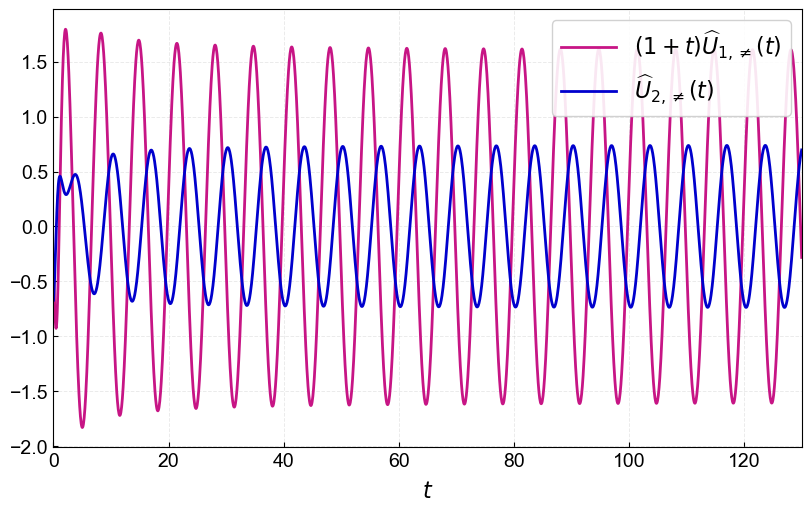}
        
    \end{minipage}
    \hspace{0.04cm}
    \begin{minipage}[b]{0.49\textwidth}
        \centering
        \includegraphics[width=\textwidth]{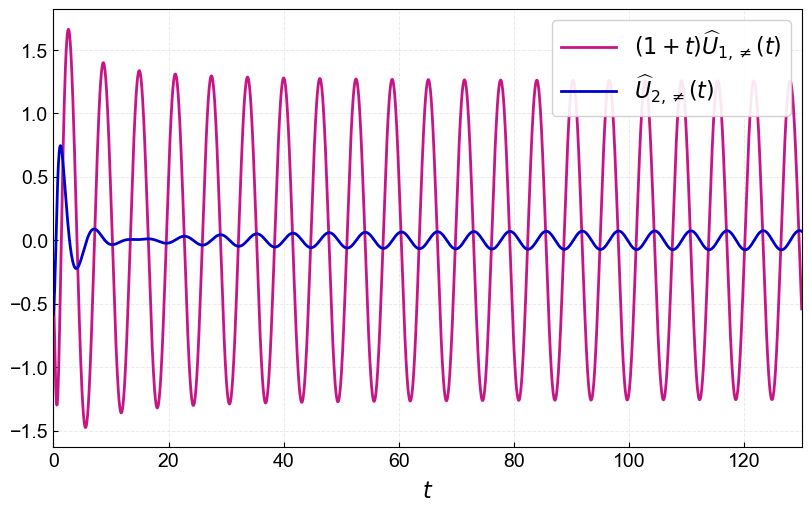}
    \end{minipage}
     \caption{Numerical solutions of $((1+t)\widehat{U}_{1,\neq},\widehat{U}_{2,\neq})$, with $\alpha=1$. Left:
$\phi=\frac{\pi}{9}$. Right:
$\phi=\frac{\pi}{90}$}
\end{figure}One thing to notice is that, as $\phi\to0$, the solution can first decay and remains small without further time-decay after certain time. To verify this, observe that\[
\left|\frac{(\eta-kt\cos\phi)(l-kt\sin\phi)}{f}\right|\leq\left|\frac{l-kt\sin\phi}{\sqrt{k^2+(\eta-kt\cos\phi^2)}}\right|\lesssim\frac{f^\frac{1}{2}(0)}{1+t\cos\phi}\cdot(1+t\sin\phi)f^\frac{1}{2}(0)\lesssim f(0)\max\left\{\frac{1}{\langle t\rangle},\tan\phi\right\}
\]as $\phi\to0$. Therefore, we obtain the estimate\[
||u_{2,\neq}||_{L^2}\lesssim  C_{\alpha,\phi}\max\left\{\frac{1}{\langle t\rangle},\tan\phi\right\}\left(||u_{\neq}(0)||_{H^{17}}+\left\|\frac{\theta_{\neq}(0)}{\alpha^\frac{1}{2}}\right\|_{H^{17}}\right).
\]In general, the following bound holds true:\[
\langle t\cos\phi\rangle||u_{1,\neq}||_{L^2}+||u_{2,\neq}||_{L^2}\lesssim  C_{\alpha,\phi}\left(||u_{\neq}(0)||_{H^{17}}+\left\|\frac{\theta_{\neq}(0)}{\alpha^\frac{1}{2}}\right\|_{H^{17}}\right).
\]\savegeometry{3pt}

\noindent\textbf{Zero modes.}

The zero mode system reads\begin{flalign*}
\left\{
\begin{aligned}
&\partial_t u_{1,0}+\cos\phi u_{2,0}+\sin\phi u_{3,0}=0,\\
&\partial_tu_{2,0}=-\partial_{yz}\Delta^{-1}\theta_0,\\
&\partial_t u_{3,0}=\partial^2_y\Delta^{-1}\theta_0,\\
    &\partial_t\theta_0=-\alpha u_{3,0}.
\end{aligned}
\right.
\end{flalign*}The \((u_{2,0},u_{3,0},\theta_0)\) system is independent of \(\phi\), since the shear effect vanishes in the linearized zero-mode dynamics. The exact solutions for these three quantities are the same as those in (\ref{eq:B9}), and hence they exhibit the same dispersive properties. On the other hand, for \(u_1\), we have $\widehat{u}_{1,0}=\widehat{u}_{1,0}(0)-\int_0^t\cos\phi u_{2,0}+\sin\phi u_{3,0}$. Integrating for $\widehat{u}_{3,0}$, we have\begin{align*} \int_0^t\widehat{u}_{3,0}(s)\ ds=-\frac{\widehat{\theta}_0(0)}{\alpha}\left(\cos\left(\frac{\sqrt{\alpha}t\eta}{\sqrt{\eta^2+l^2}}\right)-1\right)+\widehat{u}_{3,0}(0)\frac{\sqrt{\eta^2+l^2}}{\sqrt{\alpha}\eta}\sin\left(\frac{\sqrt{\alpha}t\eta}{\sqrt{\eta^2+l^2}}\right). \end{align*}We have\[ \frac{\sqrt{\eta^2+l^2}}{\eta}\widehat{u}_3(0)=\frac{\eta}{\sqrt{\eta^2+l^2}}\widehat{u}_3(0)-\frac{l}{\sqrt{\eta^2+l^2}}\widehat{u}_2(0), \]which does not have any singular term. However, when integrating \(\widehat{u}_{2,0}\) in time, we find that the singular contribution appears as in (\ref{eq:B9}). As a result, we find that a similar lift-up effect is caused by the operator
\[
   \alpha^{-\frac{1}{2}}\cos\phi\,
   \partial_z\partial_y^{-1}
   e^{\nu t\Delta}
   \left(
   e^{\pm\sqrt{\alpha}t\partial_y|\bm{\nabla}_{y,z}|^{-1}}
   -I_d
   \right)=\cos\phi\mathcal{L}_{\text{Bou}}.
\]
The corresponding results regarding transient growth can be found in Section \ref{weaklu}.

\section{\textbf{Nonlinear stability of the plane Couette flow}}\label{ch3}
We now turn to the nonlinear stability of the plane Couette flow. For simplicity, we take $\nu=\kappa$ and consider perturbations set on $\Omega_{\mathbb{T}}$.

To study the dynamics of the non-zero modes, we analyze the system $(U_{3,\neq},\Theta_{\neq},\mathcal{K}_{\neq})$ and recover $U_{1,2,\neq}$ from the div--curl system. We rescale each unknown so that they all have the same order in terms of $\alpha$. Let
\begin{equation}
\Theta=\sqrt{\alpha}\Psi,\qquad K=\alpha\Xi.\label{additional scaling}
\end{equation}Thus, the nonlinear system now reads\begin{flalign}
\left\{
\begin{aligned}
 &\partial_t\Psi=\nu\Delta_L\Psi-\sqrt{\alpha} U_3-N_1,\\
    &\partial_t U_3=\nu\Delta_L U_3+\sqrt{\alpha}\left(\Delta_H\Delta^{-1}_L+\frac{2}{\alpha}\partial^2_{XZ}\Delta^{-1}_L\Delta^{-1}_H\right)\Psi+2\partial^2_X\partial_{Z}\Delta^{-1}_L\Delta^{-1}_H\Xi -2\partial_{XY}^L\partial_Z^2\Delta^{-1}_L\Delta^{-1}_H U_3-N_2,\\
&\partial_t \Xi=\nu\Delta_L \Xi-N_3.
\end{aligned}
\right.\label{full system1}
\end{flalign}The nonlinearities are\begin{align*}
 N_1&=U\cdot\bm{\nabla}_L\Psi,\\
    N_2&=\partial_Z P^{NL}+U\cdot\bm{\nabla}_L U_3,\\
    N_3&=\left[\partial_X U\cdot\bm{\nabla}_L U_2-\partial_Y^L U\cdot\bm{\nabla}_L U_1\right]-\frac{1}{\sqrt{\alpha}}\partial_Z U\cdot\bm{\nabla}_L \Psi+U\cdot\bm{\nabla}_L \Xi,\\
    P^{NL}&=-\Delta^{-1}_L\left(\bm{\nabla}_L\cdot \left((U\cdot\bm{\nabla}_L)U\right)\right)=-\Delta^{-1}_L(\partial_{i}^L U_j\partial_j^L U_i)
\end{align*}The pair $(U_{1,\neq},U_{2,\neq})$ can be expressed as follows\begin{align}   U_{1,\neq}&=\frac{\partial_X\partial_Z}{|\bm{\nabla}_H|^2}U_{3,\neq}+\frac{\partial_Y^L}{|\bm{\nabla}_H|^2}\Xi_{\neq}+\frac{1}{\sqrt{\alpha}}\frac{\partial_Z\partial_Y^L}{|\bm{\nabla}_H|^2}\Psi_{\neq},\label{rec1}\\
U_{2,\neq}&=\frac{\partial_Y^L\partial_Z}{|\bm{\nabla}_H|^2}U_{3,\neq}-\frac{\partial_X}{|\bm{\nabla}_H|^2}\Xi_{\neq}-\frac{1}{\sqrt{\alpha}}\frac{\partial_Z\partial_X}{|\bm{\nabla}_H|^2}\Psi_{\neq}.\label{rec2}
\end{align}

We shall prove the stability of the non-zero modes via the energy method. As in the linear analysis in Section~\ref{lsa}, we define the following symmetrized variables
\[
\Psi=\bm{a}A,\qquad U_3=\bm{b}B,
\]where\[
\widehat{\bm{a}A}=C\alpha^\frac{1}{4}\frac{p^{\frac{1}{4}-\frac{\lambda}{2}}f^{\frac{1}{4}+\frac{\lambda}{2}}}{z^\frac{1}{4}}\widehat{A},\qquad \widehat{\bm{b}B}=\frac{C}{\alpha^\frac{1}{4}}\frac{z^\frac{1}{4}}{p^{\frac{1}{4}+\frac{\lambda}{2}}f^{\frac{1}{4}-\frac{\lambda}{2}}}.
\]Here, the time-independent constant $C=C(\alpha,k,\eta,l)$ is chosen as follows. Considering the solutions (\ref{rec1}) and (\ref{rec2}), we seek to ensure that\begin{equation}
\left|\widehat{U}_{1,2,\neq}\right|\lesssim\left|\widehat{A}_{\neq}\right|+\left|\widehat{B}_{\neq}\right|+\left|\widehat{\Xi}_{\neq}\right|.\label{balance}
\end{equation}As a result, we need to carefully estimate the upper bounds for $\frac{a|l|}{\alpha^\frac{1}{2}f^\frac{1}{2}}$ and $\frac{b|l|}{f^\frac{1}{2}}$ (the coefficients of $U_{3,\neq}$ and $\Psi_{\neq}$).\vspace{3pt} 

\begin{itemize}
    \item When $\lambda\geq\frac{1}{2}$, we have\begin{align*}
    \frac{a|l|}{\alpha^\frac{1}{2}f^\frac{1}{2}}=\frac{C}{\alpha^\frac{1}{4}}\frac{|l|}{f^\frac{1}{2}}\frac{f^{\frac{1}{4}+\frac{\lambda}{2}}}{p^{\frac{\lambda}{2}-\frac{1}{4}}z^\frac{1}{4}}=\frac{C|l|}{\alpha^\frac{1}{4}}\frac{f^{\frac{\lambda}{2}-\frac{1}{4}}}{p^{\frac{\lambda}{2}-\frac{1}{4}}z^\frac{1}{4}}.
\end{align*}When $\alpha f^2\geq k^2p$, we have\[
\frac{C|l|}{\alpha^\frac{1}{4}}\frac{f^{\frac{\lambda}{2}-\frac{1}{4}}}{p^{\frac{\lambda}{2}-\frac{1}{4}}z^\frac{1}{4}}\lesssim\frac{C|l|^{\frac{3}{2}-\lambda}}{\alpha^\frac{1}{4}}\frac{f^{\frac{\lambda}{2}-\frac{1}{4}}}{\alpha^\frac{1}{4}f^{\frac{\lambda}{2}-\frac{1}{4}}f^{\frac{3}{4}-\frac{\lambda}{2}}}.
\]Using the fact that $\frac{|l|}{f}\leq\frac{\sqrt{\alpha}}{|k|}$, we have\[
\frac{C|l|}{\alpha^\frac{1}{4}}\frac{f^{\frac{\lambda}{2}-\frac{1}{4}}}{p^{\frac{\lambda}{2}-\frac{1}{4}}z^\frac{1}{4}}\lesssim\frac{C}{\alpha^\frac{1}{4}}|l|^{\frac{3}{4}-\frac{\lambda}{2}}\frac{\alpha^{\frac{1}{2}\left(\frac{3}{4}-\frac{\lambda}{2}\right)}}{\alpha^\frac{1}{4}|k|^{\frac{3}{4}-\frac{\lambda}{2}}}\lesssim\frac{C}{\alpha^{\frac{\lambda}{4}+\frac{1}{8}}}\frac{|l|^{\frac{3}{4}-\frac{\lambda}{2}}}{|k|^{\frac{3}{4}-\frac{\lambda}{2}}}.
\]When $\alpha f^2\leq k^2p$, we similarly have\[
    \frac{a|l|}{\alpha^\frac{1}{2}f^\frac{1}{2}}\lesssim\frac{C}{\alpha^\frac{1}{4}}\frac{|l|f^{\frac{\lambda}{2}-\frac{1}{4}}}{p^{\frac{\lambda}{2}-\frac{1}{4}}|k|^\frac{1}{2}p^\frac{1}{4}}.
\]Using the fact that $f\leq\frac{|k|\sqrt{p}}{\sqrt{\alpha}}$, we obtain \[
   \frac{a|l|}{\alpha^\frac{1}{2}f^\frac{1}{2}}\lesssim\frac{C}{\alpha^\frac{1}{4}}\frac{|l|}{p^\frac{\lambda}{2}|k|^\frac{1}{2}}\frac{|k|^{\frac{\lambda}{2}-\frac{1}{4}}p^{\frac{\lambda}{4}-\frac{1}{8}}}{\alpha^{\frac{\lambda}{4}-\frac{1}{8}}}\lesssim\frac{C}{\alpha^{\frac{\lambda}{4}+\frac{1}{8}}}\frac{|l|^{\frac{3}{4}-\frac{\lambda}{2}}}{|k|^{\frac{3}{4}-\frac{\lambda}{2}}}.
\]For $\frac{b|l|}{f^\frac{1}{2}}$, we see that\begin{align*}
\frac{b|l|}{f^\frac{1}{2}}&\lesssim\frac{C}{\alpha^\frac{1}{4}}\frac{|l||k|^\frac{1}{2}}{p^\frac{\lambda}{2}f^{\frac{3}{4}-\frac{\lambda}{2}}}+C\frac{|l|f^{\frac{\lambda}{2}-\frac{1}{4}}}{p^{\frac{1}{4}+\frac{\lambda}{2}}}\\
&\lesssim\frac{C}{\alpha^\frac{1}{4}}\frac{|l|^{1-\lambda}}{|k|^{1-\lambda}}+C
\end{align*}Therefore, when $\lambda\geq\frac{1}{2}$, we have\[
 \left|\widehat{U}_{1,2,\neq}\right|\lesssim C\left(\frac{|l|^{1-\lambda}}{\alpha^\frac{1}{4}|k|^{1-\lambda}}+1\right)|B_{\neq}|+\frac{1}{f^{\frac{1}{2}}}|\Xi_{\neq}|+\frac{C}{\alpha^{\frac{\lambda}{4}+\frac{1}{8}}}\frac{|l|^{\frac{3}{4}-\frac{\lambda}{2}}}{|k|^{\frac{3}{4}-\frac{\lambda}{2}}}|A_{\neq}|.
\]Therefore, we choose \begin{equation}
\boxed{C\left(\frac{\langle l\rangle^{\frac{3}{4}-\frac{\lambda}{2}}}{\alpha^\frac{1}{4}}+1\right)=1}\label{sit1}
\end{equation}so that (\ref{balance}) is satisfied. 
\item When $\lambda\leq\frac{1}{2}$, we have\[
    \frac{a|l|}{\alpha^\frac{1}{2}f^\frac{1}{2}}=\frac{C|l|}{\alpha^\frac{1}{4}}\frac{p^{\frac{1}{4}-\frac{\lambda}{2}}}{f^{\frac{1}{4}-\frac{\lambda}{2}}z^\frac{1}{4}}.
\]If $\alpha f^2\geq k^2p$, then we  have\[
    \frac{a|l|}{\alpha^\frac{1}{2}f^\frac{1}{2}}\lesssim\frac{C|l|p^{\frac{1}{4}-\frac{\lambda}{2}}}{\alpha^\frac{1}{2}f^{\frac{3}{4}-\frac{\lambda}{2}}}.
\]Using the fact that $\frac{1}{f}\leq\frac{\sqrt{\alpha}}{|k|\sqrt{p}}$, we have\[
    \frac{a|l|}{\alpha^\frac{1}{2}f^\frac{1}{2}}\lesssim\frac{C}{\alpha^{\frac{1}{8}+\frac{\lambda}{4}}}\frac{|l|^{\frac{3}{4}-\frac{\lambda}{2}}}{|k|^{\frac{3}{4}-\frac{\lambda}{2}}}.
\]Alternatively, using $\frac{1}{\alpha^\frac{1}{4}f^\frac{1}{2}}\leq\frac{1}{|k|^\frac{1}{2}p^\frac{1}{4}}$, we have\[
 \frac{a|l|}{\alpha^\frac{1}{2}f^\frac{1}{2}}\lesssim\frac{C|l|p^{\frac{1}{4}-\frac{\lambda}{2}}}{\alpha^\frac{1}{2}f^{\frac{3}{4}-\frac{\lambda}{2}}}\lesssim\frac{C|l|p^{\frac{1}{4}-\frac{\lambda}{2}}}{\alpha^\frac{1}{2}f^\frac{1}{2}f^{\frac{1}{4}-\frac{\lambda}{2}}}\lesssim\frac{C|l|^{1-\lambda}}{\alpha^\frac{1}{4}|k|^{1-\lambda}}.
\]In the case where $\alpha f^2\leq k^2p$, we have\[
  \frac{a|l|}{\alpha^\frac{1}{2}f^\frac{1}{2}}\lesssim\frac{C|l|p^{\frac{1}{4}-\frac{\lambda}{2}}}{\alpha^\frac{1}{4}f^{\frac{1}{4}-\frac{\lambda}{2}}|k|^\frac{1}{2}p^\frac{1}{4}}\lesssim\frac{C|l|^{1-\lambda}}{\alpha^\frac{1}{4}|l|^{1-\lambda}}.
\]For $\frac{b|l|}{f^\frac{1}{2}}$, we similarly have\begin{align*}
    \frac{b|l|}{f^\frac{1}{2}}&\lesssim\frac{C}{\alpha^\frac{1}{4}}\frac{|l||k|^\frac{1}{2}}{p^\frac{\lambda}{2}f^{\frac{3}{4}-\frac{\lambda}{2}}}+C\frac{|l|f^{\frac{\lambda}{2}-\frac{1}{4}}}{p^{\frac{1}{4}+\frac{\lambda}{2}}}\\
&\lesssim\frac{C}{\alpha^\frac{1}{4}}\frac{|l|^{1-\lambda}}{|k|^{1-\lambda}}+C\frac{|l|^{\frac{1}{2}-\lambda}}{|k|^{\frac{1}{2}-\lambda}}.
\end{align*}Thus, we have the bound\[
 \left|\widehat{U}_{1,2,\neq}\right|\lesssim C\frac{|l|^{\frac{1}{2}-\lambda}}{|k|^{\frac{1}{2}-\lambda}}\left(\frac{|l|^{\frac{1}{2}}}{\alpha^\frac{1}{4}|k|^{\frac{1}{2}}}+1\right)|B_{\neq}|+\frac{1}{f^{\frac{1}{2}}}|\Xi_{\neq}|+\frac{C|l|^{1-\lambda}}{\alpha^\frac{1}{4}|l|^{1-\lambda}}|A_{\neq}|.
\]Thus, when $\lambda\leq\frac{1}{2}$, we need to choose\begin{equation}
\boxed{C\frac{\langle l\rangle^{\frac{1}{2}-\lambda}}{|k|^{\frac{1}{2}-\lambda}}\left(\frac{|l|^{\frac{1}{2}}}{\alpha^\frac{1}{4}|k|^{\frac{1}{2}}}+1\right)=1.}\label{sit2}
\end{equation}
\end{itemize}We first keep $\lambda$ general and then show that there is only one choice suitable for the nonlinear stability analysis. To prove the nonlinear energy estimate, we further define the following Fourier multipliers\begin{equation}
\mathcal{Z}=\mathcal{M}e^{\varpi\nu^\frac{1}{3}t}\langle\bm{\nabla}\rangle^{s}\mathcal{N},\qquad \mathcal{N}=\prod_{i=1}^3\mathcal{M}_i,\label{nonlinearmultplier}
\end{equation}where $\mathcal{M}_i$, $i=1,2,3$, are the Fourier multipliers defined by $\widehat{\mathcal{M}_if}=M_i\widehat{f}$, with $M_i(0,k,\eta,l)=1$ and\begin{align*}
 -\frac{M_1'}{M_1}&=\frac{3}{2}\left|\frac{d}{dt}\left(\frac{f'}{2}\frac{2\lambda l^2z+2k^2p^2-\alpha l^2f^2}{\sqrt{fp}z^{\frac{3}{2}}}\right)\right|,\\ -\frac{M_2'}{M_2}&=\frac{\vartheta |k|^\frac{3}{2}}{f^{\frac{3}{4}}},\\ -\frac{M_3'}{M_3}&=\frac{\nu^\frac{1}{3}k^2}{k^2+\nu^\frac{2}{3}(\eta-kt)^2}.
\end{align*}Here $\vartheta=2\left(1+\alpha^{-\frac{1}{2}}\right)$ and $\varpi>0$ is a universal constant. We will explain the motivation for these multipliers shortly. The multiplier $\mathcal{M}$ satisfies $\widehat{\mathcal{M}f}=m\widehat{f}$, with $m$ defined by \eqref{multiplier}. Taking the time derivative of $\mathcal{Z}$, we see that\[
\mathcal{Z}'=\frac{\mathcal{M}'}{\mathcal{M}}\mathcal{Z}+\varpi\nu^\frac{1}{3}\mathcal{Z}+\frac{\mathcal{N}'}{\mathcal{N}}\mathcal{Z}.
\]To study the non-zero mode system, we define the following energy\[
\mathbb{F}=\frac{1}{2}\left(||\mathcal{Z}A_{\neq}||_{L^2}^2+||\mathcal{Z}B_{\neq}||_{L^2}^2+||\mathcal{Z}\Xi_{\neq}||_{L^2}^2+\langle\mathcal{C}\mathcal{Z}A_{\neq},\mathcal{Z}B_{\neq}\rangle\right)
\]Here, $\mathcal{C}$ is the multiplier used in linear analysis, see (\ref{crossfunction}).\[
\widehat{\mathcal{C}g}=\frac{f'}{2}\frac{2\lambda l^2z+2k^2p^2-\alpha l^2f^2}{\sqrt{fp}z^{\frac{3}{2}}}\widehat{g}.
\]We now claim the following:\vspace{8pt}

\textit{Assuming that (\ref{balance}) holds, the linear energy estimate can be closed only if $\lambda\leq\frac{1}{2}$ (see (\ref{lambdas2})). As a consequence, the nonlinear energy estimate can only be closed if $\lambda=\frac{1}{2}$}.\vspace{10pt}

\noindent We first compute the time derivative of $\mathbb{F}$,
\begin{align}\partial_t\mathbb{F}=&\varpi\nu^\frac{1}{3}\sum_{F\in\{A,B,\Xi\}}\left\|\mathcal{Z}F_{\neq}\right\|^2_{L^2}-\sum_{F\in\{A,B,\Xi\}}\left\|\sqrt{-\frac{\mathcal{M}'}{\mathcal{M}}}ZF_{\neq}\right\|^2_{L^2}-\sum_{F\in\{A,B,\Xi\}}\left\|\sqrt{-\frac{\mathcal{N}'}{\mathcal{N}}}ZF_{\neq}\right\|^2_{L^2}\nonumber\\
    &+\langle\mathcal{C}\mathcal{Z}A_{\neq},\left(\frac{\mathcal{M}'}{\mathcal{M}}+\frac{\mathcal{N}'}{\mathcal{N}}\right)\mathcal{Z}B_{\neq}\rangle+\varpi\nu^\frac{1}{3}\langle\mathcal{C}\mathcal{Z}A_{\neq},\mathcal{Z}B_{\neq}\rangle+\underbrace{\frac{1}{2}\langle \mathcal{C}'\mathcal{Z}A_{\neq},\mathcal{Z}B_{\neq}\rangle}_{1}\nonumber\\
    &-\nu\sum_{F\in\{A,B,\Xi\}}\left\|\bm{\nabla}_L\mathcal{Z}F_{\neq}\right\|^2_{L^2}+\frac{\nu}{2}\langle \mathcal{C}\mathcal{Z}\Delta_LA_{\neq},\mathcal{Z}B_{\neq}\rangle+\frac{\nu}{2}\langle\mathcal{C}\mathcal{Z}\Delta_L B_{\neq},\mathcal{Z}A_{\neq}\rangle\nonumber\\
    &-\underbrace{\frac{2}{\sqrt{\alpha}}\langle \mathcal{Z}\partial^2_{X}\Delta^{-1}_L\frac{\bm{a}}{\bm{b}}A_{\neq},\mathcal{Z}B_{\neq}\rangle}_{2}+\underbrace{2\langle\mathcal{Z}\partial_Z\Delta^{-1}_L\partial^2_X\Delta^{-1}_H\bm{b}^{-1}\Xi_{\neq},\mathcal{Z}B_{\neq}\rangle}_{3}\nonumber\\
    &-\underbrace{\frac{1}{\sqrt{\alpha}}\langle\mathcal{C}\mathcal{Z}\partial^2_{X}\Delta^{-1}_L\frac{\bm{a}}{\bm{b}}A_{\neq},\mathcal{Z}A_{\neq}\rangle}_{4}+\underbrace{\langle\mathcal{C}\mathcal{Z}\partial_Z\Delta^{-1}_L\partial^2_X\Delta^{-1}_H\bm{b}^{-1}\Xi_{\neq},\mathcal{Z} A_{\neq}\rangle}_{5}\nonumber\\
    &-\langle\mathcal{Z}\bm{a}^{-1}N_{1,\neq},\mathcal{Z}A_{\neq}\rangle-\langle\mathcal{Z}\bm{b}^{-1}N_{2,\neq},\mathcal{Z}B_{\neq}\rangle-\langle\mathcal{Z}N_{3,\neq},\mathcal{Z}\Xi_{\neq}\rangle-\frac{1}{2}\langle\mathcal{C}\mathcal{Z}\bm{a}^{-1}N_{1,\neq},\mathcal{Z}B_{\neq}\rangle\nonumber\\
    &-\frac{1}{2}\langle\mathcal{C}\mathcal{Z}\bm{b}^{-1}N_{2,\neq},\mathcal{Z}A_{\neq}\rangle\nonumber\\
    &-2(1-\lambda)\langle\mathcal{Z}B_{\neq},\mathcal{Z}\left(\partial_{XY}^L\partial^2_Z\Delta^{-1}_L\Delta^{-1}_H\right)B_{\neq}\rangle-(1-\lambda)\langle\mathcal{C}\mathcal{Z}A_{\neq},\mathcal{Z}\left(\partial_{XY}^L\partial^2_Z\Delta^{-1}_L\Delta^{-1}_H\right)B_{\neq}\rangle.\label{nonlinearenergyestimate}
\end{align}We first control the linear terms on the right-hand side. To begin with, we have the estimate\[
\left|\frac{\nu}{2}\langle \mathcal{C}\mathcal{Z}\Delta_LA_{\neq},\mathcal{Z}B_{\neq}\rangle+\frac{\nu}{2}\langle\mathcal{C}\mathcal{Z}\Delta_L B_{\neq},\mathcal{Z}A_{\neq}\rangle\right|\leq\zeta\nu\left(\left\|\mathcal{Z}\bm{\nabla}_LA_{\neq}\right\|^2+\left\|\mathcal{Z}\bm{\nabla}_LB_{\neq}\right\|^2\right),
\]Moreover, the last two terms can be absorbed by the Fourier multiplier $\mathcal{M}$, i.e.\begin{align*}
&-2(1-\lambda)\langle\mathcal{Z}B_{\neq},\mathcal{Z}\left(\partial_{XY}^L\partial^2_Z\Delta^{-1}_L\Delta^{-1}_H\right)B_{\neq}\rangle-(1-\lambda)\langle\mathcal{C}\mathcal{Z}A_{\neq},\mathcal{Z}\left(\partial_{XY}^L\partial^2_Z\Delta^{-1}_L\Delta^{-1}_H\right)B_{\neq}\rangle\\
&-\sum_{F\in\{A,B,\Xi\}}\left\|\sqrt{-\frac{\mathcal{M}'}{\mathcal{M}}}ZF_{\neq}\right\|^2_{L^2}-\frac{1}{3}\nu\sum_{F\in\{A,B,\Xi\}}\left\|\bm{\nabla}_L\mathcal{Z}F_{\neq}\right\|^2_{L^2}\leq0.
\end{align*}Terms 1-5 can be absorbed using multiplier $\mathcal{M}_1$ and $\mathcal{M}_2$ as follows.\begin{itemize}
    \item 1. We estimate it as\[
    |\langle\mathcal{C}'\mathcal{Z}A_{\neq},\mathcal{Z}B_{\neq}\rangle|\leq\frac{1}{2}\left\|\sqrt{|\mathcal{C}'|}\mathcal{Z}A_{\neq}\right\|^2_{L^2}+\frac{1}{2}\left\|\sqrt{|\mathcal{C}'|}\mathcal{Z}B_{\neq}\right\|^2_{L^2}.
    \]These two terms can be absorbed using $\mathcal{M}_1$, i.e. $||\sqrt{|\mathcal{C}'|}\mathcal{Z}F_{\neq}||^2_{L^2}-\frac{3}{4}\left\|\sqrt{-\frac{M_1'}{M_1}}\mathcal{Z}F_{\neq}\right\|_{L^2}^2\leq0$ .
    \item 2-5. Since 2 and 4 are similar, we only consider term 2\begin{align*}
     \frac{2}{\sqrt{\alpha}}\left|\langle\mathcal{Z}\partial^2_X\Delta^{-1}_L\frac{\bm{a}}{\bm{b}}A_{\neq},\mathcal{Z}B_{\neq}\rangle\right|&\leq\int\left|\widehat{\mathcal{Z}A_{\neq}}\right|\frac{2k^2}{p}\frac{f^\frac{1}{2}p^\frac{1}{2}}{z^\frac{1}{2}}\left|\overline{\widehat{\mathcal{Z}B_{\neq}}}\right|\\
&\leq\int\left|\widehat{\mathcal{Z}A_{\neq}}\right|\frac{2k^2}{\alpha^\frac{1}{2}f}\left|\overline{\widehat{\mathcal{Z}B_{\neq}}}\right|.
    \end{align*}Thus, this term can be absorbed using $\mathcal{M}_2$. Similarly, we consider 3 only, as 5 can be bounded in the same way.\begin{align*}
\left|2\langle\mathcal{Z}\partial_Z\Delta^{-1}_L\partial^2_X\Delta^{-1}_H\bm{b}^{-1}\Xi_{\neq},\mathcal{Z}B_{\neq}\rangle\right|&\leq\int\left|\widehat{\mathcal{Z}\Xi_{\neq}}\right|\frac{2k^2|l|}{pf}b^{-1}\left|\overline{\widehat{\mathcal{Z}B_{\neq}}}\right|.
    \end{align*}First note that, when $\lambda\geq\frac{1}{2}$, $C$ is defined as (\ref{sit1}), thus\begin{equation}
    \frac{k^2|l|}{pf}b^{-1}=\left(1+\frac{\langle l\rangle^{\frac{3}{4}-\frac{\lambda}{2}}}{\alpha^\frac{1}{4}}\right)\alpha^\frac{1}{4}\frac{k^2|l|}{p^{\frac{3}{4}-\frac{\lambda}{2}}f^{\frac{3}{4}+\frac{\lambda}{2}}z^\frac{1}{4}}.\label{lambdas2}
    \end{equation}For the first term, we have to avoid having $\alpha^\frac{1}{4}$ in the exponent as otherwise the bound will become extremely large when $\alpha$ becomes large. Therefore, the only way to bound the first term is to extract $\alpha f^2$ from $z$, and we see that it is necessary to have $\lambda\leq\frac{1}{2}$, otherwise the multiplier will have an upper bound depending on $l$. Now, if $\lambda\leq\frac{1}{2}$, $C$ is given by (\ref{sit2}), and we have\begin{align*}
  \frac{k^2|l|}{pf}b^{-1}&=\frac{k^2|l|}{pf}\frac{\langle l\rangle^{\frac{1}{2}-\lambda}}{|k|^{\frac{1}{2}-\lambda}}\left(\frac{|l|^\frac{1}{2}}{\alpha^\frac{1}{4}|k|^\frac{1}{2}}+1\right)\alpha^\frac{1}{4}\frac{p^{\frac{1}{4}+\frac{\lambda}{2}}f^{\frac{1}{4}-\frac{\lambda}{2}}}{z^\frac{1}{4}}\\
  &\leq 2^{-\frac{1}{4}}\frac{|k|^{\frac{1}{2}+\lambda}\langle l\rangle^{2-\lambda}}{p^{1-\frac{\lambda}{2}}f^{\frac{3}{4}+\frac{\lambda}{2}}}+\frac{|k|^{1+\lambda}\langle l\rangle^{\frac{3}{2}-\lambda}}{p^{\frac{3}{4}-\frac{\lambda}{2}}f^{\frac{5}{4}+\frac{\lambda}{2}}}\\
  &\leq 2^{-\frac{1}{4}}\frac{|k|^{\frac{1}{2}+\lambda}}{f^{\frac{3}{4}+\frac{\lambda}{2}}}+\frac{1}{f^{\frac{3}{4}+\frac{\lambda}{2}}}.
    \end{align*}Therefore, $3$ and $5$ can also be absorbed by $\mathcal{M}_2$.
\end{itemize}We now claim that $\lambda\leq1/2$. Suppose first that all the linear parts can be absorbed properly, we then need to control the nonlinear terms. Consider the following term specifically\[
\langle\mathcal{Z}\bm{b}^{-1}N_2,\mathcal{Z}B_{\neq}\rangle,
\]where $N_2$ contains $U\cdot\bm{\nabla}_L U_3$, and we consider this term for simplicity\[
\langle\mathcal{Z}\bm{b}^{-1}(U\cdot\bm{\nabla}_L U_3)_{\neq},\mathcal{Z}B_{\neq}\rangle.
\]For the $(\neq,\neq)$ interaction, we first compute a lower bound for $b^{-1}$\[
b^{-1}=\frac{\langle l\rangle^{\frac{1}{2}-\lambda}}{|k|^{\frac{1}{2}-\lambda}}\left(1+\frac{\langle l\rangle^\frac{1}{2}}{\alpha^\frac{1}{4}|k|^\frac{1}{2}}\right)\frac{p^{\frac{1}{4}+\frac{\lambda}{2}}f^{\frac{1}{4}-\frac{\lambda}{2}}}{z^\frac{1}{4}}\lesssim\frac{\langle l\rangle}{|k|^{\frac{1}{2}-\lambda}}+\cdots.
\]Thus $b^{-1}$ behaves like $\partial_z$. In addition, one can verify that $b$ is an unbounded operator when $\lambda\neq\frac{1}{2}$. Consequently, $||U_{3,\neq}||_{H^s}\lesssim||B_{\neq}||_{H^{s+s'}}$ for some $s'>0$, i.e. additional regularity is required to bound $||U_{3,\neq}||_{H^s}$. Heuristically, we need to distribute more than two derivatives to the three unknowns, and it prevents us from closing the energy estimate. Besides, the multiplier $m$ also does not appear to absorb any derivative. As a result, the only reasonable choice is $\lambda=\frac{1}{2}$.\vspace{5pt}

From now on, we take $\lambda=\frac{1}{2}$. The multipliers $a$ and $b$ are then\[
a=\frac{\alpha^\frac{1}{4}}{\alpha^\frac{1}{4}+\frac{\langle l\rangle^\frac{1}{2}}{|k|^\frac{1}{2}}}\frac{f^\frac{1}{2}}{\left(\frac{z}{\alpha}\right)^\frac{1}{4}},\qquad b=\frac{\alpha^\frac{1}{4}}{\alpha^\frac{1}{4}+\frac{\langle l\rangle^\frac{1}{2}}{|k|^\frac{1}{2}}}\frac{\left(\frac{z}{\alpha}\right)^\frac{1}{4}}{p^\frac{1}{2}}
\]\begin{lemma}\label{somebounds}When $\lambda=\frac{1}{2}$, the following bounds hold,\begin{align*}
    &||\Psi_{\neq}||_{H^s}\lesssim ||A_{\neq}||_{H^s},\qquad ||U_{3,\neq}||_{H^s}\lesssim||B_{\neq}||_{H^s},\\
    &||U_{1,\neq}||_{H^s}+||U_{2,\neq}||_{H^s}\lesssim||A_{\neq}||_{H^s}+||B_{\neq}||_{H^s}+||\Xi_{\neq}||_{H^s},\\
    &1\lesssim m\lesssim \nu^{-\frac{\sqrt{2}}{3}}.
\end{align*}
\end{lemma}\vspace{5pt}

We can now control all the linear terms using the following bound\[
\left|\langle\mathcal{C}\mathcal{Z}A_{\neq},\frac{\mathcal{N}'}{\mathcal{N}}\mathcal{Z}B_{\neq}\rangle\right|\leq\frac{1}{\sqrt{2}}\left\|\sqrt{-\frac{\mathcal{N}'}{\mathcal{N}}}\mathcal{Z}A_{\neq}\right\|^2_{L^2}+\frac{1}{\sqrt{2}}\left\|\sqrt{-\frac{\mathcal{N}'}{\mathcal{N}}}\mathcal{Z}B_{\neq}\right\|^2_{L^2}.
\]As a result, the energy estimate (\ref{nonlinearenergyestimate}) now reduces to\begin{align*}
    \partial_t\mathbb{F}\leq&\ 2\varpi\nu^\frac{1}{3}\sum_{F\in\{A,B,\Xi\}}\left\|\mathcal{Z}F_{\neq}\right\|^2_{L^2}-\frac{1}{4}\sum_{F\in\{A,B,\Xi\}}\left\|\sqrt{-\frac{\mathcal{N}'}{\mathcal{N}}}ZF_{\neq}\right\|^2_{L^2}-\frac{\nu}{4}\sum_{F\in\{A,B,\Xi\}}\left\|\bm{\nabla}_LZF_{\neq}\right\|^2_{L^2}\\
    &-\langle\mathcal{Z}\bm{a}^{-1}N_{1,\neq},\mathcal{Z}A_{\neq}\rangle-\langle\mathcal{Z}\bm{b}^{-1}N_{2,\neq},\mathcal{Z}B_{\neq}\rangle-\langle\mathcal{Z}N_{3,\neq},\mathcal{Z}\Xi_{\neq}\rangle-\frac{1}{2}\langle\mathcal{C}\mathcal{Z}\bm{a}^{-1}N_{1,\neq},\mathcal{Z}B_{\neq}\rangle\\
    &-\frac{1}{2}\langle\mathcal{C}\mathcal{Z}\bm{b}^{-1}N_{2,\neq},\mathcal{Z}A_{\neq}\rangle
\end{align*}Integrating in time, we have\begin{align*}
   \sum_{F\in\{A,B,\Xi\}}\left\|\mathcal{Z}F_{\neq}\right\|^2_{L^\infty_tL^2}&\lesssim\sum_{F\in\{A,B,\Xi\}}\left\|\mathcal{Z}F_{\neq}(0)\right\|^2_{L^2}+\varpi\nu^\frac{1}{3}\sum_{F\in\{A,B,\Xi\}}\left\|\mathcal{Z}F_{\neq}\right\|^2_{L^2_tL^2}-\nu\sum_{F\in\{A,B,\Xi\}}\left\|\bm{\nabla}_L\mathcal{Z}F_{\neq}\right\|^2_{L^2_tL^2}\\&-\sum_{F\in\{A,B,\Xi\}}\left\|\sqrt{-\frac{\mathcal{N}'}{\mathcal{N}}}\mathcal{Z}F_{\neq}\right\|^2_{L^2_tL^2}+\int_0^\infty NL\ dt,
\end{align*}where \begin{align*}
 NL=   &-\langle\mathcal{Z}\bm{a}^{-1}N_{1,\neq},\mathcal{Z}A_{\neq}\rangle-\langle\mathcal{Z}\bm{b}^{-1}N_{2,\neq},\mathcal{Z}B_{\neq}\rangle-\langle\mathcal{Z}N_{3,\neq},\mathcal{Z}\Xi_{\neq}\rangle-\frac{1}{2}\langle\mathcal{C}\mathcal{Z}\bm{a}^{-1}N_{1,\neq},\mathcal{Z}B_{\neq}\rangle\\
    &-\frac{1}{2}\langle\mathcal{C}\mathcal{Z}\bm{b}^{-1}N_{2,\neq},\mathcal{Z}A_{\neq}\rangle.
\end{align*}To absorb the second term on the right-hand side of the inequality, we make use of the multiplier $M_3$\[
-\frac{M_3'}{M_3}=\frac{\nu^\frac{1}{3}k^2}{k^2+\nu^\frac{2}{3}(\eta-kt)^2}
\]This multiplier captures the enhanced dissipation effect and has already been widely used in many works e.g. \cite{9e0bb666-1e4e-3aa2-87b0-f916c2336491,bedrossian2018sobolev}. Due to the mean value theorem, it satisfies the following inequality
\[
   \nu^\frac{1}{6}\leq\sqrt{-\frac{M_3'}{M_3}}+\nu^{\frac{1}{2}}\left|\frac{\eta}{k}-t\right| 
    \]As a result, we can bound the second term as\begin{equation*}
  \varpi^\frac{1}{2} \nu^\frac{1}{6}\left\|\mathcal{Z}F_{\neq}\right\|_{L^2_tL^2}\lesssim \left\|\sqrt{-\frac{\mathcal{M}'_3}{\mathcal{M}_3}}\mathcal{Z}F_{\neq}\right\|_{L^2_tL^2}+\nu^\frac{1}{2}||\bm{\nabla}_L \mathcal{Z}F_{\neq}||_{L^2_tL^2}.
    \end{equation*}Therefore, taking $\varpi$ sufficiently small, we obtain the following inequality\begin{align}
&\sum_{F\in\{A,B,\Xi\}}\left\|\mathcal{Z}F_{\neq}\right\|_{L^\infty_tL^2}^2+\nu\sum_{F\in\{A,B,\Xi\}}\left\|\bm{\nabla}_L\mathcal{Z}F_{\neq}\right\|_{L^2_t L^2}^2+\sum_{F\in\{A,B,\Xi\}}\left\|\sqrt{-\frac{\mathcal{N}'}{\mathcal{N}}}\mathcal{Z}F_{\neq}\right\|_{L^2_tL^2}^2\nonumber\\\lesssim& \sum_{F\in\{A,B,\Xi\}}\left\|\mathcal{Z}F_{\neq}\ (0)\right\|^2_{L^2}+\int_0^\infty NL\ dt.\label{energyinequalityfornon}
    \end{align}

    In the remainder of the nonlinear analysis, our aim is to prove the bootstrap argument in Theorem \ref{boot}.

\subsection{\textbf{Non-zero modes analysis}}
In this section, we derive a bound for (\ref{energyinequalityfornon}). Recall that\begin{align*}
    \int_0^tNL(s)\ ds =&\int_0^t-\langle\mathcal{Z}\bm{a}^{-1}N_{1,\neq},\mathcal{Z}A_{\neq}\rangle-\langle\mathcal{Z}\bm{b}^{-1}N_{2,\neq},\mathcal{Z}B_{\neq}\rangle\\&-\langle\mathcal{Z}N_{3,\neq},\mathcal{Z}K_{\neq}\rangle-\frac{1}{2}\langle\mathcal{C}\mathcal{Z}\bm{a}^{-1}N_{1,\neq},\mathcal{Z}B_{\neq}\rangle\\
    &-\frac{1}{2}\langle\mathcal{C}\mathcal{Z}\bm{b}^{-1}N_{2,\neq},\mathcal{Z}A_{\neq}\rangle\ ds.
\end{align*}It suffices to consider the first three nonlinear terms since $\mathcal{C}$ is a bounded operator. Therefore, we denote\begin{align*}    J_1:=&\int_0^t\langle\mathcal{Z}\bm{a}^{-1}N_{1,\neq},\mathcal{Z}A_{\neq}\rangle\\J_2:=&\int_0^t\langle\mathcal{Z}\bm{b}^{-1}N_{2,\neq},\mathcal{Z}B_{\neq}\rangle\\J_3:=&\int_0^t\langle\mathcal{Z}N_{3,\neq},\mathcal{Z}\Xi_{\neq}\rangle
\end{align*}

\begin{lemma}\label{nonzeromodesbounds} Under the bootstrap assumptions of Theorem \ref{boot}, the following bounds hold for $J_i$, $i=1,2,3$,\begin{align}
    |J_1|&\lesssim\nu^{-\frac{\sqrt{2}}{3}-1}\varepsilon^3+\nu^{-\frac{\sqrt{2}}{3}-\frac{2}{3}-\Upsilon}\varepsilon^3,\label{J1}\\
    |J_2|&\lesssim\nu^{-\frac{\sqrt{2}}{3}-1}\varepsilon^3+\nu^{-\frac{\sqrt{2}}{3}-\frac{2}{3}-\Upsilon}\varepsilon^3\label{J2},\\
    |J_3|&\lesssim\nu^{-\frac{\sqrt{2}}{3}-1}\varepsilon^3+\nu^{-\frac{\sqrt{2}}{3}-\frac{2}{3}-\Upsilon}\varepsilon^3.\label{J3}
    \end{align}
\end{lemma}\vspace{9pt}

\begin{lemma}\textit{(Bounds on $a$ and $b$.)}\label{boundsonthemultipliers} The lower and upper bounds for $a$ and $b$ are given by
\begin{align}
\frac{1}{1+\left(\frac{1}{\alpha^\frac{1}{4}}+\frac{1}{\alpha^\frac{1}{2}}\right)\frac{\langle l\rangle^\frac{1}{2}}{|k|^\frac{1}{2}}+\frac{|l|}{\alpha^\frac{1}{2}|k|}}  \lesssim &a\lesssim 1\label{ma1}\\
\frac{1}{1+\frac{|l|}{|k|}+\frac{\langle l\rangle^\frac{1}{2}}{\alpha^\frac{1}{4}|k|^\frac{1}{2}}}\lesssim &b\lesssim 1+\frac{1}{\alpha^\frac{1}{4}},\qquad \frac{1}{1+\frac{\langle l\rangle^\frac{1}{2}f^\frac{1}{4}}{|k|}+\frac{|l|}{|k|}}\lesssim b.\label{mb1}
\end{align}
\end{lemma}

\noindent\textsc{Proof}. For (\ref{mb1}), the lower bound can be derived as follows\[
b^{-1}=\left(1+\frac{\langle l\rangle^\frac{1}{2}}{\alpha^\frac{1}{4}|k|^\frac{1}{2}}\right)\frac{p^\frac{1}{2}}{\left(\frac{z}{\alpha}\right)^\frac{1}{4}}\lesssim1+\frac{|l|}{|k|}+\frac{\langle  l\rangle^\frac{1}{2}(f^\frac{1}{2}+|l|)}{|k|^\frac{1}{2}z^\frac{1}{4}}.
\]The latter term can be bounded by either $\frac{\langle l\rangle^\frac{1}{2}}{\alpha^\frac{1}{4}|k|^\frac{1}{2}}+\frac{|l|}{|k|}$ or $\frac{\langle l\rangle^\frac{1}{2}f^\frac{1}{4}}{|k|}+\frac{|l|}{|k|}$. The other bounds can be derived similarly.\hfill \qedsymbol

\subsubsection*{\texorpdfstring{\textbf{Estimates for $J_1$}}{Estimates for J1}}
In this subsection, we prove (\ref{J1}). Recall that\[
N_1=U\cdot\bm{\nabla}_L\Psi.
\]We first consider the $(\neq,\neq)$ interaction. We have\begin{align*}
  \int_0^t \left| \langle\mathcal{Z}\bm{a}^{-1}\left(U_{\neq}\cdot(\bm{\nabla}_L \Psi)_{\neq})_{\neq},\mathcal{Z}A_{\neq}\right)\right|\  ds&=\int_0^t   \left| \langle\mathcal{Z}\left(U_{\neq}\cdot(\bm{\nabla}_L \Psi)_{\neq})_{\neq},\bm{a}^{-1}\mathcal{Z}A_{\neq}\right)\right|\ ds\\
   &\lesssim\nu^{-\frac{\sqrt{2}}{3}}\int_0^t\left\|\mathcal{Z}U_{\neq}\right\|_{L^2}\left\|\bm{\nabla}_L\bm{\nabla}_L\Psi_{\neq}\right\|_{L^2}\left\|\bm{\nabla}_L\mathcal{Z}A_{\neq}\right\|_{L^2}\ ds\\
   &\lesssim\nu^{-\frac{\sqrt{2}}{3}}\left\|\mathcal{Z}U_{\neq}\right\|_{L^\infty_tL^2}\left\|\bm{\nabla}_L\mathcal{Z}\Psi_{\neq}\right\|_{L^2_tL^2}\left\|\bm{\nabla}_L\mathcal{Z}A_{\neq}\right\|_{L^2_tL^2}\\
   &\lesssim\nu^{-\frac{\sqrt{2}}{3}-1}\varepsilon^3.
\end{align*}Here we used the bound $1\lesssim  m\lesssim\nu^{-\sqrt{2}/3}$ and (\ref{ma1}), and the algebra property of $H^s$, i.e. for $s>\frac{d}{2}$ \footnote{This inequality still holds with $\mathbb{R}$ replaced by $\mathbb{T}$.},\begin{align*}
||hg||_{H^s\left(\mathbb{R}^d\right)}&\lesssim||h||_{L^\infty\left(\mathbb{R}^d\right)}||g||_{H^s\left(\mathbb{R}^d\right)}+||h||_{H^s\left(\mathbb{R}^d\right)}||g||_{L^\infty\left(\mathbb{R}^d\right)}\\
&\lesssim||h||_{H^s\left(\mathbb{R}^d\right)}||g||_{H^s\left(\mathbb{R}^d\right)}.
    \end{align*}
    
Next, for the $(\neq,0)$ interaction, we have\begin{align*}
  \int_0^t \left| \langle\mathcal{Z}\bm{a}^{-1}\left(U_{\neq}\cdot(\bm{\nabla}_L \Psi)_{0})_{\neq},\mathcal{Z}A_{\neq}\right)\right|\  ds&=\int_0^t   \left| \langle\mathcal{Z}\left(U_{\neq}\cdot(\bm{\nabla}_L \Psi)_{0})_{\neq},\bm{a}^{-1}\mathcal{Z}A_{\neq}\right)\right|\ ds\\
   &\lesssim\nu^{-\frac{\sqrt{2}}{3}}\left\|\mathcal{Z}U_{2,3,\neq}\right\|_{L^\infty_tL^2}\left\|\bm{\nabla}_{y,z}\Psi_{0}\right\|_{L^2_tH^s}\left\|\bm{\nabla}_L\mathcal{Z}A_{\neq}\right\|_{L^2_tL^2}\\
   &\lesssim\nu^{-\frac{\sqrt{2}}{3}-1}\varepsilon^3.
\end{align*}The $(0,\neq)$ interaction should be considered more carefully, we write the integral as follows\begin{align*}
  \int_0^t \left| \langle\mathcal{Z}\bm{a}^{-1}\left(U_{0}\cdot(\bm{\nabla}_L \Psi)_{\neq})_{\neq},\mathcal{Z}A_{\neq}\right)\right|\  ds=&\int_0^t   \left| \langle\mathcal{Z}\left(U_{0}\cdot(\bm{\nabla}_L \Psi)_{\neq})_{\neq},\bm{a}^{-1}\mathcal{Z}A_{\neq}\right)\right|\ ds\\
  \lesssim&\int_0^t   \left| \langle\mathcal{Z}\left(U_{1,0}\cdot(\partial_X \Psi)_{\neq})_{\neq},\bm{a}^{-1}\mathcal{Z}A_{\neq}\right)\right|\ ds\\
  +&\int_0^t   \left| \langle\mathcal{Z}\left(U_{2,0}\partial_Y^L \Psi_{\neq}+U_{3,0}\partial_Z\Psi_{\neq})_{\neq},\bm{a}^{-1}\mathcal{Z}A_{\neq}\right)\right|\ ds\\
   \lesssim& T+\nu^{-\frac{\sqrt{2}}{3}}\left\|(U_2,U_3)_{0}\right\|_{L^\infty_tH^s}\left\|\bm{\nabla}_{L}\mathcal{Z}\Psi_{\neq}\right\|_{L^2_tL^2}\left\|\bm{\nabla}_L\mathcal{Z}A_{\neq}\right\|_{L^2_tL^2},
\end{align*}where \begin{align*}
    T&\lesssim\int_0^t   \left| \langle\mathcal{Z}\left(U_{1,0}\partial_X\Psi_{\neq})_{\neq},\mathcal{Z}A_{\neq}\right)\right|\ ds+\int_0^t   \left| \langle\mathcal{Z}\left(U_{1,0}|\partial_X|^\frac{1}{2}\Psi_{\neq})_{\neq},|\partial_Z|^\frac{1}{2}\mathcal{Z}A_{\neq}\right)\right|\ ds\\&+\int_0^t   \left| \langle\mathcal{Z}\left(U_{1,0}\Psi_{\neq})_{\neq},|\partial_Z|\mathcal{Z}A_{\neq}\right)\right|\ ds,\\
    &\lesssim\nu^{-\frac{\sqrt{2}}{3}}||\bm{\nabla}_L\mathcal{Z}A_{\neq}||_{L^2_tL^2}||\mathcal{Z}A_{\neq}||_{L^2_tL^2}\left(\left\|\overline{U}_1\right\|_{L^\infty_tH^s}+\left\|\widetilde{U}_{1}\right\|_{L^\infty_tH^s}\right)\\
    &\lesssim\nu^{-\frac{\sqrt{2}}{3}-\frac{2}{3}}\varepsilon^2\left(\nu^{-\Upsilon}\varepsilon+\varepsilon\right),
\end{align*}where we have used (\ref{ma1}). By assuming $\Upsilon\leq\frac{1}{3}$, we see that (\ref{J1}) holds.

\subsubsection*{\texorpdfstring{\textbf{Estimates for $J_2$}}{Estimates for J2}}
We now prove estimate (\ref{J2}) here. Recall that\[
N_2=-\partial_Z\Delta_L^{-1}\left(\partial_i^L U_j \partial^L_j U_i\right)+U\cdot\bm{\nabla}_L U_3
\]Again, we first consider the $(\neq,\neq)$ interaction for $U\cdot\bm{\nabla}_L U_3$, i.e. the term\[
\int_0^t \langle\mathcal{Z}\bm{b}^{-1}\left(U_{\neq}\cdot\bm{\nabla}_L (\bm{b}B_{\neq}))_{\neq},\mathcal{Z}B_{\neq}\right)\ ds.
\]We have\begin{align*}
   &\left| \langle\mathcal{Z}\left(U_{\neq}\cdot\bm{\nabla}_L (U_{3,\neq}))_{\neq},\bm{b}^{-1}\mathcal{Z}B_{\neq}\right)\right|\lesssim\nu^{-\frac{\sqrt{2}}{3}}\left\|\bm{\nabla}_L\mathcal{Z}B_{\neq}\right\|_{L^2_tL^2}^2\left\|\mathcal{Z}U_{\neq}\right\|_{L^\infty_tL^2},
\end{align*}where we used (\ref{mb1}) and the fact that $|\widehat{U}_{3,\neq}|\lesssim|\widehat{B}_{\neq}|$. Therefore\begin{align*}
\left|\int_0^t \langle\mathcal{Z}\bm{b}^{-1}\left(U_{\neq}\cdot\bm{\nabla}_L (\bm{b}B_{\neq})\right)_{\neq},\mathcal{Z}B_{\neq}\rangle\ ds\right|&\lesssim \nu^{-\frac{\sqrt{2}}{3}}||\bm{\nabla}_L\mathcal{Z}(U,B)_{\neq}||_{L^2_tL^2}^2||\mathcal{Z}(U,B)_{\neq}||_{L^\infty_tL^2}\lesssim\nu^{-\frac{\sqrt{2}}{3}-1}\varepsilon^3.
\end{align*}

For the $(\neq,\neq)$ interaction of $\partial_Z\Delta^{-1}_L\left(\partial_i^L U_j\partial_j^L U_i\right)$, the bound can be  derived similarly,\begin{align*}
  \int_0^t \left| \langle\mathcal{Z}\partial_Z\Delta^{-1}_L\bm{\nabla}_L\cdot\left((U_{\neq}\cdot\bm{\nabla}_L) U_{\neq}\right)_{\neq},\bm{b}^{-1}\mathcal{Z}B_{\neq}\rangle\right|=&\int_0^t\left| \langle\mathcal{Z}\left((U_{\neq}\cdot\bm{\nabla}_L) U_{\neq}\right)_{\neq},\bm{\nabla}_L\partial_Z\bm{b}^{-1}\Delta^{-1}_L\mathcal{Z}B_{\neq}\rangle\right|\\
   \lesssim&\nu^{-\frac{\sqrt{2}}{3}}\int_0^t\left\|\bm{\nabla}_L\mathcal{Z}B_{\neq}\right\|_{L^2}\left\|\mathcal{Z}U_{\neq}\right\|_{L^2}\left\|\bm{\nabla}_L\mathcal{Z}U_{\neq}\right\|_{L^2},\\
   \lesssim&\nu^{-\frac{\sqrt{2}}{3}}\left\|\bm{\nabla}_L\mathcal{Z}B_{\neq}\right\|_{L^2_tL^2}^2\left\|\mathcal{Z}U_{\neq}\right\|_{L^\infty_tL^2}\\
\lesssim&\nu^{-\frac{\sqrt{2}}{3}-1}\varepsilon^3.
\end{align*}

We now deal with the $(\neq,0)$ and $(0,\neq)$ interactions. We consider the second nonlinearity first and split the integral into the following two parts:\begin{align*}
    I_1&=\int_0^t\langle\mathcal{Z}(U_{\neq}\cdot(\bm{\nabla}_L U_3)_0)_{\neq},\bm{b}^{-1}\mathcal{Z}B_{\neq}\rangle\ ds,\\
    I_2&=\int_0^t\langle\mathcal{Z}(U_{0}\cdot(\bm{\nabla}_L U_3)_{\neq})_{\neq},\bm{b}^{-1}\mathcal{Z}B_{\neq}\rangle\ ds,
\end{align*}

Consider $I_1$, since $U_{\neq}\cdot(\bm{\nabla}_L U_3)_0=U_{2,\neq}\partial_y U_{3,0}+U_{3,\neq}\partial_z U_{3,0}$, the component $U_{1,0}$ does not contribute to this term. For $j=2,3$, define\[
I_{1j}=\int_0^t\langle\mathcal{Z} (U_{j,\neq}\partial_j U_{3,0})_{\neq},\bm{b}^{-1}\mathcal{Z}B_{\neq}\rangle
\]The bound is obtained as follows,\begin{align*}
|I_{1j}|&\lesssim \nu^{-\frac{\sqrt{2}}{3}}\int_0^t\left\|\bm{\nabla}_L\mathcal{Z}B_{\neq}\right\|_{L^2}\left\|\mathcal{Z}U_{j,\neq}\right\|_{L^2}||\partial_j U_{3,0}||_{H^s},\\
&\lesssim\nu^{-\frac{\sqrt{2}}{3}}\left\|\bm{\nabla}_L\mathcal{Z}B_{\neq}\right\|_{L^2_tL^2}\left\|\mathcal{Z}U_{j,\neq}\right\|_{L^\infty_tL^2}||\partial_j U_{3,0}||_{L^2_tH^s},\\
&\lesssim\nu^{-\frac{\sqrt{2}}{3}-1}
\varepsilon^3.\end{align*}

 For $I_2$, we first consider the term involving $U_{1,0}$. Let\begin{align*}
|I_{21}|&=\left|\int_0^t\langle\mathcal{Z}(U_{1,0}\partial_X U_{3,\neq})_{\neq},\bm{b}^{-1}\mathcal{Z}B_{\neq}\rangle\ ds\right|\\
&\lesssim\int_0^t|\langle\mathcal{Z}(U_{1,0}\partial_X U_{3,\neq})_{\neq},\mathcal{Z}B_{\neq}\rangle\ |+\langle\mathcal{Z}(U_{1,0}\partial_X U_{3,\neq})_{\neq},\left|\partial_X\right|^{-1}\left|\bm{\nabla}_L\right|\mathcal{Z}B_{\neq}\rangle\ |\ ds\\
&\lesssim I_{211}+I_{212},
\end{align*}where we have used the property (\ref{mb1}) that $b^{-1}\lesssim 1+\frac{p^\frac{1}{2}}{|k|}$ . For $I_{212}$, the above becomes\[
I_{212}\approx\int_0^t|\langle\mathcal{Z}(U_{1,0}U_{3,\neq})_{\neq},\bm{\nabla}_L\mathcal{Z}B_{\neq}\rangle|\ ds.
\]We bound the integral as follows\begin{align*}
I_{212}&\lesssim \nu^{-\frac{\sqrt{2}}{3}} \int_0^t\left\|\bm{\nabla}_L\mathcal{Z}B_{\neq}\right\|_{L^2}\left\|\mathcal{Z}B_{\neq}\right\|_{L^2}\left\|U_{1,0}\right\|_{H^s}\ ds\\
   &\lesssim \nu^{-\frac{\sqrt{2}}{3}}\left\|\bm{\nabla}_L\mathcal{Z}B_{\neq}\right\|_{L^2_tL^2}\left\|\mathcal{Z}B_{\neq}\right\|_{L^2_tL^2}\left\|U_{1,0}\right\|_{L^\infty_tH^s}\\
   &\lesssim\nu^{-\frac{\sqrt{2}}{3}-\frac{2}{3}-\Upsilon}\varepsilon^3.
\end{align*}

Next, we bound $I_{211}$ in a similar way,\begin{align*}
    I_{211}&\lesssim\nu^{-\frac{\sqrt{2}}{3}}\int_0^t\left\|\mathcal{Z}B_{\neq}\right\|_{L^2}||\bm{\nabla}_L\mathcal{Z}U_{3,\neq}||_{L^2}||U_{1,0}||_{H^s}\ ds\\
   &\lesssim \nu^{-\frac{\sqrt{2}}{3}}\left\|\bm{\nabla}_L\mathcal{Z}B_{\neq}\right\|_{L^2_tL^2}\left\|\mathcal{Z}B_{\neq}\right\|_{L^2_tL^2}\left\|U_{1,0}\right\|_{L^\infty_tH^s},\\
   &\lesssim\nu^{-\frac{\sqrt{2}}{3}-\frac{2}{3}-\Upsilon}\varepsilon^3.
\end{align*}The terms containing $U_{2,0}$ and $U_{3,0}$ are bounded as follows,\begin{align*}
\left|\int_0^t\langle\mathcal{Z}(U_{2,0}(\partial_{Y}^L U_{3})_{\neq}+U_{3,0}\partial_ZU_{3,\neq})_{\neq},\bm{b}^{-1}\mathcal{Z}B_{\neq}\rangle\right|\lesssim& \nu^{-\frac{\sqrt{2}}{3}}\int_0^t\left\|\bm{\nabla}_L\mathcal{Z}B_{\neq}\right\|_{L^2}^2\left(||U_{2,0}||_{H^s}+||U_{3,0}||_{H^s}\right)\\
    \lesssim&\nu^{-\frac{\sqrt{2}}{3}}\left\|\bm{\nabla}_L\mathcal{Z}B_{\neq}\right\|_{L^2_tL^2}^2\left(||U_{2,0}||_{L^\infty_tH^s}+||U_{3,0}||_{L^\infty_tH^s}\right)\\
    \lesssim&\nu^{-\frac{\sqrt{2}}{3}-1}\varepsilon^3.
\end{align*}

Finally, we compute the bounds for $(\neq,0)$ and $(0,\neq)$ interactions in the pressure term. We have\begin{align*}
    \int_0^t\left| \langle\mathcal{Z}\left((U_{\neq}\cdot\bm{\nabla}_{y,z}) U_{0}\right)_{\neq},\bm{\nabla}_L\partial_Z\bm{b}^{-1}\Delta^{-1}_L\mathcal{Z}B_{\neq}\rangle\right|\lesssim&\nu^{-\frac{\sqrt{2}}{3}}\int_0^t||\bm{\nabla}_L\mathcal{Z}B_{\neq}||_{L^2}\left\|\bm{\nabla}_{y,z}(U_{2},U_3)_0\right\|_{H^s}||\mathcal{Z}U_{\neq}||_{L^2}\\
    +&\nu^{-\frac{\sqrt{2}}{3}}\int_0^t||\mathcal{Z}B_{\neq}||_{L^2}\left\|\mathcal{Z}U_{\neq}\right\|_{L^2}||\bm{\nabla}_{y,z}U_{1,0}||_{H^s}\\
    \lesssim&\nu^{-\frac{\sqrt{2}}{3}-1}\varepsilon^3+\nu^{-\frac{\sqrt{2}}{3}-\frac{2}{3}-\Upsilon}\varepsilon^3,\\
 \int_0^t\left| \langle\mathcal{Z}\left((U_{0}\cdot\bm{\nabla}_{L}) U_{\neq}\right)_{\neq},\bm{\nabla}_L\partial_Z\bm{b}^{-1}\Delta^{-1}_L\mathcal{Z}B_{\neq}\rangle\right|\lesssim&\nu^{-\frac{\sqrt{2}}{3}}\int_0^t||\bm{\nabla}_L\mathcal{Z}B_{\neq}||_{L^2}\left\|\mathcal{Z}U_{\neq}\right\|_{L^2}||U_{1,0}||_{H^s}\\
 +& \nu^{-\frac{\sqrt{2}}{3}}\int_0^t||\bm{\nabla}_L\mathcal{Z}U_{\neq}||_{L^2}\left\|\mathcal{Z}B_{\neq}\right\|_{L^2}||U_{1,0}||_{H^s}\\
 +&\nu^{-\frac{\sqrt{2}}{3}}\int_0^t||\bm{\nabla}_L\mathcal{Z}B_{\neq}||_{L^2}\left\|\bm{\nabla}_{L}\mathcal{Z}U_{\neq}\right\|_{L^2}\left\|(U_2,U_3)_0\right\|_{H^s}\\
    \lesssim&\nu^{-\frac{\sqrt{2}}{3}-1}\varepsilon^3+\nu^{-\frac{\sqrt{2}}{3}-\frac{2}{3}-\Upsilon}\varepsilon^3,
\end{align*}which gives (\ref{J2}).

\subsubsection*{\texorpdfstring{\textbf{Estimates for $J_3$}}{Estimates for J3}}
Now we have\[
 N_3=\partial_X((U\cdot\bm{\nabla}_L)U_2)-\partial_{Y}^L((U\cdot\bm{\nabla}_L) U_1)-\frac{1}{\sqrt{\alpha}}\partial_Z((U\cdot\bm{\nabla}_L)\Psi)
 \]In this case, the $(\neq,\neq)$ interaction yields the bound\begin{align*}
&\left|\int_0^t\langle \mathcal{Z}N_{3,(\neq,\neq)},\mathcal{Z}\Xi_{\neq}\rangle\right|\\&\lesssim  ||\bm{\nabla}_L\mathcal{Z}\Xi_{\neq}||_{L^2_tL^2}\left(\left\|\bm{\nabla}_L\mathcal{Z}\left(U_1,U_2\right)_{\neq}\right\|_{L^2_tL^2}+\alpha^{-\frac{1}{2}}\left\|\bm{\nabla}_L\mathcal{Z}\Psi_{\neq}\right\|_{L^2_tL^2}\right)||\mathcal{Z}U_{\neq}||_{L^\infty_tL^2},\\
&\lesssim  \nu^{-\frac{\sqrt{2}}{3}-1}\varepsilon^3.
 \end{align*}For the $(0,\neq)$ interaction, notice that (\ref{initialdata229}) gives\[
|m|\lesssim\left(\nu^{-\frac{2}{3}}\frac{k^2+l^2}{k^2\nu^{-\frac{2}{3}}+l^2}\right)^\frac{1}{\sqrt{2}}.
\]We thus have the combined bound \[
|km|\lesssim|k|\left(\nu^{-\frac{2}{3}}\cdot\frac{k^2+l^2}{k^2\nu^{-\frac{2}{3}}+l^2}\right)^{\frac{1}{\sqrt{2}}-g}\left(\nu^{-\frac{2}{3}}\cdot\frac{k^2+l^2}{k^2\nu^{-\frac{2}{3}}}\right)^{g}.
\]Taking $g=\frac{1}{2}$, we have\[
|km|\lesssim\nu^{-\frac{\sqrt{2}-1}{3}}\sqrt{k^2+l^2}.
\]That is, the multiplier $m$ can absorb some $x$-derivatives.  Consequently,\begin{align*}
\left|\int_0^t\langle \mathcal{Z}N_{3,(0,\neq)},\mathcal{Z}\Xi_{\neq}\rangle\right|\lesssim&\nu^{-\frac{\sqrt{2}-1}{3}}\left\|\overline{U}_1\right\|_{L^\infty_t H^s}\left\|\bm{\nabla}_{L}\mathcal{Z}\left(U_2,U_3,\Psi\right)_\neq\right\|_{L^2_tL^2}||\bm{\nabla}_L\mathcal{Z}\Xi_{\neq}||_{L^2_tL^2}\\
+&\nu^{-\frac{\sqrt{2}-1}{3}}\left\|\bm{\nabla}_{y,z}\overline{U}_1\right\|_{L^2_t H^s}\left\|\mathcal{Z}\left(U_2,U_3,\Psi\right)_\neq\right\|_{L^\infty_tL^2}||\bm{\nabla}_L\mathcal{Z}\Xi_{\neq}||_{L^2_tL^2}\\
+&\nu^{-\frac{\sqrt{2}}{3}}\left(\left\|(U_2,\Psi)_0\right\|_{L^\infty_t H^s}+\left\|\widetilde{U}_1\right\|_{L^\infty_tH^s}\right)\left\|\bm{\nabla}_{L}\mathcal{Z}\left(U_2,U_3,\Psi\right)_\neq\right\|_{L^2_tL^2}||\bm{\nabla}_L\mathcal{Z}\Xi_{\neq}||_{L^2_tL^2}
\\
+&\nu^{-\frac{\sqrt{2}}{3}}\left(\left\|\bm{\nabla}_{y,z}(U_2,\Psi)_0\right\|_{L^2_t H^s}+\left\|\partial_y\widetilde{U}_1\right\|_{L^2_tH^s}\right)\left\|\mathcal{Z}\left(U_2,U_3,\Psi\right)_\neq\right\|_{L^\infty_tL^2}||\bm{\nabla}_L\mathcal{Z}\Xi_{\neq}||_{L^2_tL^2}\\
\lesssim&\nu^{-\frac{\sqrt{2}}{3}-\frac{2}{3}-\Upsilon}\varepsilon^3+\nu^{-\frac{\sqrt{2}}{3}-1}\varepsilon^3.
\end{align*}

For the $(\neq,0)$ interaction, the most problematic term is $(U_{2,\neq}\partial_y\overline{U}_1+U_{3,\neq}\partial_z\overline{U}_1)_{\neq}$. Notice that $\partial_y U_1$ does not experience lift-up as $\partial_y$ can balance the singular operator $\partial^{-1}_y$. Therefore, we only need to worry about $(U_{3,\neq}\partial_z\overline{U}_1)_{\neq}$. Integrating by parts, we have\begin{align*}
    &\langle\mathcal{Z}(U_{3,\neq}\partial_z\overline{U}_1)_{\neq},\partial_Y^L\mathcal{Z}\Xi_{\neq}\rangle=-\langle\mathcal{Z}(\partial_ZU_{3,\neq}\overline{U}_1)_{\neq},\partial_Y^L\mathcal{Z}\Xi_{\neq}\rangle-\langle\mathcal{Z}(U_{3,\neq}\overline{U}_1)_{\neq},\partial_Y^L\partial_Z\mathcal{Z}\Xi_{\neq}\rangle\\=&-\langle\mathcal{Z}(\partial_ZU_{3,\neq}\overline{U}_1)_{\neq},\partial_Y^L\mathcal{Z}\Xi_{\neq}\rangle+\langle\mathcal{Z}(\partial_Y^LU_{3,\neq}\overline{U}_1)_{\neq},\partial_Z\mathcal{Z}\Xi_{\neq}\rangle+\langle\mathcal{Z}(U_{3,\neq}\partial_y\overline{U}_1)_{\neq},\partial_Z\mathcal{Z}\Xi_{\neq}\rangle.
\end{align*}Similarly, we can bound the $(\neq,0)$ interaction in the following way:\begin{align}
    \left|\int_0^t\langle\mathcal{Z}\mathcal{N}_{3,(\neq,0)},\mathcal{Z}\Xi_{\neq}\rangle\right|&\lesssim\nu^{-\frac{\sqrt{2}}{3}-1}\varepsilon^3+\nu^{-\frac{\sqrt{2}}{3}}\int_0^t\left\|\bm{\nabla}_L\mathcal{Z}\Xi_{\neq}\right\|_{L^2}\left\|\mathcal{Z}(U_{2},U_3)_{\neq}\right\|_{L^2}\left\|\partial_y\overline{U}_1\right\|_{H^s}\nonumber\\
    &+\nu^{-\frac{\sqrt{2}}{3}}\int_0^t\left\|\bm{\nabla}_L\mathcal{Z}\Xi_{\neq}\right\|_{L^2}\left\|\mathcal{Z}(\bm{\nabla}_{L}U_{3,\neq}\overline{U}_{1})_{\neq}\right\|_{L^2\nonumber}\\
    &\lesssim \nu^{-\frac{\sqrt{2}}{3}-1}\varepsilon^3+\nu^{-\frac{\sqrt{2}}{3}}\int_0^t||\bm{\nabla}_L\mathcal{Z}\Xi_{\neq}||_{L^2}\left(||\mathcal{Z}U_{3,\neq}||_{L^2}\left\|\overline{U}_{1}\right\|_{H^s}+||\bm{\nabla}_L\mathcal{Z}U_{3,\neq}||_{L^2}\left\|\overline{U}_{1}\right\|_{L^\infty}\right)\nonumber\\
    &\lesssim\nu^{-\frac{\sqrt{2}}{3}-1}\varepsilon^3+\nu^{-\frac{\sqrt{2}}{3}-\frac{2}{3}}\left\|\overline{U}_{1}\right\|_{L^\infty_t H^s}\varepsilon^2+\nu^{-\frac{\sqrt{2}}{3}-1}\varepsilon^2\left\|\overline{U}_{1}\right\|_{L^\infty_tL^\infty}.\label{largebound}
\end{align}By Lemma \ref{inftyonu1} and the bootstrap assumptions of Theorem \ref{boot}, inequality (\ref{energyinequalityfornon}) becomes\begin{align*}
\sum_{F\in\{A,B,K\}}\left\|\mathcal{Z}F_{\neq}\right\|_{L^\infty_tL^2}^2+&\nu\sum_{F\in\{A,B,K\}}\left\|\bm{\nabla}_L\mathcal{Z}F_{\neq}\right\|_{L^2_t L^2}^2+\sum_{F\in\{A,B,K\}}\left\|\sqrt{-\frac{\mathcal{N}'}{\mathcal{N}}}\mathcal{Z}F_{\neq}\right\|_{L^2_tL^2}^2\nonumber\\\lesssim& \varepsilon^2+\nu^{-\frac{\sqrt{2}}{3}-1}\varepsilon^3+\nu^{-\frac{\sqrt{2}}{3}-\frac{2}{3}-\Upsilon}\varepsilon^3.
    \end{align*}

\vspace{9pt}

\begin{remark} In principle, $N_3$ can be simplified by expressing some of the terms using $\Xi$. However, this does not seem to help improve the estimate and would require additional zero mode analysis for $\Xi_0$, which is difficult as the unknowns $A$ and $B$ are not well-defined when $k=0$. As a result, we do not simplify the nonlinearity so that $\Xi$ is an auxiliary unknown only for the non-zero-mode analysis.\end{remark}

\subsection{\textbf{Zero modes analysis}}
In this section, we focus on the nonlinear stability \textit{in the absence of the linear lift-up effect}. That is, we prove Theorem \ref{nonlinear}. The proof for Theorem \ref{nonlinear1} is similar and will be discussed in Section \ref{nonlinearforlinearliftup}
\subsubsection{\textbf{The simple zero modes system}}
In this part, we study the simple zero mode system. The nonlinear system of the zero modes is written as \begin{flalign}
\left\{
\begin{aligned}
&\partial_t U_{1,0}=-U_{2,0}+\nu\Delta_{y,z}U_{1,0}-(U\cdot\bm{\nabla}_LU_1)_0,\\
&\partial_tU_{2,0}=-\partial_Y P_0+\nu\Delta U_{2,0}-(U\cdot\bm{\nabla}_LU_2)_0,\\
&\partial_t U_{3,0}=-\partial_Z P_0+\sqrt{\alpha}\Psi_0+\nu\Delta U_{3,0}-(U\cdot\bm{\nabla}_LU_3)_0,\\
&\partial_t\Psi_0=-\sqrt{\alpha} U_{3,0}+\nu\Delta\Psi_0-(U\cdot\bm{\nabla}_L\Psi)_0,\\
&\partial_y U_{2,0}+\partial_z U_{3,0}=0.
\end{aligned}
\right.\label{nonliearzeromodessys}
\end{flalign}The simple zero mode system then reduces to\begin{flalign}
\left\{
\begin{aligned}
&\partial_t \overline{U}_{1}=-\overline{U}_{2}+\nu\Delta_{y,z}\overline{U}_{1}-\overline{(U\cdot\bm{\nabla}_LU_1)},\\
&\partial_t\overline{U}_{2}=-\partial_Y \overline{P}+\nu\Delta\overline{U}_{2}-\overline{(U\cdot\bm{\nabla}_LU_2)},\\
&\partial_t \overline{U}_{3}=-\partial_Z \overline{P}+\sqrt{\alpha}\overline{\Psi}+\nu\Delta \overline{U}_{3}-\overline{(U\cdot\bm{\nabla}_LU_3)},\\
&\partial_t\overline{\Psi}=-\sqrt{\alpha} \overline{U}_{3}+\nu\Delta\overline{\Psi}-\overline{(U\cdot\bm{\nabla}_L\Psi)},\\
&\partial_y \overline{U}_{2}+\partial_z \overline{U}_{3}=0.
\end{aligned}
\right.\label{nonliearsimzero}
\end{flalign}

To analyze the simple-zero-mode system, we define the following new unknowns:\begin{itemize}
\item  \textsc{The vorticity}:\[
S:=\partial_z  \overline{U}_2-\partial_y  \overline{U}_3.
\]\item \textsc{Diagonalized variables}:\[
J_{\pm}=S\pm |\bm{\nabla}_{y,z}|\overline{\Psi}.
\]
\end{itemize}Thus, by incompressibility, we find that\begin{align*}
    \overline{U}_2=\partial_z\Delta^{-1}_{y,z}S,\qquad \overline{U}_3=-\partial_y\Delta^{-1}_{y,z}S,
\end{align*}and that\[
    S=\frac{J_++J_-}{2},\qquad \Psi=|\bm{\nabla}_{y,z}|^{-1}\frac{J_+-J_-}{2}.
\]

Thanks to these new unknowns, we are able to derive the explicit expressions for $ \overline{U}_2, \overline{U}_3$ and $ \overline{\Psi}$. It can be shown that $J_{\pm}$, $\Xi$ and $\overline{U}_1$ solve the following system\begin{flalign*}
\left\{
\begin{aligned}
&\partial_t \overline{U}_1-\mathcal{L}_1\overline{U}_1=-\partial_z\Delta^{-1}_{y,z}S+R_1,\\
&\partial_t J_{\pm}-\mathcal{L}_\pm J_{\pm}=R_\pm.
\end{aligned}
\right.
\end{flalign*}Here\[
\mathcal{L}_1:=\nu\Delta_{y,z},\qquad \mathcal{L}_{\pm}:=\nu\Delta_{y,z} \mp \sqrt{\alpha}\partial_y|\bm{\nabla}_{y,z}|^{-1}:=\mathcal{L}_1\mp\sqrt{\alpha}\mathcal{R}_y
\]and the nonlinearities are given by\begin{align*}
    R_1&=-\overline{(U\cdot\bm{\nabla}_LU_1)} ,\\
    R_{\pm}&=-\overline{\partial_z(U\cdot\bm{\nabla}_LU_2)}+\overline{\partial_y(U\cdot\bm{\nabla}_LU_3)}\mp|\bm{\nabla}_{y,z}|\overline{(U\cdot\bm{\nabla}_L\Psi)}.
\end{align*}The solutions can be expressed via the Duhamel formula, i.e.\begin{align}
\overline{U}_1(t,y,z)&=e^{\mathcal{L}_1t} \overline{U}_1(0,y,z)+\int_0^t e^{\mathcal{L}_1(t-s)} \left(-\partial_z\Delta^{-1}_{y,z}S(s,y,z)+R_1(s,y,z)\right)\ ds,\label{duhamel1}\\
    J_{\pm}(t,y,z)&=e^{\mathcal{L}_\pm t}J_{\pm}(0,y,z)+\int_0^t e^{\mathcal{L}_\pm (t-s)}R_{\pm}(s,y,z)\ ds.\label{duhamel2}
\end{align}Unfortunately, due to the presence of $\overline{U}_2$ in the $\overline{U}_1$ equation, it is not possible to apply the energy estimate directly to $\overline{U}_1$. Therefore, to study the dynamics of $\overline{U}_1$, we have to make use of the Duhamel formula. The linear term $-\partial_z\Delta^{-1}S$ in the integral can be computed further. We have\begin{align*}
\int_0^t e^{\mathcal{L}_1(t-s)}\partial_z\Delta^{-1}_{y,z}S\ ds=&\partial_z\Delta^{-1}_{y,z}\int_0^t e^{\mathcal{L}_1(t-s)}\frac{J_++J_-}{2}\ ds,\\
=&\frac{1}{2}\partial_z\Delta^{-1}_{y,z}e^{\mathcal{L}_1 t}\int_0^t e^{(\mathcal{L}_+-\mathcal{L}_1)s}J_+(0)+e^{(\mathcal{L}_--\mathcal{L}_1)s}J_-(0)\ ds\\
+&\frac{1}{2}\partial_z\Delta^{-1}_{y,z}e^{\mathcal{L}_1 t}\int_0^t\int_0^s e^{-\mathcal{L}_1s}e^{\mathcal{L}_+(s-\tau)}R_+(\tau)+e^{-\mathcal{L}_1 s}e^{\mathcal{L}_-(s-\tau)}R_-(\tau)\ d\tau ds.
\end{align*}The second integral can be simplified by using Fubini's theorem. Hence the full Duhamel formula for $\overline{U}_1$ reads
\begin{align}
   \overline{U}_1(t,y,z)&=e^{\mathcal{L}_1t} \overline{U}_1(0,y,z)+\frac{1}{2\sqrt{\alpha}}\partial_z\Delta^{-1}_{y,z}e^{\mathcal{L}_1 t}\mathcal{R}^{-1}_y\left(e^{-\sqrt{\alpha}t\mathcal{R}_y}-I_d\right)J_+(0,y,z)\nonumber\\&-\frac{1}{2\sqrt{\alpha}}\partial_z\Delta^{-1}_{y,z}e^{\mathcal{L}_1 t}\mathcal{R}^{-1}_y\left(e^{\sqrt{\alpha}t\mathcal{R}_y}-I_d\right)J_-(0,y,z)\nonumber\\
    &+\int_0^t e^{\mathcal{L}_1(t-s)} R_1(s,y,z)\ ds\nonumber\\
    &-\frac{1}{2}\partial_z\Delta_{y,z}^{-1}\left[-\frac{e^{\mathcal{L}_+t}}{\sqrt{\alpha}}\mathcal{R}_y^{-1}\int_0^t e^{-\mathcal{L}_+s}R_+(s)\ ds+\frac{e^{\mathcal{L}_1t}}{\sqrt{\alpha}}\mathcal{R}_y^{-1}\int_0^t e^{-\mathcal{L}_1s}R_+(s)\ ds\right]\nonumber\\
    &-\frac{1}{2}\partial_z\Delta_{y,z}^{-1}\left[\frac{e^{\mathcal{L}_-t}}{\sqrt{\alpha}}\mathcal{R}_y^{-1}\int_0^t e^{-\mathcal{L}_-s}R_-(s)\ ds-\frac{e^{\mathcal{L}_1t}}{\sqrt{\alpha}}\mathcal{R}_y^{-1}\int_0^t e^{-\mathcal{L}_1s}R_-(s)\ ds\right].\label{explicitsoluU1}
\end{align}

\subsubsection{\texorpdfstring
  {Bounds for $\overline{U}_2$, $\overline{U}_3$ and $\overline{\Psi}$}
  {Bounds for U2, U3 and Psi}}

At the nonlinear level, the dispersion observed in the linear analysis helps to improve the threshold. In this part, we derive a dispersive bound for the simple zero modes $\overline{U}_2,\overline{U}_3$ and $\overline{\Psi}$. In particular, we prove the following.

\begin{proposition}\label{nonlinear dispersive} Under the bootstrap assumptions of Theorem \ref{boot}, the following dispersive estimate holds.
    \[
\left\|\overline{U}_2\right\|_{W^{1,\infty}}+\left\|\overline{U}_3\right\|_{W^{1,\infty}}+\left\|\overline{\Psi}\right\|_{W^{1,\infty}}\lesssim  \frac{e^{-\nu t}}{\langle \alpha^\frac{1}{2}t\rangle^\frac{1}{3}}\varepsilon+\alpha^{-\frac{1}{6}}\nu^{-\frac{2}{3}}\varepsilon^2.
\]
\end{proposition}

\noindent\textsc{Proof}.  Here we only prove the $L^\infty$ bound, as the $\dot{W}^{1,\infty}$ bound can be similarly deduced. The first term on the right-hand side directly follows from the linear estimate \ref{dis1}. As a result, we see that\begin{align*}
\left\|\overline{U}_2\right\|_{L^\infty}+\left\|\overline{U}_3\right\|_{L^\infty}+\left\|\overline{\Psi}\right\|_{L^\infty}&\lesssim \frac{e^{-\nu t}}{\langle\alpha^\frac{1}{2}t\rangle^\frac{1}{3}}\varepsilon+\left\|(\partial_z,\partial_y)\Delta^{-1}_{y,z}\int_0^t e^{\mathcal{L}_+(t-s)}R_++e^{\mathcal{L}_-(t-s)}R_-\ ds\right\|_{L^\infty}\\
    &+\left\||\bm{\nabla}_{y,z}|^{-1}\int_0^t e^{\mathcal{L}_+(t-s)}R_+-e^{\mathcal{L}_-(t-s)}R_-\ ds\right\|_{L^\infty}\\
    &\lesssim \frac{e^{-\nu t}}{\langle\alpha^\frac{1}{2}t\rangle^\frac{1}{3}}\varepsilon+\int_0^t \frac{e^{-\nu(t-s)}}{\langle \alpha^\frac{1}{2}(t-s)\rangle^\frac{1}{3}}\left(\left\||\bm{\nabla}_{y,z}|^{-1}R_+\right\|_{W^{3,1}}+\left\||\bm{\nabla}_{y,z}|^{-1}R_-\right\|_{W^{3,1}}\right)\ ds\\
    &\lesssim\frac{e^{-\nu t}}{\langle\alpha^\frac{1}{2}t\rangle^\frac{1}{3}}\varepsilon+\alpha^{-\frac{1}{6}}\nu^{-\frac{1}{6}}\left(\left\||\bm{\nabla}_{y,z}|^{-1}R_+\right\|_{L^2_tW^{3,1}}+\left\||\bm{\nabla}_{y,z}|^{-1}R_-\right\|_{L^2_tW^{3,1}}\right),
\end{align*}where we have used the fact that\[
\left\|\frac{e^{-\nu t}}{\langle \alpha^\frac{1}{2}t\rangle^\frac{1}{3}}\right\|_{L^2}\lesssim\alpha^{-\frac{1}{6}}\nu^{-\frac{1}{6}}.
\]Consider the single term\begin{align*}
\left\||\bm{\nabla}_{y,z}|^{-1}R_+\right\|_{L^2_tW^{3,1}}&\lesssim\left\|\overline{(U\cdot\bm{\nabla}_L U_2)}\right\|_{L^2_tW^{3,1}}+\left\|\overline{(U\cdot\bm{\nabla}_L U_3)}\right\|_{L^2_tW^{3,1}}+\left\|\overline{(U\cdot\bm{\nabla}_L \Psi)}\right\|_{L^2_tW^{3,1}}\\
&\lesssim\left\|\bm{\nabla}_L\mathcal{Z}(U,\Psi)_{\neq}\right\|_{L^2_tL^2}\left\|\mathcal{Z}U_{\neq}\right\|_{L^\infty_t L^2}+\left\|(\overline{U}_2,\overline{U}_3)\right\|_{L^\infty_tH^s}\left\|\bm{\nabla}\left(\overline{U}_2,\overline{U}_3,\overline{\Psi}\right)\right\|_{L^2_tH^s}\\
&+\left\|\overline{U}_2\right\|_{L^\infty_t H^s}\left\|\partial_y\left(\widetilde{U}_3,\widetilde{\Psi}\right)\right\|_{L^2_tH^s}+\left\|\widetilde{U}_3\right\|_{L^\infty_t H^s}\left\|\partial_z\left(\overline{U}_2,\overline{U}_3,\overline{\Psi}\right)\right\|_{L^2_tH^s}\\
&\lesssim\nu^{-\frac{1}{2}}\varepsilon^2,
\end{align*}where we have used the simple zero mode decomposition: for $\overline{(U\cdot\bm{\nabla}_L F)}$, since $\widetilde{U}_2=0$, we have\begin{align}
\overline{(U\cdot\bm{\nabla}_L F)}&=\overline{(U_{\neq}\cdot\bm{\nabla}_L F_{\neq})}+\overline{(U_0\cdot\bm{\nabla}_{y,z} F_0)}\nonumber\\
&=\overline{(U_{\neq}\cdot\bm{\nabla}_L F_{\neq})}+\overline{(\overline{U}\cdot\bm{\nabla}_{y,z}\overline{F})}+\overline{(\overline{U}_2\partial_y\widetilde{F})}+\overline{(\widetilde{U}_3\partial_z\overline{F})}.\label{simplezeromoded9}
\end{align}\hfill \qedsymbol 

 Besides, we also decompose a function $\overline{\psi}$ into the following form\begin{equation}
 \overline{\psi}:=\overline{\psi}_{\text{in}}(y,z)+\overline{\psi}_{\text{nl}}(y,z),\label{functiondecomposition999}
 \end{equation}where $\overline{\psi}_{\text{in}}$ corresponds to the linear part of the solution, while $\overline{\psi}_{\text{nl}}$ refers to the nonlinear part. This decomposition is well-defined via the Duhamel formula for $\psi\in\{U_1,U_2,U_3,\Psi\}$ (\ref{duhamel1}) and (\ref{duhamel2}). As an example, we denote\begin{align*}
\overline{U}_2=&\underbrace{\frac{1}{2}\partial_z\Delta^{-1}_{y,z}\left(e^{\mathcal{L}_+ t}J_+(0,y,z)+e^{\mathcal{L}_-t}J_-(0,y,z)\right)}_{\overline{U}_{2,\text{in}}(y,z)}\\
+&\underbrace{\frac{1}{2}\partial_z\Delta^{-1}_{y,z}\int_0^t e^{\mathcal{L}_+(t-s)}R_+(s,y,z)+e^{\mathcal{L}_-(t-s)}R_-(s,y,z)\ ds}_{\overline{U}_{2,\text{nl}}(y,z)}.
\end{align*}Here the first term corresponds to the linear part, while the latter corresponds to the nonlinear part. Such decomposition is used to capture different properties of the linear part and the nonlinear part to improve the estimate. 

\begin{lemma}\label{modedeco} For $F\in\{U_2,U_3,\Psi\}$, under the bootstrap assumptions of Theorem \ref{boot}, assume that $\varepsilon\lesssim \nu$, then\begin{align}
\left\|\langle\bm{\nabla}_{y,z}\rangle^s\overline{F}_{\text{in}}\right\|_{W^{1,\infty}}&\lesssim \frac{e^{-\nu t}}{\langle\alpha^\frac{1}{2}t\rangle^\frac{1}{3}}\varepsilon,\label{d1}\\
\left\|\overline{F}_{\text{nl}}\right\|_{L^\infty_tH^s}+\nu^\frac{1}{2}\left\|\bm{\nabla}_{y,z}\overline{F}_{\text{nl}}\right\|_{L^2_tH^s}&\lesssim\nu^{-\frac{2}{3}}\varepsilon^2.
\label{d2}
\end{align}\end{lemma}

\noindent\textit{Proof.} The estimate (\ref{d1}) is trivial. The estimate \eqref{d2} can be derived from direct energy estimate, we have\begin{align*}
    \frac{1}{2}\sum_{F\in\{U_2,U_3,\Psi\}}\left\|\overline{F}_{\text{nl}}\right\|_{L^\infty_t(0,t) H^s}^2+\nu\sum_{F\in\{U_2,U_3,\Psi\}}\left\|\bm{\nabla}_{y,z}\overline{F}_{\text{nl}}\right\|_{L^2_t(0,t) H^s}^2=-\sum_{F\in\{U_2,U_3,\Psi\}}\int_0^t\langle\overline{(U\cdot\bm{\nabla}_L F)},\overline{F}_{\text{nl}}\rangle_{H^s}\ ds.
\end{align*}Using (\ref{simplezeromoded9}), we have\[
\left\|\overline{U\cdot\bm{\nabla}_LF}\right\|_{H^s}\leq\underbrace{\left\|\overline{U_{\neq}\cdot\bm{\nabla}_L F_{\neq}}\right\|_{H^s}}_{D_1}+\underbrace{\left\|\overline{\overline{U}\cdot\bm{\nabla}_{y,z}\overline{F}}\right\|_{H^s}}_{D_2}+\underbrace{\left\|\overline{\overline{U}_2\partial_y\widetilde{F}}\right\|_{H^s}}_{D_3}+\underbrace{\left\|\overline{\widetilde{U}_3\partial_z\overline{F}}\right\|_{H^s}}_{D_4},
\]for $F\in\{U_2,U_3,\Psi\}$. For each term, we have\begin{align*}
    D_1&\lesssim e^{-\varpi\nu^{\frac{1}{3}}t}||\mathcal{Z}U_{\neq}||_{L^2}||\bm{\nabla}\mathcal{Z}F_{\neq}||_{L^2},\\
    D_2&\lesssim\left\|\overline{(U_2,U_3)}\right\|_{H^s}\left\|\bm{\nabla}_{y,z}\overline{F}\right\|_{L^\infty}+\left\|\overline{(U_2,U_3)}\right\|_{L^\infty}\left\|\bm{\nabla}_{y,z}\overline{F}\right\|_{H^s}\\
    &\lesssim\left(\frac{e^{-\nu t}}{\langle\alpha^\frac{1}{2}t\rangle^\frac{1}{3}}\varepsilon+\alpha^{-\frac{1}{6}}\nu^{-\frac{2}{3}}\varepsilon^2\right)\sum_{F\in{\{U_2,U_3,\Psi\}}}\left\|\bm{\nabla}_{y,z}\overline{F}\right\|_{H^s},\\
    D_3&\lesssim\left\|\langle\bm{\nabla}_{y,z}\rangle^s\overline{U}_{2,\text{in}}\right\|_{L^\infty}\left\|\partial_y\widetilde{F}\right\|_{H^s}+\left\|\overline{U}_{2,\text{nl}}\right\|_{H^s}\left\|\partial_y\widetilde{F}\right\|_{H^s}\\
    &\lesssim\frac{e^{-\nu t}}{\langle\alpha^\frac{1}{2}t\rangle^\frac{1}{3}}\varepsilon\left\|\partial_y\widetilde{F}\right\|_{H^s}+\left\|\overline{U}_{2,\text{nl}}\right\|_{H^s}\left\|\partial_y\widetilde{F}\right\|_{H^s},\\
    D_4&\lesssim\left\|\widetilde{U}_3\right\|_{H^s}\left(\left\|\langle\bm{\nabla}_{y,z}\rangle^s\partial_z\overline{F}_{\text{in}}\right\|_{L^\infty}+\left\|\partial_z\overline{F}_{\text{nl}}\right\|_{H^s}\right)\\
    &\lesssim\left(\frac{e^{-\nu t}}{\langle\alpha^\frac{1}{2}t\rangle^\frac{1}{3}}\varepsilon+\left\|\partial_z\overline{F}_{\text{nl}}\right\|_{H^s}\right)\left\|\widetilde{U}_3\right\|_{H^s}.
\end{align*}We then see that\begin{align*}
    \int_0^t\left|\langle\overline{U_{\neq}\cdot\bm{\nabla}_L F_{\neq}},\overline{F}_{\text{nl}}\rangle_{H^s}\right|&\lesssim\nu^{-\frac{2}{3}}\varepsilon^2\left\|\overline{F}_{\text{nl}}\right\|_{L^\infty_tH^s},\\
    \int_0^t \left|\langle \overline{(\overline{U}\cdot\bm{\nabla}_{y,z}\overline{F}}),\overline{F}_{\text{nl}}\rangle_{H^s}\right|&\lesssim\sum_{F\in{\{U_2,U_3,\Psi\}}}\int_0^t \frac{e^{-\nu s}}{\langle\alpha^\frac{1}{2}s\rangle^\frac{1}{3}}\varepsilon \left\|\overline{F}_{\text{nl}}\right\|_{L^\infty_tH^s}\left\|\bm{\nabla}_{y,z}\overline{F}\right\|_{H^s}+\alpha^{-\frac{1}{6}}\nu^{-\frac{2}{3}}\varepsilon^2 \left\|\bm{\nabla}_{y,z}\overline{F}\right\|_{H^s}\left\|\bm{\nabla}_{y,z}\overline{F}_{\text{nl}}\right\|_{H^s}\\
    &\lesssim\alpha^{-\frac{1}{6}}\nu^{-\frac{2}{3}}\varepsilon^2\left\|\overline{F}_{\text{nl}}\right\|_{L^\infty_tH^s}+\alpha^{-\frac{1}{6}}\nu^{-\frac{7}{6}}\varepsilon^3\left\|\bm{\nabla}_{y,z}\overline{F}_{\text{nl}}\right\|_{L^2_tH^s},\\
    \int_0^t \left|\langle(\overline{\overline{U}_2\partial_y\widetilde{F}}),\overline{F}_{\text{nl}}\rangle_{H^s}\right|&\lesssim \alpha^{-\frac{1}{6}}\nu^{-\frac{2}{3}}\varepsilon^2\left\|\overline{F}_{\text{nl}}\right\|_{L^\infty_tH^s}+\nu^{-\frac{1}{2}}\varepsilon\left\|\overline{U}_{2,\text{nl}}\right\|_{L^\infty_tH^s}\left\|\bm{\nabla}_{y,z}\overline{F}_{\text{nl}}\right\|_{L^2_tH^s},\\
    \int_0^t \left|\langle(\overline{\widetilde{U}_3\partial_z\overline{F}}),\overline{F}_{\text{nl}}\rangle_{H^s}\right|&\lesssim\alpha^{-\frac{1}{6}}\nu^{-\frac{2}{3}}\varepsilon^2\left\|\overline{F}_{\text{nl}}\right\|_{L^\infty_tH^s}+\varepsilon\sum_{F\in\{U_2,U_3,\Psi\}}\left\|\bm{\nabla}_{y,z}\overline{F}_{\text{nl}}\right\|_{L^2_t H^s}.
    \end{align*}Here we have used the Poincar\'e inequality $\left\|\overline{F}\right\|_{H^s}\lesssim\left\|\bm{\nabla}_{y,z}\overline{F}\right\|_{H^s}$. Since $\varepsilon\lesssim\nu$, we have\begin{align*}
    &\nu^{-\frac{2}{3}}\varepsilon^2\left\|\overline{F}_{\text{nl}}\right\|_{L^\infty_tH^s}\lesssim\left(\nu^{-\frac{2}{3}}\varepsilon^2\right)^2+\left\|\overline{F}_{\text{nl}}\right\|_{L^\infty_tH^s}^2,\\ &\nu^{-\frac{1}{2}}\varepsilon\left\|\overline{U}_{2,\text{nl}}\right\|_{L^\infty_tH^s}\left\|\bm{\nabla}_{y,z}\overline{F}_{\text{nl}}\right\|_{L^2_tH^s}\lesssim\left\|\overline{F}_{\text{nl}}\right\|_{L^\infty_tH^s}^2+\nu\left\|\bm{\nabla}_{y,z}\overline{F}_{\text{nl}}\right\|_{L^2_tH^s}^2.
    \end{align*}\hfill \qedsymbol

Finally, we prove the following bound, which turns out to be the key estimate for the energy estimate of $\overline{U}_1$.\begin{lemma}\label{lemma9090} Under the bootstrap assumptions of Theorem \ref{boot}, assume that $\varepsilon\lesssim\nu^{\frac{4}{3}}$, the following estimate holds\[
\left\|\overline{U}_{2,\text{nl}}\right\|_{L^2_tH^s}\lesssim\nu^{-\frac{5}{6}}\varepsilon^2+\alpha^{-\frac{1}{6}}\nu^{-\frac{7}{6}}\varepsilon^2.
\]\begin{proof} It suffices study the following integral\[
\left\|\int_0^te^{\mathcal{L}_+(t-s)}\overline{(U\cdot\bm{\nabla}_L F)}\ ds\right\|_{L^2_tH^s},\qquad F\in\left\{U_2,U_3,\Psi\right\}.
\]Due to the Minkowski's and Young's inequalities, we have\[
\left\|\int_0^te^{\mathcal{L}_+(t-s)}h(s)\ ds\right\|_{L^2_t(0,\tau)H^s}\leq\int_0^\tau\left\|\bm{1}_{\{t\geq s\}}e^{\mathcal{L}_+(t-s)}h(s)\right\|_{L^2_t(0,\tau)H^s}\ ds\leq\int_0^\tau\left\|e^{\mathcal{L}_+(t-s)}h(s)\right\|_{L^2_t(0,\tau)H^s}\ ds,
\]we have\begin{align*}
    \left\|\int_0^te^{\mathcal{L}_\pm(t-s)}h(s)\ ds\right\|_{L^2_t H^s}&\lesssim\nu^{-\frac{1}{2}}||h||_{L^1_t H^{s-1}},\\
    \left\|\int_0^te^{\mathcal{L}_\pm(t-s)}h(s)\ ds\right\|_{L^2_t H^s}&\lesssim\nu^{-1}||h||_{L^2_t H^{s-1}}.
\end{align*}Note that\[
\overline{U\cdot\bm{\nabla}_L F}=\bm{\nabla}_{y,z}\cdot\overline{(U_{\neq}F_{\neq})}+\bm{\nabla}_{y,z}\cdot\overline{\left[(\overline{U}_2,\overline{U}_3)^T,\overline{F}\right]}+\overline{(\overline{U}_2\partial_y\widetilde{F})}+\partial_z\overline{(\widetilde{U}_3\overline{F})}.
\]Therefore, we have\begin{align*}
    \left\|\int_0^te^{\mathcal{L}_+(t-s)}\overline{(U_{\neq}\cdot\bm{\nabla}_L F_{\neq})}\ ds\right\|_{L^2_tH^s}&\lesssim\nu^{-\frac{1}{2}}||\mathcal{Z}U_{\neq}||_{L^2_tL^2}||\mathcal{Z}F_{\neq}||_{L^2_tL^2}\lesssim\nu^{-\frac{5}{6}}\varepsilon^2,\\
    \left\|\int_0^te^{\mathcal{L}_+(t-s)}\overline{(\overline{U}\cdot\bm{\nabla}_L \overline{F})}\ ds\right\|_{L^2_tH^s}&\lesssim\nu^{-1}\left\|\overline{F}\right\|_{L_t^\infty H^s}\left\|\langle\bm{\nabla}_{y,z}\rangle^s\overline{U}_{2,3,\text{in}}\right\|_{L^2_tL^\infty}+\nu^{-1}\left\|\bm{\nabla}_{y,z}\overline{F}\right\|_{L^2_tH^s}\left\|\overline{U}_{2,3,\text{nl}}\right\|_{L^\infty_tH^s},\\
    &\lesssim\alpha^{-\frac{1}{6}}\nu^{-\frac{7}{6}}\varepsilon^2+\nu^{-\frac{13}{6}}\varepsilon^3, \\
     \left\|\int_0^te^{\mathcal{L}_+(t-s)}\overline{(\overline{U}_2\partial_y\widetilde{F})}\ ds\right\|_{L^2_tH^s}&\lesssim \alpha^{-\frac{1}{6}}\nu^{-\frac{7}{6}}\varepsilon^2+\nu^{-\frac{13}{6}}\varepsilon^3,\\
      \left\|\int_0^te^{\mathcal{L}_+(t-s)}\overline{(\widetilde{U}_3\partial_z\overline{F})}\ ds\right\|_{L^2_tH^s}&\lesssim \alpha^{-\frac{1}{6}}\nu^{-\frac{7}{6}}\varepsilon^2+\nu^{-\frac{13}{6}}\varepsilon^3.
\end{align*}
    
\end{proof}
    
\end{lemma}

\subsubsection{\textbf{Double zero modes}}\label{double}
Due to the incompressibility, the double zero mode system reads\begin{flalign*}
\left\{
\begin{aligned}
&\partial_t \widetilde{U}_1=\nu\partial^2_y\widetilde{U}_{1}-\partial_y\widetilde{(U_2U_1)},\\
&\partial_t \widetilde{U}_{3}=\sqrt{\alpha}\widetilde{\Psi}+\nu\partial^2_y \widetilde{U}_{3}-\partial_y\widetilde{(U_2U_3)},\\
&\partial_t\widetilde{\Psi}=-\sqrt{\alpha}\widetilde{U}_{3}+\nu\partial^2_y\widetilde{\Psi}-\partial_y\widetilde{(U_2\Psi)},\\
&\partial_y \widetilde{U}_{2}=0.
\end{aligned}
\right.
\end{flalign*}The last line is the continuity equation for the double zero modes. Since $\widetilde{U}_2\in L^2_y$, the only possibility is that $\widetilde{U}_2=0$ for all $t\geq0$.

\begin{lemma}\label{doublezero} 
Under the bootstrap assumptions of Theorem \ref{boot}, assuming that $\varepsilon\lesssim\nu$ and $\Upsilon\leq 1/3$, then the following estimates hold, \begin{align*}
&\left\|\widetilde{U}_1\right\|_{L^\infty_tH^s}^2+\nu\left\|\partial_y\widetilde{U}_1\right\|_{L^2_tH^s}^2 \lesssim\varepsilon^2,\\
  \sum_{W\in\{U_3,\Psi\}}&\left\|\widetilde{W}\right\|_{L^\infty_tH^s}^2+\nu\sum_{W\in\{U_3,\Psi\}}\left\|\partial_y\widetilde{W}\right\|_{L^2_tH^s}^2 \lesssim\varepsilon^2.
\end{align*}
\end{lemma}\vspace{3pt}

\noindent\textit{Proof.} For $\widetilde{U}_3$ and $\widetilde{\Psi}$, a standard energy estimate yields \[
\sum_{W\in\{U_3,\Psi\}}\left\|\widetilde{W}\right\|_{L^\infty_tH^s}^2+\nu\sum_{W\in\{U_3,\Psi\}}\left\|\partial_y\widetilde{W}\right\|_{L^2_tH^s}^2\lesssim\varepsilon^2+\left|\left\langle\int_0^\infty\widetilde{(U_2U_3)}\partial_y\widetilde{U}_3+\widetilde{(U_2\Psi)}\partial_y\widetilde{\Psi}\ ds\right\rangle_{H^s}\right|.
\]Notice that\begin{align*}
\widetilde{(U_2F)}&=\widetilde{(U_{2,\neq}F_{\neq})}+\widetilde{(\overline{U}_2\overline{F})},\qquad F\in\{U_1,U_3,\Psi\}.
\end{align*}Thus, using the fact\[
\left\|\overline{U}\right\|_{H^s}\lesssim\left\|\bm{\nabla}_{y,z}\overline{U}\right\|_{H^s},
\]we have the following bound\begin{align*}
\left\|\int^\infty_0\partial_y\widetilde{(U_2F)}\widetilde{F}\ ds\right\|_{H^s}&\leq\int_0^\infty\left\|\partial_y\widetilde{F}\right\|_{H^s}\left\|\widetilde{U_2F}\right\|_{H^s}\ ds\\
&\lesssim\sum_{E\in\{U_2,F\}}\int_0^\infty\left\|\partial_y\widetilde{F}\right\|_{H^s}\left(e^{-\varpi\nu^{\frac{1}{3}}s}||\mathcal{Z}U_{2,\neq}||_{L^2}||\mathcal{Z}F_{\neq}||_{L^2}+\left\|\overline{E}\right\|_{L^\infty}\left\|\overline{E}\right\|_{H^s}\right) ds\\
&\lesssim\nu^{-\frac{1}{6}}\varepsilon^2\left\|\partial_y\widetilde{F}\right\|_{L^2_tH^s}+\sum_{E\in\{U_2,F\}}\int_0^\infty\left\|\partial_y\widetilde{F}\right\|_{H^s}\left(\frac{e^{-\nu s}}{\langle\alpha^\frac{1}{2}s\rangle^\frac{1}{3}}\varepsilon+\alpha^{-\frac{1}{6}}\nu^{-\frac{2}{3}}\varepsilon^2\right)\left\|\overline{E}\right\|_{H^s}\\
&\lesssim\nu^{-\frac{2}{3}}\varepsilon^3+\alpha^{-\frac{1}{6}}\nu^{-\frac{5}{3}}\varepsilon^4.
\end{align*}
For $\widetilde{U}_1$, the estimate can be derived similarly. We have\begin{align*}
\left\|\widetilde{U}_1\right\|_{L^\infty_tH^s}^2+\nu\left\|\partial_y\widetilde{U}_1\right\|_{L^2_tH^s}^2&\lesssim\varepsilon^2+\nu^{-\frac{2}{3}}\varepsilon^3+\int_0^\infty\left\|\partial_y\widetilde{U}_1\right\|_{H^s}\left(\left\|\langle\bm{\nabla}_{y,z}\rangle^s\overline{U}_{2,\text{in}}\right\|_{L^\infty}+\left\|\overline{U}_{2,\text{nl}}\right\|_{H^s}\right)\left\|\overline{U}_1\right\|_{H^s}
    \lesssim\varepsilon^2\\
    &\lesssim\varepsilon^2+\alpha^{-\frac{1}{6}}\nu^{-\frac{2}{3}-\Upsilon}\varepsilon^3+\nu^{-\frac{5}{3}-\Upsilon}\varepsilon^4.
\end{align*}where we have used the fact $\varepsilon\lesssim\nu$ and assumed $\Upsilon\leq1/3$.\hfill \qedsymbol

\subsubsection{\texorpdfstring{\textbf{A bound on $\left\|\overline{U}_1\right\|_{L^\infty}$}}{A bound on U1}}\label{linfty}

Prior to the energy estimate of $\overline{U}_1$, we establish the following $L^\infty$ bound on $\overline{U}_1$, which can be used to improve the threshold.

\begin{lemma}\label{inftyonu1} Assume that $\varepsilon\lesssim\nu^{\frac{4}{3}}$. Then the following estimate holds\[
\left\|\overline{U}_1\right\|_{L^\infty}\lesssim\varepsilon.
\]
\end{lemma}

\noindent\textit{Proof}. We use the Duhamel formula (\ref{explicitsoluU1}), the $L^\infty$ bound on the initial condition is clearly of order $\varepsilon$. We therefore estimate the following two integrals\begin{align*}
    Y_1:=&\left\|\int_0^t e^{\mathcal{L}_1(t-s)}R_1(s,y,z)\ ds\right\|_{L^\infty},\\
    Y_2:=&\frac{1}{\sqrt{\alpha}}\left\|\int_0^t\partial_z\partial^{-1}_ye^{\nu (t-s)\Delta}\left(e^{\pm\sqrt{\alpha}(t-s)\mathcal{R}_y}-I_d\right)|\bm{\nabla}_{y,z}|^{-1}R_{\pm}\right\|_{L^\infty}.
\end{align*}The first one can be simply bounded by\begin{align*}
    Y_1\leq\int_0^t\left\|e^{\nu (t-s)\Delta}\overline{(U\cdot\bm{\nabla}_LU_1)}\right\|_{H^s}\ ds,
\end{align*}and we have\[
Y_1\lesssim\nu^{-1-\Upsilon}\varepsilon^2\lesssim\varepsilon.
\]Next, we consider $Y_2$, by lemma \ref{suppofl3}, we have\begin{align*}
    Y_2\lesssim \nu^{-\frac{1}{2}}\alpha^{-\frac{1}{4}}\int_0^t e^{-c\nu (t-s)}\left\||\bm{\nabla}_{y,z}|^{-1}R_{\pm}\right\|_{L^2}.
\end{align*}Recall that\[
 R_{+}=-\overline{\partial_z(U\cdot\bm{\nabla}_LU_2)}+\overline{\partial_y(U\cdot\bm{\nabla}_LU_3)}-|\bm{\nabla}_{y,z}|\overline{(U\cdot\bm{\nabla}_L\Psi)}.
\]Due to the mode decomposition formula (\ref{simplezeromoded9}), we therefore consider the following integral for simplicity\begin{align*}
\int_0^t e^{-c\nu (t-s)}\left\|\overline{(U\cdot\bm{\nabla}_L F)}\right\|_{L^2}&\lesssim\int_0^te^{-c\nu(t-s)}\left(\left\|\overline{U_{\neq}\cdot\bm{\nabla}_L F}\right\|_{L^2}+\left\|\overline{\overline{U}\cdot\bm{\nabla}_{y,z}\overline{F}}\right\|_{L^2}+\left\|\overline{\overline{U}_2\partial_y\widetilde{F}}\right\|_{L^2}+\left\|\overline{\widetilde{U}_3\partial_z\overline{F}}\right\|_{L^2}\right)\\
&\lesssim Y_{21}+Y_{22}+Y_{23}+Y_{24},
\end{align*}where $F\in\{U_2,U_3,\Psi\}$. We derive a bound for each term.\begin{align*}
    Y_{21}&\lesssim\int_0^te^{-\varpi\nu^\frac{1}{3}s}\left\|\mathcal{Z}U_{\neq}\right\|_{L^2}||\mathcal{Z}F_{\neq}||_{L^2} \\
    &\lesssim\nu^{-\frac{1}{3}}\varepsilon^2.\\
    Y_{22}&\lesssim\sum_{F\{U_2,U_3,\Psi\}}\int_0^te^{-c\nu (t-s)}\left(\left\|\overline{U}\right\|_{L^{\infty}}\left\|\overline{F}\right\|_{H^s}+\left\|\overline{U}\right\|_{H^s}\left\|\overline{F}\right\|_{W^{1,\infty}}\right)\\
    &\lesssim\sum_{F\{U_2,U_3,\Psi\}}\int_0^te^{-c\nu (t-s)}\left(\frac{e^{-\nu s}}{\langle\alpha^\frac{1}{2}s\rangle^\frac{1}{3}}\varepsilon+\alpha^{-\frac{1}{6}}\nu^{-\frac{2}{3}}\varepsilon^2\right)\left\|\overline{F}\right\|_{H^s}\\
    &\lesssim\alpha^{-\frac{1}{6}}\nu^{-\frac{2}{3}}\varepsilon^2.\\
    Y_{23}&\lesssim\int_0^t e^{-c\nu(t-s)}\left(\left\|\overline{U}_{2,\text{in}}\right\|_{L^\infty}+\left\|\overline{U}_{2,\text{nl}}\right\|_{H^s}\right)\left\|\widetilde{F}\right\|_{H^s}\\
    &\lesssim\int_0^t e^{-c\nu(t-s)}\left(\frac{e^{-\nu s}}{\langle\alpha^\frac{1}{2}s\rangle^\frac{1}{3}}\varepsilon+\nu^{-\frac{2}{3}}\varepsilon^2\right)\left\|\widetilde{F}\right\|_{H^s}\\
    &\lesssim\alpha^{-\frac{1}{6}}\nu^{-\frac{2}{3}}\varepsilon^2+\nu^{-\frac{5}{3}}\varepsilon^3.\\
    Y_{24}&\lesssim\int_0^te^{-c\nu (t-s)}\left(\left\|\overline{F}_{\text{in}}\right\|_{W^{1,\infty}}+\left\|\overline{F}_{\text{nl}}\right\|_{H^s}\right)\left\|\widetilde{U}_3\right\|_{H^s}\\
    &\lesssim\alpha^{-\frac{1}{6}}\nu^{-\frac{2}{3}}\varepsilon^2+\nu^{-\frac{5}{3}}\varepsilon^3.
\end{align*}\hfill \qedsymbol

\begin{remark} Thanks to the weak lift-up effect, we are able to derive a smaller bound on $\left\|\overline{U}_1\right\|_{L^\infty}$. This is the key estimate to improve the estimate of $J_3$ in the non-zero-mode analysis (\ref{largebound}). From the proof, we can actually derive a more precise bound on $\left\|\overline{U}_1\right\|_{L^\infty}$ as follows (although Lemma \ref{inftyonu1} is sufficient.)\[
\left\|\overline{U}_1\right\|_{L^\infty}\lesssim e^{-\nu t}\varepsilon+\left(\nu^{-1-\Upsilon}+\alpha^{-\frac{1}{4}}\nu^{-\frac{5}{6}}+\alpha^{-\frac{5}{12}}\nu^{-\frac{7}{6}}\right)\varepsilon^2.
\]\end{remark}

\subsubsection{\texorpdfstring
  {Energy estimate for $\overline{U}_1$}
  {Energy estimate for U1}}\label{enzero2}
 Finally, we prove the energy estimate for $\overline{U}_1$. \begin{lemma}\label{simplezeromodeofu1} Under the bootstrap assumptions of Theorem \ref{boot}, if $\varepsilon\lesssim\nu$, then
\begin{align}
\left\|\overline{U}_1\right\|_{L^\infty_tH^s}^2+&\nu\left\|\bm{\nabla}_{y,z}\overline{U}_1\right\|_{L^2_tH^s}^2\lesssim\varepsilon^2+\nu^{-1-2\Upsilon}\varepsilon
^3+\nu^{-2\Upsilon}\varepsilon^2\left(\alpha^{-\frac{1}{6}}\nu^{-\frac{5}{3}+\Upsilon}\varepsilon+\nu^{-\frac{4}{3}+\Upsilon}\varepsilon\right)\nonumber\label{fin}
\end{align}
\end{lemma}\begin{proof} We consider an auxiliary unknown\[
G:=\int_0^t e^{\nu(t-s)\Delta_{y,z}}\overline{U}_{2,\text{in}}(s)\ ds,
\]the motivation for this new unknown is to capture the lift-up part of the dynamics. In this case, we see that\[
\partial_t G-\nu\Delta_{y,z}G=\overline{U}_{2,\text{in}}.
\]Therefore, we can decompose $\overline{U}_1=\overline{U}_1+G-G$ such that\[
(\partial_t-\nu\Delta_{y,z})\left(\overline{U}_1+G\right)=-\overline{(U\cdot\bm{\nabla}_L U_1)}-\overline{U}_{2,\text{nl}}.
\]Direct energy estimate yields\[
\left\|\overline{U}_1+G\right\|_{L^\infty_t H^s}^2+\nu\left\|\bm{\nabla}_{y,z}\left(\overline{U}_1+G\right)\right\|_{L^2_tH^s}^2\lesssim\varepsilon^2+\left(\left\|\overline{(U\cdot\bm{\nabla}_L U_1)}\right\|_{L^2_tH^s}+\left\|\overline{U}_{2,\text{nl}}\right\|_{L^2_tH^s}\right)\left\|\overline{U}_1+G\right\|_{L^2_tH^s}.
\]We have\begin{align*}
    \left\|\overline{(U\cdot\bm{\nabla}_LU_1)}\right\|_{L^2_tH^s}&\lesssim\left\|U_{\neq}\cdot\bm{\nabla}_L U_{1,\neq}\right\|_{L^2_tH^s}+\left\|\overline{U}\cdot\bm{\nabla}_{y,z} \overline{U}_1\right\|_{L^2_tH^s}+\left\|\overline{U}_2\partial_y\widetilde{U}_1\right\|_{L^2_tH^s}+\left\|\widetilde{U}_3\partial_z\overline{U}_1\right\|_{L^2_t H^s}\\
    &\lesssim\nu^{-\frac{1}{2}-\Upsilon}\varepsilon^2.
\end{align*}Due to the Poincar\'e inequality, we have\[
\left\|\overline{U}_1+G\right\|_{L^2_tH^s}\lesssim\left\|\bm{\nabla}_{y,z}\overline{U}_1\right\|_{L^2_tH^s}+\nu^{-\frac{1}{2}}\varepsilon\lesssim\nu^{-\frac{1}{2}-\Upsilon}\varepsilon.
\]Finally, due to Lemma \ref{lemma9090}, we obtain the desired energy estimate.\end{proof}

\noindent\textbf{Remark}. One can easily deduce the bound\[
\left\|\partial_y\overline{U}_1\right\|_{L^2_tH^s}\lesssim\nu^{-\frac{1}{2}}\varepsilon.
\]To see this, consider the linearized PV $\overline{\Xi}:=\partial_y\overline{U}_1+\alpha^{-\frac{1}{2}}\partial_z\overline{\Psi}$, we can write $\partial_y\overline{U}_1=\overline{\Xi}-\alpha^{-\frac{1}{2}}\partial_z\overline{\Psi}$. Therefore, we have\[
\left\|\partial_y\overline{U}_1\right\|_{L^2_tH^s}\lesssim\nu^{-\frac{1}{2}}\varepsilon+\left\|\int_0^t e^{\nu(t-s)\Delta}\partial_y\overline{(U\cdot\bm{\nabla}_L U_1)}\ ds\right\|_{L^2_tH^s}+\left\|\int_0^t e^{\nu(t-s)\Delta}\partial_z\overline{(U\cdot\bm{\nabla}_L \Psi)}\ ds\right\|_{L^2_tH^s}.
\]By a similar argument to that in Lemma \ref{lemma9090}, we can obtain the bound.

Gathering all the estimates, we obtain the threshold in the form of the following inequality,\[
\varepsilon\lesssim\min\left\{\nu^{1+\frac{\sqrt{2}}{3}},\nu^{\frac{2+\sqrt{2}}{3}+\Upsilon},\nu^{\frac{5}{3}-\Upsilon}\alpha^{\frac{1}{6}},\nu^{\frac{4}{3}-\Upsilon}\right\}.
\]
\subsection{\textbf{The case with linear lift-up}}\label{nonlinearforlinearliftup}
The aim of this section is to prove Theorem \ref{nonlinear1}. We do not assume any additional assumptions on the initial data. In this case, $\overline{U}_1$ experiences both linear and nonlinear lift-up effect, and the bootstrap argument for $\overline{U}_1$ is set to be\[
\left\|\overline{U}_1\right\|_{L^\infty_tH^s}+\nu^\frac{1}{2}\left\|\bm{\nabla}_{y,z}\overline{U}_1\right\|_{L^2_tH^s}\lesssim\left(1+\nu^{-\frac{1}{2}}\alpha^{-\frac{1}{4}}\right)\varepsilon.
\]In this case, we have\[
\left\|\overline{U}_1\right\|_{L^\infty}\lesssim \left(1+\nu^{-\frac{1}{2}}\alpha^{-\frac{1}{4}}\right) e^{-\frac{\nu t}{2}}\varepsilon+\nu^{-2}\varepsilon^2.
\]Assuming again $\varepsilon\lesssim\nu$, the non-zero modes satisfy the following bound\begin{align*}
\sum_{F\in\{A,B,K\}}\left\|\mathcal{Z}F_{\neq}\right\|_{L^\infty_tL^2}^2+&\nu\sum_{F\in\{A,B,K\}}\left\|\bm{\nabla}_L\mathcal{Z}F_{\neq}\right\|_{L^2_t L^2}^2+\sum_{F\in\{A,B,K\}}\left\|\sqrt{-\frac{\mathcal{N}'}{\mathcal{N}}}\mathcal{Z}F_{\neq}\right\|_{L^2_tL^2}^2\nonumber\\\lesssim& \varepsilon^2+\nu^{-\frac{\sqrt{2}}{3}-1}\varepsilon^3+\nu^{-\frac{\sqrt{2}}{3}-\frac{2}{3}}\left(1+\nu^{-\frac{1}{2}}\alpha^{-\frac{1}{4}}\right)\varepsilon^3.
    \end{align*}This suggests that the best threshold we can obtain is $\varepsilon\lesssim \nu^{7/6+\sqrt{2}/3}$, and can be continuously improved to $\nu^{1+\sqrt{2}/3}$ as $\alpha\gtrsim \nu^{-2/3}$. For the zero mode part, Lemma \ref{nonlinear dispersive}, \ref{modedeco} and \ref{lemma9090} remain unchanged since the dynamics of $\overline{U}_2$, $\overline{U}_3$ and $\overline{\Psi}$ are not affected by $\overline{U}_1$. In the double zero mode analysis, the mode decomposition of $\widetilde{(U_2U_1)}$ contains $\widetilde{(\overline{U}_2\overline{U_1})}$, thus we have\[
\int_0^\infty\left\|\partial_y\widetilde{U}_1\right\|_{H^s}\left\|\widetilde{\overline{U}_2\overline{U}_1}\right\|_{H^s}\ ds\lesssim\int_0^\infty\left\|\partial_y\widetilde{U}_1\right\|_{H^s}\left(\frac{e^{-\nu s}}{\langle\alpha^\frac{1}{2}s\rangle^\frac{1}{3}}\varepsilon\left\|\overline{U}_1\right\|_{H^s}+\nu^{-\frac{2}{3}}\varepsilon^2\left\|\bm{\nabla}_{y,z}\overline{U}_1\right\|_{H^s}\right)
    \]Thus, the following holds\[
    \left\|\widetilde{U}_1\right\|_{L^\infty_tH^s}^2+\nu\left\|\partial_y\widetilde{U}_1\right\|_{L^2_tH^s}^2\lesssim\varepsilon^2+\nu^{-\frac{1}{2}}\varepsilon\left(\alpha^{-\frac{1}{6}}\nu^{-\frac{1}{6}}\varepsilon+\nu^{-\frac{7}{6}}\varepsilon^2\right)\left(1+\nu^{-\frac{1}{2}}\alpha^{-\frac{1}{4}}\right)\varepsilon\lesssim\varepsilon^2.
    \]Finally, the simple zero mode has the following bound\begin{align*}
\left\|\overline{U}_1\right\|_{L^\infty_tH^s}^2+&\nu\left\|\bm{\nabla}_{y,z}\overline{U}_1\right\|_{L^2_tH^s}^2\lesssim\varepsilon^2+\left(1+\nu^{-\frac{1}{2}}\alpha^{-\frac{1}{4}}\right)^2\varepsilon^2\left(\frac{\nu^{-\frac{4}{3}}+\alpha^{-\frac{1}{6}}\nu^{-\frac{5}{3}}}{1+\nu^{-\frac{1}{2}}\alpha^{-\frac{1}{4}}}\right)\varepsilon\lesssim\left(1+\nu^{-\frac{1}{2}}\alpha^{-\frac{1}{4}}\right)^2\varepsilon^2.
\end{align*}

\appendix
\section{Appendix}

\subsection*{Bound on the cross-term coefficient}

Here, we prove Lemma \ref{upperboundofs}.

\begin{lemma}\label{appendix1}
For $\lambda\geq0$, define $\zeta=\sup_{t,k\neq0,\eta,l}\left|\frac{s}{2}\right|$. Then\[
\zeta=\frac{1}{2\sqrt{2}}+\frac{\lambda}{\sqrt{2}}.
\]
\end{lemma}

\noindent\textit{Proof.} We can write $\frac{s}{2}$ as\[
\frac{s}{2}=\frac{f'}{4}\frac{2\lambda l^2z+2k^2p^2-\alpha f^2l^2}{z^{\frac{3}{2}}f^\frac{1}{2}p^\frac{1}{2}}.
\]We first claim that\[
\left|\frac{f'}{4}\frac{2\lambda l^2z+2k^2p^2-\alpha f^2l^2}{z^{\frac{3}{2}}f^\frac{1}{2}p^\frac{1}{2}}\right|\leq\frac{1+2\lambda}{2\sqrt{2}}.
\]To prove the claim, observe that\begin{align*}
    \left|\frac{f'}{4}\frac{2\lambda l^2z+2k^2p^2-\alpha f^2l^2}{z^{\frac{3}{2}}f^\frac{1}{2}p^\frac{1}{2}}\right|&\leq\frac{|k|f^\frac{1}{2}}{2}\frac{|2\lambda l^2z+2k^2p^2-\alpha f^2l^2|}{\sqrt{2k^2p}zf^\frac{1}{2}p^\frac{1}{2}}=\frac{1}{2\sqrt{2}}\cdot\frac{|2\lambda l^2z+2k^2p^2-\alpha f^2l^2|}{zp}.
\end{align*}
It then remains to show that\[
\frac{|2\lambda l^2z+2k^2p^2-\alpha f^2l^2|}{zp}\leq 1+2\lambda.
\]This is equivalent to showing that\[
\left(2\lambda l^2z+2k^2p^2-\alpha f^2l^2-(1+2\lambda)zp\right)\left(2\lambda l^2z+2k^2p^2-\alpha f^2l^2+(1+2\lambda)zp\right)\leq0.
\]For the second term, notice that\[
zp-\alpha f^2l^2\geq \alpha f^2p-\alpha f^2l^2\geq0.
\]Thus the second term is positive. For the first term, we need to show that\[
2\lambda l^2z+2k^2p^2\leq \alpha f^2l^2+zp+2\lambda zp,
\]which is immediate. It remains to prove that this bound is sharp. Letting $l^2\to\infty$ gives\[
\left|\frac{f'}{4}\frac{2\lambda l^2z+2k^2p^2-\alpha f^2l^2}{z^{\frac{3}{2}}f^\frac{1}{2}p^\frac{1}{2}}\right|\rightarrow\left|\frac{f'}{4}\frac{2\lambda l^2\cdot2k^2l^2+2k^2l^4}{(2k^2l^2)^\frac{3}{2}f^\frac{1}{2}|l|}\right|=\frac{|\eta-kt|}{2}\cdot\frac{4\lambda+2}{2\sqrt{2}f^\frac{1}{2}},
\]sending $|\eta-kt|\rightarrow\infty$, $\zeta$ attains the upper bound.\hfill \qedsymbol

\subsection*{Simpler proof of linear inviscid damping}\label{simpleproof}

Indeed, the linear inviscid damping and enhanced dissipation of the non-zero modes can be derived in a much simpler way, though the method cannot be extended to study the nonlinear problem. Here we provide a simple approach to obtain the decay estimates in the inviscid case. Recall the $(U_3,\Theta)$ system in the Fourier variables\begin{flalign}
\left\{
\begin{aligned}
&\partial_t\widehat{\Theta}=-\alpha\widehat{U}_3,\\
    &\partial_t \widehat{U}_3=-2\frac{kl}{pf}\left(\frac{ik\widehat{\mathcal{K}}-kl\widehat{\Theta}}{\alpha}+l(\eta-kt)\widehat{U}_3\right)+\frac{f}{p}\widehat{\Theta}.
\end{aligned}
\right.\label{eq:B7}
\end{flalign}Here $\widehat{\mathcal{K}}=\alpha \widehat{W}_3-il\widehat{\Theta}$, which is conserved along the Couette flow, i.e. $\widehat{\mathcal{K}}=\widehat{\mathcal{K}}(t=0)$.\begin{theorem}\label{isisi}The following estimates hold\begin{align}
    &||u_{1,\neq}||_{L^2}+||u_{2,\neq}||_{L^2}\lesssim \frac{\left(1+\alpha^{-1/2}\right)}{\langle t\rangle}\left[\left\|\frac{\theta_{\neq}(0)}{\alpha^\frac{1}{2}}\right\|_{H^5}+||u_{\neq}(0)||_{H^5}\right],\label{simple1}\\
   & \left\|\frac{\theta_{\neq}}{\alpha^\frac{1}{2}}\right\|_{L^2}+||u_{3,\neq}||_{L^2}\lesssim \left(1+\alpha^{-1/2}\right)\left[\left\|\frac{\theta_{\neq}(0)}{\alpha^\frac{1}{2}}\right\|_{H^3}+||u_{\neq}(0)||_{H^3}\right].\label{simple2}
\end{align}
\end{theorem}

\noindent\textsc{Proof.} We make the following change of variables\[
\widehat{\Theta}_{\neq}=A,\qquad \widehat{U}_{3,\neq}=b(t)B,
\]where\[
b=\frac{1}{\alpha}\sqrt{\frac{z}{fp}},\qquad z=2k^2p+\alpha f^2.
\]The system reads\begin{align*}
&\partial_tA=-\alpha bB,\\
    &\partial_t B=-2\frac{ik^2l}{\alpha bpf}\widehat{\mathcal{K}}+\left[\frac{2}{\alpha}\frac{k^2}{f}+\frac{f}{p}-\frac{2k^2}{\alpha p}\right]\frac{A}{b}+\left(\frac{f'l^2}{pf}-\frac{b'}{b}\right)B.
\end{align*}
Define the energy $2\mathbb{E}=|A|^2+|B|^2$ (assuming the system is purely real or imaginary), we see that\[
\partial_t\mathbb{E}\leq\frac{k^2}{\alpha bp}\mathbb{E}+\frac{2k^2|l|}{\alpha bpf}\left(\frac{1}{\sqrt{\alpha}}\left|\widehat{\mathcal{K}}\right|^2+\sqrt{\alpha}\mathbb{E}\right)+\frac{1}{2}\left(\frac{3f'}{f}-\frac{p'}{p}-\frac{z'}{z}\right)|B|^2.
\]The first two coefficients satisfy the following bounds\begin{align*}
&\frac{k^2}{\alpha bp}=\frac{k^2\sqrt{f}}{\sqrt{zp}}\leq\frac{k^2}{\sqrt{\alpha}f},\\ &\frac{k^2|l|}{\alpha bpf}=\frac{k^2|l|}{\sqrt{zpf}}\leq\frac{|k|}{\sqrt{\alpha}f}.
\end{align*}Since\[
\frac{1}{2}\left(\frac{3f'}{f}-\frac{p'}{p}-\frac{z'}{z}\right)=\frac{f'(\alpha f^2l^2+2k^2p^2+4k^2pl^2)}{2fpz},
\]we have\[
\frac{1}{2}\left(\frac{3f'}{f}-\frac{p'}{p}-\frac{z'}{z}\right)|B|^2\leq\frac{\bm{1}_{f'\geq 0}}{2}\left(\frac{3f'}{f}-\frac{p'}{p}-\frac{z'}{z}\right)|B|^2\leq\bm{1}_{f'\geq 0}\left(\frac{3f'}{f}-\frac{p'}{p}-\frac{z'}{z}\right)\mathbb{E}.
\]As a result, we see that\begin{align*}
&\partial_t\mathbb{E}\leq\left(\frac{k^2}{f}+\frac{2|k|}{f}\right)\mathbb{E}+\frac{2|k|}{\alpha f}\left|\widehat{\mathcal{K}}\right|^2+\bm{1}_{f'\geq 0}\left(\frac{3f'}{f}-\frac{p'}{p}-\frac{z'}{z}\right)\mathbb{E},\\
\Rightarrow&\mathbb{E}\lesssim e^{\int_0^t\bm{1}_{f'\geq 0}\left(\frac{3f'}{f}-\frac{p'}{p}-\frac{z'}{z}\right)\ ds}\left(\mathbb{E}(0)+\frac{1}{\alpha}\left|\widehat{\mathcal{K}}(0)\right|^2\right).
\end{align*}One can easily deduce that\[
\int_0^t\bm{1}_{f'\geq 0}\left(\frac{3f'}{f}-\frac{p'}{p}-\frac{z'}{z}\right)\ ds\leq\sup_{m_1,m_2\geq0}\log\left(\frac{f^3(m_1)}{p(m_1)z(m_1)}\cdot\frac{p(m_2)z(m_2)}{f^3(m_2)}\right).
\]Since $\frac{f^3}{pz}\lesssim\alpha^{-1}$ and $\frac{pz}{f^3}\lesssim \langle l\rangle^4$, we see that\[
\mathbb{E}\lesssim p^2(0)\left(|A(0)|^2+|B(0)|^2+p(0)\left|\widehat{U}_{\neq}(0)\right|^2+\alpha^{-1}p(0)\left|\widehat{\Theta}_{\neq}(0)\right|^2\right)
\]Since $\alpha^{-1/2}\langle l\rangle^{-1}\lesssim b\lesssim \alpha^{-1/2}$, the estimate \eqref{simple1} can be derived. By inverting $u_{1,\neq}$ and $u_{2,\neq}$ in terms of $K_{\neq}, u_{3,\neq}$ and $\theta_{\neq}$, we can derive \eqref{simple2}.
\hfill \qedsymbol

\subsection*{Operator estimate}
To prove Lemma \ref{suppofl3}, we first prove the following lemma.

\begin{lemma}\label{osci} Let $\lambda\gg 1$ and $p>1$, then for any $q>p-1$, the following estimate holds\begin{equation*}
\int_0^{\frac{\pi}{2}}\frac{|e^{i\lambda \cos\theta}-1|^q}{|\cos\theta|^p}\ d\theta\lesssim\lambda^{p-1}.
\end{equation*}
As a result, if $f\in L^1\left(0,\frac{\pi}{2}\right)\cap L^\infty\left(\frac{\pi}{2}-c,\frac{\pi}{2}\right)$ for some $\frac{\pi}{2}\geq c>0$, we have\[
\int_0^{\frac{\pi}{2}}\frac{|e^{i\lambda \cos\theta}-1|^q}{|\cos\theta|^p}|f|\ d\theta\lesssim\lambda^{p-1} c_{f}
\]where $c_f>0$ is a finite constant depending on $f$.

\end{lemma}

\noindent\textsc{Proof}. We split the integral into two parts:\[
\int_0^{\frac{\pi}{2}-\varepsilon}+\int_{\frac{\pi}{2}-\varepsilon}^\frac{\pi}{2}\frac{|e^{i\lambda \cos\theta}-1|^q}{|\cos\theta|^p}\ d\theta=I+II
\]For $II$, we have\[
II\lesssim\int_{\frac{\pi}{2}-\varepsilon}^\frac{\pi}{2}\lambda^q\frac{1}{|\cos\theta|^{p-q}}\ d\theta=\lambda^q\int_0^\varepsilon\frac{1}{|\sin\theta|^{p-q}}\ d\theta.
\]When $\varepsilon\ll 1$, $\sin\theta\sim\theta$, thus,\[
II\lesssim\lambda^q\varepsilon^{1-p+q}.
\]The integral converges if $p-q<1$. For $I$, we have\[
I\lesssim\frac{1}{\varepsilon^{p-1}}
\]Thus, the bound can be minimized when $\varepsilon=\lambda^{-1}$. The second bound can be proved via similar argument.\hfill \qedsymbol\vspace{9pt}

We therefore obtain the following.\begin{lemma}\label{coro345}
Suppose that $f$ has zero $z$-average. Then\begin{align*}
   \left\|\partial_y^{-1} \partial_z e^{\nu t \Delta}\left(e^{\sqrt{\alpha}t\mathcal{R}_y}-I_d\right)f\right\|_{L^\infty(\mathbb{R}\times\mathbb{T})}&\lesssim  e^{-c\nu t}\nu^{-\frac{1}{2}}\alpha^{\frac{1}{4}}||f||_{L^{2}(\mathbb{R}\times\mathbb{T})},\\
    \left\|\partial_y^{-1} \partial_z e^{\nu t \Delta}\left(e^{\sqrt{\alpha}t\mathcal{R}_y}-I_d\right)f\right\|_{L^2(\mathbb{R}\times\mathbb{T})}&\lesssim  e^{-c\nu t}\nu^{-\frac{1}{2}}\alpha^{\frac{1}{4}}||f||_{L^{1}(\mathbb{R}\times\mathbb{T})}
\end{align*}for any sufficiently small $\delta>0$.\end{lemma}

\noindent\textit{Proof.} Observe that, for any $p\in[2,\infty]$, we have\[
\left|\sum_{l}\int \hat{f}\hat{g}\ d\eta\right|\lesssim\sum_l||\hat{f}||_{L^p_\eta}||\hat{g}||_{L^{p'}_\eta}\lesssim \left\|||\hat{f}||_{L^p_\eta}\right\|_{l^p_l}\left\|||\hat{g}||_{L^{p'}_\eta}\right\|_{l^{p'}_l}\lesssim||f||_{L^{p'}}\left\|||\hat{g}||_{L^{p'}_\eta}\right\|_{l^{p'}_l}.
\]Thus we have\begin{align*}
    \left|\partial_y^{-1} \partial_z e^{\nu t \Delta}\left(e^{\sqrt{\alpha}t\mathcal{R}_y}-I_d\right)f\right|&\lesssim\sum_{l\neq 0}\int_{\mathbb{R}}\frac{|l|}{|\eta|} e^{-\nu t(\eta^2+l^2)}\left|e^{i\sqrt{\alpha}t\frac{\eta}{\sqrt{\eta^2+l^2}}}-1\right||\widehat{f}|\ d\eta\\
    &\lesssim e^{-c\nu t}\left\|\widehat{f}\right\|_{L^2}\left\|\left\|\frac{|l|}{|\eta|} e^{-\nu t(\eta^2+l^2)}\left|e^{i\sqrt{\alpha}t\frac{\eta}{\sqrt{\eta^2+l^2}}}-1\right|\right\|_{L^{2}_\eta}\right\|_{l^{2}_l}\\
    &\lesssim e^{-c\nu t}\left\|f\right\|_{L^{2}}\cdot I.
\end{align*}$I$ is estimated as follows\begin{align*}
    I^{2}&=\sum_{l\neq0}\int_{\mathbb{R}}\frac{l^2}{\eta^2} \left|e^{i\sqrt{\alpha}t\frac{\eta}{\sqrt{\eta^2+l^2}}}-1\right|^{2}e^{-c\nu t(\eta^2+l^2)}\ d\eta\\
    &\lesssim e^{-c\nu t}\sum_{l\neq0}e^{-c\nu l^2 t}\int_{0}^\pi|l|\left|\frac{e^{i\sqrt{\alpha}t\cos\theta}-1}{\cos\theta}\right|^{2}e^{-\frac{c\nu tl^2}{\sin^2\theta}}\ d\theta,
\end{align*}
where we have used the change of variable $\eta=|l|\cot\theta$. By Lemma~\ref{osci}, with $p=2$, we have\begin{align*}
I^2&\lesssim e^{-c\nu t}\sum_{l\neq0}e^{-c\nu l^2t}|l|\alpha^\frac{1}{2}t.
\end{align*}We claim that\[
\sum_{l\geq 1}le^{-\nu l^2t}\lesssim\frac{1}{\nu t}.
\]To see this, notice that the maximum value of $Y(x):=xe^{-x^2\nu t}$ is located at $x^*=(2\nu t)^{-1/2}$. If $x^*\leq1$, we have\[
\sum_{l\geq 1}le^{-l^2 \nu t}\leq e^{-\nu t}+\int_{\mathbb{R}}xe^{-x^2 \nu t}\lesssim\frac{1}{\nu t}.
\]If $x^*>1$, we have\[
\sum_{l\geq 1}le^{-l^2\nu t}\leq 2\sup\left(xe^{-x^2\nu t}\right)+\int_{\mathbb{R}}xe^{-x^2 \nu t}\lesssim\frac{\sqrt{2}}{\sqrt{\nu t}}+\frac{1}{\nu t}.
\]Since $1<\frac{1}{\sqrt{2\nu t}}$, we obtain the desired bound. As a result, we conclude that\begin{align*}
    I^2\lesssim e^{-c\nu t}\nu^{-1}\alpha^\frac{1}{2}.
\end{align*}The bound on the $L^1-L^2$ estimate can be similarly obtained. Due to the Reisz-Thorin theorem, we have\[
\left\|\partial_y^{-1} \partial_z e^{\nu t \Delta}\left(e^{\sqrt{\alpha}t\mathcal{R}_y}-I_d\right)f\right\|_{L^q(\mathbb{R}\times\mathbb{T})}\lesssim  e^{-c\nu t}\nu^{-\frac{1}{2}}\alpha^{\frac{1}{4}}||f||_{L^{p}(\mathbb{R}\times\mathbb{T})},\qquad 1\leq p\leq 2,
\]where $1/q=1/p-1/2.$

\hfill \qedsymbol

\subsection*{Dispersive estimate}
Here we prove the dispersive decay on $\mathbb{R}^2$.\vspace{9pt}

To start with, we consider the following oscillatory integral\[
\int_{\mathbb{R}^2}e^{i\sqrt{\alpha}t\Phi_\pm(\eta,l)}\widehat{\psi}(\eta,l)\ d\eta dl=: e^{\sqrt{\alpha}t\mathcal{L}_\pm} \psi,
 \]where the phase function $\Phi_\pm$ is defined by\[
\Phi_\pm(\eta,l):=\pm\frac{\eta}{\sqrt{\eta^2+l^2}}+\frac{\eta y+lz}{\sqrt{\alpha}t}
\]We consider $\Phi_+$ for simplicity and drop the $+$ subscript.\begin{lemma}\label{APPENDIS1}Suppose that $\psi(y,z)$ has vanishing average, then
  \[  \left\|e^{\sqrt{\alpha}t\mathcal{L}+\nu t\Delta_{y,z}}\psi\right\|_{L^\infty}\lesssim\frac{1}{\langle\alpha^\frac{1}{2}t\rangle^\frac{1}{2}}||\psi||_{W^{3,1}}.\]
\end{lemma}

\noindent\textsc{Proof}. Young's inequality gives\[
\left\|e^{\sqrt{\alpha}t\mathcal{L}+\nu t\Delta_{y,z}}\psi\right\|_{L^\infty}\lesssim\left\|e^{\sqrt{\alpha}t\mathcal{L}}\psi\right\|_{L^\infty}.
\]The boundedness of the oscillatory integral is clear, thus we consider the case where $\sqrt{\alpha}t\gtrsim 1$. We use the following Littlewood-Paley decomposition. Let $r=\sqrt{\eta^2+l^2}$\begin{flalign*}
\phi\in C^\infty(\mathbb{R}^+),\qquad \phi(r)=\left\{
\begin{aligned}
& 1&r\leq 1\\
& 0&r>\frac{4}{3}
\end{aligned}
\right.,\qquad \varphi(r)=\phi\left(\frac{r}{2}\right)-\phi(r),\qquad \varphi_j(r)=\varphi(2^{-j}r),
\end{flalign*}
with supp$(\phi)=\left[0,\frac{4}{3}\right]$. Therefore we can decompose $\psi$ as follows\[
\psi=\sum_{j=-\infty}^\infty\mathbb{P}_{j}\psi,\qquad \mathbb{P}_j\psi=\mathcal{F}^{-1}\left(\varphi_j(r)\cdot \widehat{\psi}\right).
\]Therefore, for $\psi\in\mathcal{S}$, we have the following estimates:\begin{align*}
    \left\|e^{\sqrt{\alpha}t\mathcal{L}} \psi\right\|_{L^\infty}&\leq\sum_{j=-\infty}^\infty\left\|e^{\sqrt{\alpha }t\mathcal{L}}\mathbb{P}_{j}\psi \right\|_{L^\infty},\\
    &\leq\sum_{j=-\infty}^\infty\left\|\mathcal{F}^{-1}\left(e^{i\sqrt{\alpha}t\frac{\eta}{\sqrt{\eta^2+l^2}}}\varphi_{j}\right)*\mathcal{F}^{-1}\left(\varphi^1_j\cdot\widehat{\psi}\right) \right\|_{L^\infty},\\
    &\leq\sum_{j=-\infty}^\infty \underbrace{\left\|\mathcal{F}^{-1}\left(e^{i\sqrt{\alpha}t\frac{\eta}{\sqrt{\eta^2+l^2}}}\varphi_{j}\right)\right\|_{L^\infty}}_{I_1}\underbrace{\left\|\mathcal{F}^{-1}\left(\varphi^1_j\cdot\widehat{\psi}\right)\right\|_{L^1}}_{I_2},
\end{align*}
where we have used the convolution inequality. Here $\varphi^1$ has similar support property such that $\varphi=\varphi\varphi^1$.

Let us first consider $I_1$. It suffices to bound the integral\[
\mathcal{S}_1:=\int_{\mathbb{R}^2} e^{i\sqrt{\alpha}t\Phi}\varphi\left(2^{-j}\sqrt{\eta^2+l^2}\right)\ d\eta dl,\qquad \Phi=\frac{\eta}{\sqrt{\eta^2+l^2}}+\frac{\eta y+lz}{\sqrt{\alpha}t}
\]We use polar coordinates to deal with this integral. Let $\eta=r\cos\theta$, $l=r\sin\theta$, so that\begin{equation}
\mathcal{S}_1=\int_0^\infty\int_0^{2\pi} r\varphi(2^{-j}r)e^{i\sqrt{\alpha}t\Phi(r,\theta)}\ d\theta dr,\qquad \Phi=\cos\theta+\frac{\rho r}{\sqrt{\alpha}t}\cos(\theta-\theta_0),\label{oscinotation}
\end{equation}
where\begin{equation}
\rho=\sqrt{y^2+z^2},\qquad \cos\theta_0=\frac{y}{\rho},\qquad \sin\theta_0=\frac{z}{\rho}.\label{oscnotation2}
\end{equation}
Let us define \[
\int_0^\infty r\varphi_j(r)e^{i\rho r\cos(\theta-\theta_0)}\ dr=\int_{\mathbb{R}^+}r\varphi_j(r)e^{irz}\ dr=:B(z),\qquad z=\rho\cos(\theta-\theta_0).
\]We claim that \begin{equation}
||B(\cdot)||_{W^{1,1}(\mathbb{R})}\lesssim 2^{2j}+2^j.\label{eq:B10}
\end{equation}
We first prove this claim. To estimate $||B||_{L^1}$, split the integral into two parts:\[
\int_{\mathbb{R}}|B(z)|\ dz=\int_{-\delta}^\delta+\int_{\mathbb{R}\backslash[-\delta,\delta]}|B(z)|\ dz=: I_1+I_2.
\]For $I_1$, we have the trivial bound\[
I_1\leq 2\delta\int_{\mathbb{R}^+} r\varphi_j(r)\ dr\lesssim 2^{2j}\delta.
\]For $I_2$, we integrate twice to see that\[
|B(z)|=\frac{1}{z^2}\cdot\left|\int_{\mathbb{R}^+}\frac{d^2}{dr^2}(r\varphi_j(r))\cdot e^{irz}\ dr\right|\lesssim\frac{1}{z^2}.
\]We obtain \[
I_2\lesssim\int_{\delta}^\infty\frac{1}{z^2}\ dz\approx\frac{1}{\delta}.
\]Letting $\delta\approx 2^{-j}$, we then have \[
||B||_{L^1}\lesssim 2^j.
\]The bound for $\frac{dB}{dz}$ can be similarly computed, we obtain\[
\left|\frac{dB}{dz}\right|=\left|\int_{0}^\infty r^2\varphi_j(r)e^{irz}\ dr\right|.
\]We again split into $I_1$ and $I_2$ with\[
I_1\leq 2\delta\int_{\mathbb{R}^+}r^2\varphi_j(r)\ dr\lesssim 2^{3j}\delta.
\]For $I_2$,\[
\left|\frac{dB}{dz}\right|=\frac{1}{z^2}\cdot\left|\int_{\mathbb{R}^+}\frac{d^2}{dr^2}(r^2\varphi_j(r))\cdot e^{irz}\ dr\right|\lesssim \frac{2^{j}}{z^2}.
\]Thus $\left\|\frac{dB}{dz}\right\|_{L^1}\lesssim 2^{2j}$ by choosing $\delta\approx 2^{-j}$. Combining these estimates, we can conclude that \eqref{eq:B10} holds. In addition, we can similarly prove that \[
||B(\cdot)||_{L^\infty}\lesssim 2^{2j}+2^j.
\]It remains to estimate the oscillatory integral\begin{equation}
\int_{0}^{2\pi} e^{i\sqrt{\alpha}t\cos\theta}B(\rho\cos(\theta-\theta_0))\ d\theta.\label{osci10}
\end{equation}
The phase function has critical points at $\theta=0,\pi,2\pi$. Let\[
\int_{\left[0,\delta\right]\cup\left[\pi-\delta,\pi+\delta\right] \cup[2\pi-\delta,2\pi]}+\int_{\text{complement}}e^{i\sqrt{\alpha}t\cos\theta}B(\rho\cos(\theta-\theta_0))\ d\theta=: h_1+h_2.
\]$h_1$ has the trivial bound\[
|h_1|\lesssim\delta||B(\rho\cos(\theta-\theta_0))||_{L^\infty_\theta}\lesssim\delta ||B(\cdot)||_{L^\infty_z(\mathbb{R})}\lesssim(2^{2j}+2^j)\delta.
\]For $h_2$, notice that in the domain of integration, $|\sin\theta|\gtrsim\delta$, so that by integrating by parts, we have\[
|h_2|\lesssim \frac{1}{\delta\sqrt{\alpha}t}\left(||B||_{L^\infty}+||\partial_\theta A||_{L^1_\theta}\right)\lesssim\frac{2^{2j}+2^j}{\delta\sqrt{\alpha}t}
\]Here $A(\theta)=B(\rho\cos(\theta-\theta_0))$ so that $||A'||_{L^1_\theta}\lesssim||B'||_{L^1_z}$. Choosing $\delta^2=\frac{1}{\sqrt{\alpha}t}$, we therefore have\[
|\mathcal{S}_1|\lesssim\frac{2^{2j}+2^j}{\alpha^\frac{1}{4}t^\frac{1}{2}}.
\]Summing over $j$, we have\[
  \left\|e^{\sqrt{\alpha}t\mathcal{L}} \psi\right\|_{L^\infty}\leq\sum_{j=-\infty}^\infty \frac{2^{2j}+2^j}{1+\alpha^\frac{1}{4}t^\frac{1}{2}}\left\|\mathcal{F}^{-1}\left(\varphi^1_j\cdot\widehat{\psi}\right)\right\|_{L^1}\lesssim\frac{1}{\langle\alpha^\frac{1}{2}t\rangle^\frac{1}{2}}||\psi||_{W^{3,1}}.
\]\hfill \qedsymbol

With some minor adaptions, the estimate  \eqref{eq:D1} can be simultaneously proved.
\section*{\textbf{Acknowledgment}}
The author would like to thank his supervisor, Professor Michele Coti Zelati, for his invaluable help.

\section*{Statements and Declarations}

\noindent\textbf{Competing interests.}
The author affirms that he has no conflict of interest regarding this work.\vspace{5pt}

\noindent\textbf{AI Declaration.}
This work is the author's master's thesis and the author declares that all mathematical ideas and original contributions were developed independently. ChatGPT (GPT-5.6) was only used for language editing and checking the correctness of mathematical proofs. The author takes full responsibility for the content and correctness of the manuscript.\vspace{5pt}

\noindent\textbf{Data and code availability.}
The Python code used to generate the illustrative numerical figures is available from the corresponding author upon reasonable request. No separate datasets were generated.\vspace{5pt}

\noindent\textbf{Funding.}
No funding was received for conducting this study.

\printbibliography
\end{document}